\documentclass[final,reqno]{siamart190516}

\usepackage{amsfonts}
\usepackage{graphicx,epstopdf}
\usepackage{mathrsfs}
\usepackage{mathtools}
\usepackage{bm}
\usepackage{multirow}
\usepackage{color}
\usepackage{hyperref} 
\usepackage{xcolor}

\usepackage{caption, subcaption}%
\usepackage{array}
\usepackage{booktabs}

\usepackage{xparse} 
\usepackage{stmaryrd} 
\usepackage{rotfloat}

\usepackage{setspace}
\usepackage{enumerate}
\usepackage{amssymb}

\usepackage{accents}
\newcommand{\dbtilde}[1]{\accentset{\approx}{#1}}

\newsiamthm{claim}{Claim}
\newsiamremark{remark}{Remark}
\newsiamremark{expl}{Example}
\newsiamthm{problem}{Problem}
\newsiamthm{algorithm1}{Algorithm}
\newsiamthm{conjecture}{Conjecture}
\numberwithin{theorem}{section}

\counterwithout{equation}{section}

\allowdisplaybreaks

\def\jump#1{\llbracket #1 \rrbracket }
\def\average#1{\lbrace\!\lbrace #1  \rbrace\!\rbrace }

\makeatletter
\def\widebreve{\mathpalette\wide@breve}
\def\wide@breve#1#2{\sbox\z@{$#1#2$}%
	\mathop{\vbox{\m@th\ialign{##\crcr
				\kern0.08em\brevefill#1{0.8\wd\z@}\crcr\noalign{\nointerlineskip}%
				$\hss#1#2\hss$\crcr}}}\limits}
\def\brevefill#1#2{$\m@th\sbox\tw@{$#1($}%
	\hss\resizebox{#2}{\wd\tw@}{\rotatebox[origin=c]{90}{\upshape(}}\hss$}
\makeatletter

\newcommand{\TheTitle}{A new discretely divergence-free positivity-preserving high-order finite volume method for ideal MHD equations} 

\newcommand{\dt}{\Delta t}
\newcommand{\dx}{\Delta x}
\newcommand{\dy}{\Delta y}

\newcommand{\dd}{\mathrm{d}}

\newcommand{\halfone}{\frac{1}{2}}
\newcommand{\alphaLF}{\alpha^{\rm{LF}}}

\title{{\TheTitle}\thanks{This work is partially supported by National Natural Science Foundation of China (Grant No.~12171227) and Shenzhen Science and Technology Program (Grant No. RCJC20221008092757098).}} 

\author{
	Shengrong Ding
	\thanks{Department of Mathematics and SUSTech International Center for Mathematics, Southern University of Science and Technology, and National Center for Applied Mathematics Shenzhen (NCAMS), Shenzhen 518055, China.}
	\and
	Kailiang Wu \footnotemark[2]
	\thanks{Corresponding author. e-mail: {\tt wukl@sustech.edu.cn}.}
}

\begin{document}
	
	\maketitle
	
	
	\begin{abstract}
This paper proposes and analyzes a novel efficient high-order finite volume method for the ideal magnetohydrodynamics (MHD). As a distinctive feature, the method simultaneously preserves two critical physical constraints: a discretely divergence-free (DDF) constraint on the magnetic field and the positivity-preserving (PP) property, which ensures the positivity of density, pressure, and internal energy. To enforce the DDF condition, we design a new discrete projection approach that projects the reconstructed point values at the cell interface into a DDF space, without using any approximation polynomials. This projection method is highly efficient, easy to implement, and particularly suitable for standard high-order finite volume WENO methods, which typically return only the point values in the reconstruction. Moreover, we also develop a new finite volume framework for constructing provably PP schemes for the ideal MHD system. The framework comprises the discrete projection technique, a suitable approximation to the Godunov--Powell source terms, and a simple PP limiter. We provide rigorous analysis of the PP property of the proposed finite volume method, demonstrating that the DDF condition and the proper approximation to the source terms eliminate the impact of magnetic divergence terms on the PP property. The analysis is challenging due to the internal energy function's nonlinearity and the intricate relationship between the DDF and PP properties. To address these challenges, the recently developed geometric quasilinearization approach is adopted, which transforms a nonlinear constraint into a family of linear constraints.  Finally, we validate the effectiveness of the proposed method through several benchmark and demanding numerical examples. The results demonstrate that the proposed method is robust, accurate, and highly effective, confirming the significance of the proposed DDF projection and PP techniques. 
	\end{abstract}
	
	

\section{Introduction}	

In this paper, we focus on the development of efficient and robust high-order accurate numerical schemes for 
the 
 ideal compressible MHD equations, which can be written into a system of hyperbolic conservation laws:  
\begin{equation}\label{eq:MHD}
	\mathbf{U}_t + \nabla \cdot \mathbf{F}(\mathbf{U})={\bf 0},
\end{equation} 
where $t$ denotes the time. The conservative vector $ \mathbf{U} = \left(\rho, \mathbf{m}, \mathbf{B}, E \right)^{\top}$ takes values in $\mathbb{R}^8$, with $\rho$ denoting the density, $\mathbf{m}=(m_1,m_2,m_3)$ being momentum vector, $\mathbf{B}=(B_1,B_2,B_3)$ standing for the magnetic field, and $E$ representing the total energy. In the $d$-dimensional space, we have 
$
	\nabla \cdot \mathbf{F} = \sum_{i=1}^{d} \frac{\partial \mathbf{F}_i}{\partial x_i}
$ 
with $(x_1,\cdots,x_d)$ denoting the spatial coordinate variables, 
and the flux $\mathbf{F}=(\mathbf{F}_1, \cdots, \mathbf{F}_d)$  is defined by
$$
\mathbf{F}_i(\mathbf{U})=\Big( m_i, m_i \mathbf{v} - B_i \mathbf{B} + p_{tot} \mathbf{e}_i, v_i \mathbf{B} - B_i \mathbf{v}, v_i(E+p_{tot}) - B_i (\mathbf{v} \cdot \mathbf{B}) \Big)^{\top},
$$
where  $\mathbf{v}=(v_1,v_2,v_3)=\mathbf{m}/\rho$ denotes the fluid velocity, $p_{tot}=p+\frac12 {| \mathbf{B} |}^2$ is the total pressure including the thermal pressure $p$ and the magnetic pressure, and the vector $\mathbf{e}_i$ is the $i$th row of the identity matrix of size $3\times 3$. The total energy $E=\rho e + \frac{1}{2}\rho |\mathbf{v}|^2 + \frac12 {| \mathbf{B} |}^2$ consists of internal, kinetic, and magnetic energies, where $e$ is the specific internal energy. 
In addition, an equation of state (EOS) is required to close the above MHD equations. 
A general EOS can be formulated as $e=e(\rho,p)$ and typically satisfies $e>0 \Leftrightarrow p>0$ when $\rho >0$ (cf.~\cite{Wu2017a,WuShu2018,WuShu2019}). 
This condition is assumed in this paper, and it satisfied by the widely-used ideal EOS $e=p/(\rho(\gamma -1))$, where 
$\gamma \in (1,+\infty)$ is a constant and denotes the adiabatic index.

The magnetic field $\mathbf{B}$ should satisfy a divergence-free (DF) constraint  
\begin{equation}\label{eq:DF}
	\nabla \cdot \mathbf{B} := \sum_{i=1}^{d} \frac{\partial B_i}{\partial x_i} = 0.
\end{equation}
 In numerical simulations, preserving this DF constraint is crucial, as violation of it can cause numerical instability and generate nonphysical structures in the approximated solutions (cf.~\cite{Brackbill1980,Evans1988,Toth2000,Li2005}). 
While in the one-dimensional case, the DF constraint can be easily preserved by ensuring $B_1$ is constant, in the multidimensional cases, exact preservation of the DF constraint in numerical simulations is very challenging. 
Over the past few decades, various numerical techniques have emerged to handle the DF constraint \eqref{eq:DF}. 
One of the earliest methods is the projection approach \cite{Brackbill1980}, which is a post-processing divergence correction procedure utilizing Hodge decomposition to project the non-divergence-free magnetic field onto a DF subspace through solving an elliptic Poisson equation. Another widely used technique is the constrained transport (CT) method, initially proposed in \cite{Evans1988} for simulating MHD flows. 
The CT method maintains a discrete version of the DF condition on staggered grids, and numerous variants have been developed by researchers within different frameworks, as documented in \cite{devore1991flux,Dai1998,Londrillo2004divergence,Gardiner2005,xu2016divergence}. 
Unstaggered CT methods have also been developed (see, e.g.,
\cite{Rossmanith2006,helzel2011unstaggered,mishra_tadmor_2012,helzel2013high,Christlieb2014}), typically based on numerically evolving the magnetic potential and computing the DF magnetic field via the (discrete) curl of the magnetic potential. 
Furthermore, locally DF discontinuous Galerkin methods have been developed to enforce the zero divergence of the magnetic field within each cell in \cite{Li2005}. More recently, globally DF high-order methods have been proposed to maintain the exact zero divergence of the magnetic field within the finite-volume or (central) discontinuous Galerkin framework, as detailed in \cite{Balsara2009,Li2011,Li2012,Fu2018}. 
Another class of methods for mitigating divergence errors are the eight-wave schemes, which were introduced by Powell \cite{Powell1995,powell1997approximate} 
 as a discretization of Godunov's modified MHD system  
\begin{equation}\label{eq:ModMHD} 
	\mathbf{U}_t + \nabla \cdot \mathbf{F}(\mathbf{U})=-(\nabla \cdot \mathbf{B}) \mathbf{S}(\mathbf{U}) ,
\end{equation}
where $\mathbf{S}(\mathbf{U})=\left(0, \mathbf{B}, \mathbf{v}, \mathbf{v}\cdot \mathbf{B}\right)^\top$. 
This modified model includes nonconservative source terms, referred to as Godunov--Powell source terms, that are proportional to the divergence of the magnetic field. By including these source terms, the modified model becomes Galilean invariant \cite{Dellar2001}, symmetrizable \cite{Godunov1972}, and useful in constructing entropy stable schemes (see, e.g., \cite{Chandrashekar2016}). 
The eight-wave schemes do not enforce exact divergence-free solutions but instead aim to advect the magnetic divergence errors with the flow. This makes them effective at controlling these errors, although the nonconservative nature of the Godunov--Powell modified MHD equations can introduce certain drawbacks \cite{Toth2000}. 
 Another approach for dealing with divergence errors is the hyperbolic divergence-cleaning method \cite{Dedner2002}, which uses a mixed hyperbolic-parabolic equation to damp out the divergence errors. 

Another challenge in the numerical simulations of MHD is to preserve 
the positivity of density, pressure, and internal energy. In other words,  
the conservative variables $\bf U$ should be preserved in the admissible state set 
\begin{equation}\label{eq:G}
	G = \left\{
	\mathbf{U}=(\rho, \mathbf{m}, \mathbf{B}, E)^\top:~ \rho>0, ~  \mathcal{E}(\mathbf{U}):=E-\frac{1}{2}\left(\frac{|\mathbf{m}|^2}{\rho} + |\mathbf{B}|^2  \right) >0 \right\},
\end{equation} 
where $\mathcal{E}(\mathbf{U})=\rho e$ represents the internal energy. (Notice that the function $\mathcal{E}(\mathbf{U})$ is concave with respect to $\bf U$ when $\rho >0$, implying that $G$ is a convex set.)  
Positivity preservation is a critical requirement for developing reliable numerical methods for MHD. Many existing schemes for MHD often produce negative values for density or pressure, particularly in the presence of strong discontinuities, high Mach numbers, low internal energy, low plasma beta, low density, and/or strong magnetic fields. Negative density or pressure can cause the loss of hyperbolicity of the system, leading to severe numerical instability and simulation breakdown. Although this issue also occurs in pure hydrodynamic cases, it is more severe in MHD due to the influence of magnetic divergence errors on positivity \cite{Wu2017a}. 
In the last two decades, researchers have devoted some efforts to mitigating this issue; see, e.g., \cite{BalsaraSpicer1999a,Waagan2009,Balsara2012,cheng,Christlieb,Christlieb2016,liu2021new}  
and recently developed provably positivity-preserving (PP) schemes \cite{Wu2017a,WuShu2018,WuShu2019,WuShu2020NumMath,WuJiangShu2022}. 
The early efforts in this direction were devoted to designing robust multi-state approximate Riemann solvers \cite{Janhunen2000,Bouchut2007,Bouchut2010,klingenberg2010relaxation} for the 1D ideal MHD equations. Waagan proposed  
a second-order PP MUSCL-Hancock scheme \cite{Waagan2009} based on relaxation Riemann solvers of \cite{Bouchut2007,Bouchut2010}, whose robustness was further demonstrated through benchmark numerical tests in \cite{waagan2011robust}. 
In recent years, remarkable progresses have been made in seeking high-order PP or bound-preserving methods for conservation laws \cite{zhang2010,zhang2010b,zhang2011b,Balsara2012,cheng,Xu2014,xiong2016parametrized,ZHANG2017301,WuTangM3AS,VegtXu2019,Kuzmin2022Bound}. 
Christlieb et al.~\cite{Christlieb,Christlieb2016} developed high-order PP finite difference schemes for ideal MHD, which utilize parametrized flux-correction limiters \cite{Xu2014,xiong2016parametrized} and rely on the presumed PP property of the Lax--Friedrichs scheme (later proven rigorously in \cite{Wu2017a}). 
In high-order finite volume or discontinuous Galerkin (DG) methods, the PP property may be lost in two situations  \cite{zhang2010,zhang2010b}: when the reconstructed or DG solution polynomials fail to be positive, or when the cell averages become negative during the updating process. While PP limiters \cite{Balsara2012,cheng} can be used to recover positivity in the first case, ensuring the positivity of the cell averages in the updating process is more challenging and critical to obtain a genuinely PP scheme. 
The validity of the PP limiters \cite{Balsara2012,cheng} relies on the positivity of the cell averages in the updating process. However, this was not rigorously proven for the methods in \cite{Balsara2012,cheng}. Only the 1D methods in \cite{cheng} were formally shown to be PP by invoking some assumptions on the exact Riemann solutions, while the multi-dimensional methods in \cite{cheng} were only conjectured to be PP. 
Given that finite numerical tests alone may not be sufficient to demonstrate the PP property thoroughly and genuinely in all cases, the development of {\em provably PP} schemes  for MHD is essential. 

In recent studies \cite{Wu2017a,WuShu2018,WuShu2019,WuJiangShu2022}, important progress has been made in  the development of high-order {\em provably PP} schemes for the ideal MHD system. 
A significant finding in these studies is the close relationship between the PP property, which is an algebraic feature, and the DF constraint, which is a differential restriction. This connection has been established both at the continuous and discrete levels \cite{Wu2017a,WuShu2018,WuJiangShu2022}. 
In particular, the PP property of finite volume and DG schemes for the standard MHD system is closely tied to a discrete DF condition, as proven in \cite{Wu2017a}. 
Furthermore, violating this discrete DF condition even slightly may cause cell averages to lose their PP property in the updating process \cite[Theorem 4.1 and Remark 4.4]{Wu2017a}. 
In \cite[Appendix A]{WuShu2018}, it was shown that even the exact smooth solutions of the standard MHD system \eqref{eq:MHD} may not be PP if the continuous DF condition \eqref{eq:DF} is slightly violated. 
 However, this issue does not occur in the modified MHD system \eqref{eq:ModMHD}, as its exact smooth solutions are always PP, regardless of whether the DF constraint is satisfied or not \cite[Proposition 1]{WuShu2019}. 
 Exploiting these insights, researchers have systematically developed high-order provably PP schemes for the modified MHD system \cite{WuShu2018,WuShu2019,WuJiangShu2022}.  See also \cite{WuShu2020NumMath,liu2021new,zhang2021provably} for a few applications or extensions.

The purpose of this paper is to design and analyze a new high-order robust finite volume method for solving the ideal MHD system in multiple dimensions. This method maintains both a discrete DF constraint and the desirable PP property, making it unique compared to existing methods. The contributions and novelty of this work are summarized as follows:
\begin{itemize}
	\item We propose a novel discrete projection technique that exactly enforces a discrete version of the DF condition.  
	As a notable feature, this projection method modifies the reconstructed point values at the cell interface and projects them into a discretely DF (DDF) space.   Such a ``discrete'' projection approach is significantly different from the existing ``continuous'' DF methods, which typically enforces the piecewise polynomial solutions in a DF function space (see, for example, \cite{Balsara2004,Li2005,Balsara2009,Li2011,xu2016divergence,Fu2018,liu2021new}).  Thanks to this
	feature, our DDF projection technique is particularly suitable for standard high-order finite volume WENO methods, in which the reconstruction procedure typically returns only the point values instead of approximation polynomials. This makes the proposed DDF projection  very efficient and easy to implement. We provide insights into the projection technique and prove that the DDF projection procedure preserves the high-order accuracy of the reconstruction.
	\item We design a new finite volume framework for constructing provably PP schemes for the ideal MHD system. The key ingredients of the framework include the DDF projection, a suitable approximation to the Godunov--Powell source terms, and 
	a simple PP limiter. Due to the intrinsic relation between the PP and DDF properties, 
	the cell averages may lose their PP property in the updating process if the DDF projection is not used or if the Godunov--Powell source terms are not discretized appropriately. Yet, seeking a qualified source approximation is highly nontrivial. Based on our rigorous PP analysis, we devise a novel source approximation, which differs from those in  \cite{WuShu2018,WuShu2019,WuJiangShu2022} and leads to a milder PP CFL condition. We also present a PP limiter to enforce the PP property of the reconstructed point values without using any approximation polynomials, as a generalization of the simple maximum-principle-preserving limiter \cite{zhang2011b} for scalar conservation law. Furthermore, the PP limiter does not destroy the DDF condition enforced by the projection.    
	 \item  We provide rigorous analysis of the PP property of our finite volume method. 
	 The analysis is challenging due to the nonlinearity of  $\mathcal{E}(\mathbf{U})$ and the relation of DDF condition to the PP property.  
	 We overcome these challenges by using the recently developed geometric quasilinearization (GQL) approach \cite{Wu2017a,WuShu2021GQL}.  The key idea of GQL is to equivalently transfer the nonlinear constraint 
	 $\mathcal{E}(\mathbf{U})>0$ in \eqref{eq:G} into a family of linear constraints by introducing extra auxiliary variables. Based on GQL and technical estimates, our numerical analysis shows that the enforced DDF condition and our suitable approximation to the Godunov--Powell source terms  eliminate the impact of magnetic divergence terms on the PP property. 
	 \item We implement the new second-order and fifth-order DDFPP finite volume schemes and demonstrate their robustness through several benchmark and demanding numerical experiments. Our numerical results also validate the importance of the proposed DDF projection and PP techniques. 
\end{itemize}

\section{New DDFPP finite volume method}\label{Numericalmethod}

This section presents the framework and computational details of our new finite volume method, 
which preserves the desired positivity and a DDF constraint, 
 for the 2D ideal MHD system. For better readability, the rigorous proof of the PP property, which is highly technical and nontrivial, will be given in \cref{PCP:Sec}.   
It is worth noting that our numerical method and analysis are directly extensible to the 3D case. 


Let the symbols $(x,y)$ represent the 2D spatial coordinates for convenience. 
Assume that the computational domain is partitioned into a uniform rectangular mesh with cells 
$\{ I_{ij}=(x_{i-\halfone},x_{i+\halfone})\times(y_{j-\halfone}, y_{j+\halfone}) \}$. 
Let $\dx$ and $\dy$ denote the constant spatial step-sizes in the $x$ and $y$ directions, respectively. 
 The center of the cell $I_{ij}$ is denoted as $(x_i,y_j)$ with $x_i = \frac{1}{2}(x_{i-\halfone}+x_{i+\halfone})$ and $y_j = \frac{1}{2}(y_{j-\halfone}+y_{j+\halfone})$. 
Integrating the 2D MHD equations \eqref{eq:ModMHD} over $I_{ij}$ gives 
\begin{equation}\label{eq:IntForm}
	\begin{aligned}
		&\frac{\mathrm{d}}{\mathrm{d}t} \int_{I_{ij}} \mathbf{U} \mathrm{d}x\mathrm{d}y=  -\int_{y_{j-\halfone}}^{y_{j+\halfone}} (\mathbf{F}_1(\mathbf{U}(x_{i+\halfone},y,t))-\mathbf{F}_1(\mathbf{U}(x_{i-\halfone},y,t))) \mathrm{d}y  \\ 
		&-\int_{x_{i-\halfone}}^{x_{i+\halfone}}  (\mathbf{F}_2(\mathbf{U}(x,y_{j+\halfone},t))-\mathbf{F}_2(\mathbf{U}(x,y_{j-\halfone},t))) \mathrm{d}x 
		-\int_{I_{ij}} (\nabla \cdot \mathbf{B}) \mathbf{S}(\mathbf{U}) \mathrm{d}x\mathrm{d}y.
	\end{aligned}
\end{equation}

\subsection{Outline and key ingredients of the DDFPP method} 
In order to obtain a high-order accurate scheme, the flux integrals in \eqref{eq:IntForm} on the cell edges   should be approximated by using a quadrature rule of sufficient accuracy. 
For example, to achieve the $k$th order accuracy,  one can employ the $Q$-point Gauss quadrature with $Q=\lceil \frac{k}{2} \rceil$ or the $Q$-point Gauss--Lobatto quadrature with $Q=\lceil \frac{k+2}{2} \rceil$. 
Let $\mathbb{Q}_i^{x}:=\{x_i^{\mu}\}_{\mu=1}^{Q}$ and $\mathbb{Q}_j^{y}:=\{y_j^{\mu}\}_{\mu=1}^{Q}$ denote the quadrature nodes in the intervals $[x_{i-\halfone},x_{i+\halfone}]$ and $[y_{j-\halfone},y_{j+\halfone}]$, respectively. Let $\{ \omega_{\mu} \}_{\mu=1}^{Q}$ be the corresponding quadrature weights, which are normalized such that $\sum_{\mu=1}^{Q} \omega_{\mu} =1$.

Let $\overline{\mathbf{U}}_{ij}$ be the numerical approximation to 
$\frac{1}{\Delta x \Delta y} \int_{I_{ij}} \mathbf{U} \mathrm{d}x\mathrm{d}y$. 
From \eqref{eq:IntForm}, we can derive a semi-discrete high-order finite volume method for 2D MHD system \eqref{eq:ModMHD} as 
\begin{equation}\label{FV:DiscreteForm}
	\begin{aligned}
		\frac{\mathrm{d} \overline{\mathbf{U}}_{ij}(t)}{\mathrm{d}t} =  
		-\frac{1}{\dx} \left( \widehat{\mathbf{F}}_{1,i+\halfone,j}-\widehat{\mathbf{F}}_{1,i-\halfone,j} \right)  
		-\frac{1}{\dy} \left( \widehat{\mathbf{F}}_{2,i,j+\halfone}-\widehat{\mathbf{F}}_{2,i,j-\halfone} \right) 
		+{\mathbf{S}}_{ij}, 
	\end{aligned}
\end{equation}
where the numerical fluxes are defined by 
\begin{equation}\label{FV:Flux}
	\resizebox{0.92\hsize}{!}{$
		\begin{aligned}
			\widehat{\mathbf{F}}_{1,i+\halfone,j} := \sum_{\mu=1}^{Q} \omega_{\mu} 
			\widehat{\mathbf{F}}_1\left(\mathbf{U}_{i+\halfone,j}^{-,\mu}, \mathbf{U}_{i+\halfone,j}^{+,\mu} \right), \quad
			\widehat{\mathbf{F}}_{2,i,j+\halfone} := \sum_{\mu=1}^{Q} \omega_{\mu} 
			\widehat{\mathbf{F}}_2\left(\mathbf{U}_{i,j+\halfone}^{\mu,-}, \mathbf{U}_{i,j+\halfone}^{\mu,+} \right)
		\end{aligned}
		$}
\end{equation}
as approximations to 
$
\frac{1}{\dy} \int_{y_{j-\halfone}}^{y_{j+\halfone}} \mathbf{F}_1 (\mathbf{U}(x_{i+\halfone},y,t)) \mathrm{d}y
$ and 
$
\frac{1}{\dx} \int_{x_{i-\halfone}}^{x_{i+\halfone}} \mathbf{F}_2(\mathbf{U}(x,y_{j+\halfone},t)) \mathrm{d}x 
$, respectively, and ${\mathbf{S}}_{ij}$ is a suitable approximation to $\frac{1}{\dx \dy}\int_{I_{ij}} (-\nabla \cdot \mathbf{B}) \mathbf{S}(\mathbf{U}) \mathrm{d}x\mathrm{d}y$ which will be discussed later. 
The notations $\widehat{\mathbf{F}}_1(\cdot,\cdot)$ and $\widehat{\mathbf{F}}_2(\cdot,\cdot)$ 
denote the numerical flux functions. In this paper, we adopt the Lax--Friedrichs flux
\begin{equation}\label{eq:LFFlux}
	\widehat{\mathbf{F}}_\ell(\mathbf{U}^-, \mathbf{U}^+) = \frac{1}{2} \left(\mathbf{F}_\ell(\mathbf{U}^-) + \mathbf{F}_\ell(\mathbf{U}^+) - \alphaLF_\ell (\mathbf{U}^+-\mathbf{U}^-)\right),  \qquad  \ell=1,2
\end{equation}
with $\alphaLF_\ell$, $\ell=1,2$, denoting the numerical viscosity parameters. Some other numerical fluxes such as the HLL flux can also be used in our framework. 
The quantities $\mathbf{U}_{i+\halfone,j}^{\pm,\mu}$ and $\mathbf{U}_{i,j+\halfone}^{\mu,\pm}$ in \eqref{FV:Flux} denote 
the high-order accurate approximations to the point values of the solution at $(x_{i+\halfone}^{\pm},y_j^{\mu})$ and $(x_{i}^{\mu},y_{j+\halfone}^{\pm})$, respectively, 
and they are typically reconstructed from the cell averages $\{ \overline{\mathbf{U}}_{ij} \}$.

	Based on our numerical analysis (see \cref{PCP:Sec}), we find that  
	in order to achieve the PP property, 
	the values $\mathbf{U}_{i+\halfone,j}^{\pm,\mu}$ and $\mathbf{U}_{i,j+\halfone}^{\mu,\pm}$ should be reconstructed very carefully such that the following conditions \eqref{eq:DDF}--\eqref{LF:cond1} are satisfied simultaneously: 
	\begin{equation}\label{eq:DDF}
			\nabla_h \cdot \mathbf{B}_{ij} := \sum_{\mu=1}^{Q} \omega_{\mu}
			\left( \frac{ (B_{1})_{i+\halfone,j}^{-,\mu} -   (B_{1})_{i-\halfone,j}^{+,\mu}}{\dx}
			+ \frac{ (B_{2})_{i,j+\halfone}^{\mu,-}  -  (B_{2})_{i,j-\halfone}^{\mu,+} }{\dy} \right)=0,
	\end{equation}
    and 
    \begin{equation}\label{LF:cond1}
    	\resizebox{0.92\hsize}{!}{$
    		\begin{cases}
    			\mathbf{U}_{i+\halfone,j}^{\pm,\mu}\in G, \quad \mathbf{U}_{i,j+\halfone}^{\mu,\pm} \in G \qquad \forall i,j,\mu, \\[2mm]
    			\displaystyle
    			{\bf \Pi}_{ij}=\frac{\overline{\mathbf{U}}_{ij} - \frac{\widehat{\omega}_1}{\lambda}  \left[ \lambda_1 \left( {\bf \Pi}_{i-\halfone,j}^{+} + {\bf \Pi}_{i+\halfone,j}^{-} \right) + \lambda_2 \left( {\bf \Pi}_{i,j-\halfone}^{+} + {\bf \Pi}_{i,j+\halfone}^{-} \right) \right]}{1-2\widehat{\omega}_1}  \in G, \mbox{ if } k \ge 3. 
    		\end{cases}
    		$}
    \end{equation}
	where 
	$\widehat{\omega}_1 = \frac{1}{ L(L-1) }$ is the first weight of the $L$--point Guass--Lobatto quadrature with $L=\lceil \frac{k+2}{2} \rceil$ for a $k$th order finite volume scheme, and  
	\begin{align} \nonumber
		& \lambda_1 = \alphaLF_1 \Delta t/\Delta x, \quad \lambda_2 = \alphaLF_2 \Delta t/\Delta y, \quad   
		\lambda  = \lambda_1 + \lambda_2, 
		\\
		\label{FV:InterfacePi}
		& {\bf \Pi}_{i\pm\halfone,j}^{\mp} := \sum_{\mu=1}^{Q} \omega_{\mu} \mathbf{U}_{i\pm\halfone,j}^{\mp,\mu}, \qquad
		{\bf \Pi}_{i,j\pm\halfone}^{\mp} := \sum_{\mu=1}^{Q} \omega_{\mu} \mathbf{U}_{i,j\pm\halfone}^{\mu,\mp}, 
	\end{align}
	The condition \eqref{eq:DDF} is a DDF condition for the magnetic field, because $\nabla_h \cdot \mathbf{B}_{ij}$ can be regarded as an approximation to the weak divergence  
	$
		 \frac{1}{\dx\dy}\int_{\partial I_{ij}} \mathbf{B} \cdot \mathbf{n}_{\partial I_{ij}} \dd s = \frac{1}{\dx\dy}\iint_{ I_{ij}} \nabla \cdot \mathbf{B} \dd x \dd y 
	$ 
	on the cell $I_{ij}$, where $\mathbf{n}_{\partial I_{ij}}$ denotes the unit outward normal vector of $\partial I_{ij}$. 
	As it will be shown in \Cref{thm:main}, 
	both conditions \eqref{eq:DDF} and \eqref{LF:cond1} play key roles in ensuring the PP property. 
	However, the values $\mathbf{U}_{i+\halfone,j}^{\pm,\mu}$ and $\mathbf{U}_{i,j+\halfone}^{\mu,\pm}$ obtained from a conventional high-order reconstruction generally do not satisfy 
	    \eqref{eq:DDF} and \eqref{LF:cond1}. To address this, we will propose an efficient DDF projection technique (see \cref{DivF:Sec}) and a simple PP limiter (see \cref{PCP:limiter}) to enforce the conditions \eqref{eq:DDF} and \eqref{LF:cond1}, respectively; see \Cref{fig:DDF-PP} for an illustration.


\begin{figure}[htbp]
	\centering
	\includegraphics[width=1\textwidth]{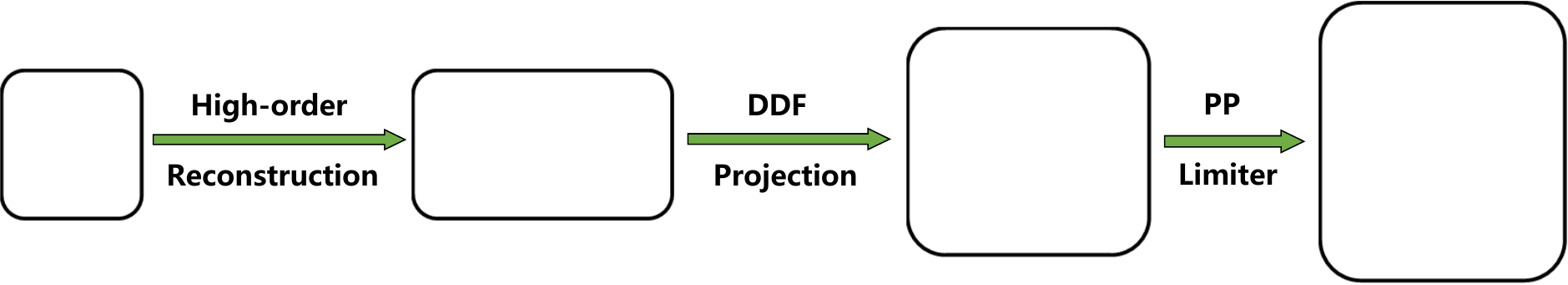}
	\put(-57,48){\scriptsize \color{blue} \mbox{$\mathbf{U}_{i+\halfone,j}^{\pm,\mu}, \mathbf{U}_{i,j+\halfone}^{\mu,\pm}$}}
	\put(-55,30){\small \color{blue} \mbox{$\nabla_h \cdot \mathbf{B}_{ij}=0$}}
	\put(-53,12){\small \color{blue} \mbox{satisfy \eqref{LF:cond1}}}
	\put(-155,43){\scriptsize \color{blue} \mbox{$\widetilde{\mathbf{U}}_{i+\halfone,j}^{\pm,\mu}, \widetilde{\mathbf{U}}_{i,j+\halfone}^{\mu,\pm}$}}
	\put(-153,20){\small \color{blue} \mbox{$\nabla_h \cdot \widetilde{\mathbf{B}}_{ij}=0$}}
	\put(-272,32){\small \color{blue} \mbox{$\widehat{\mathbf{U}}_{i+\halfone,j}^{\pm,\mu}, \widehat{\mathbf{U}}_{i,j+\halfone}^{\mu,\pm}$}}
	\put(-367,32){\small \color{blue} \mbox{$\left\{ \overline {\bf U}_{ij} \right\}$}}
	\caption{\small Outline of the DDF projection and PP limiter for conditions \eqref{eq:DDF} and \eqref{LF:cond1}.}
	\label{fig:DDF-PP}
\end{figure}
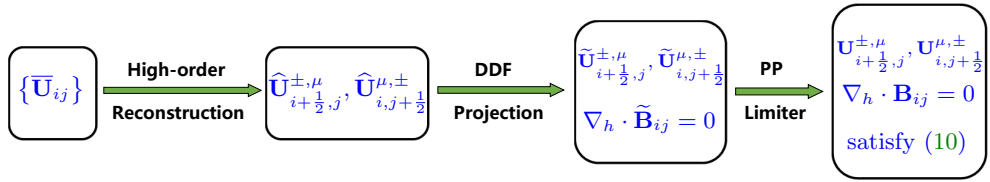


Besides the satisfaction of conditions \eqref{eq:DDF} and \eqref{LF:cond1}, 
a suitable discretization of the source term, which should exactly eliminate the impact of divergence error on the PP property \cite{WuShu2018,WuShu2019}, is also very crucial for guaranteeing the PP property. Seeking the qualified source approximation is highly nontrivial. 
Based on our numerical analysis presented in \cref{PCP:Sec}, we find out the following source term approximation
\begin{equation}\label{Dis:LFSourT2}
	\resizebox{0.92\hsize}{!}{$
		\begin{aligned}
			{\mathbf{S}}_{ij}=&-\frac{1}{\dx}\sum_{\mu=1}^{Q} \omega_{\mu} 
			\left(\frac{1}{2} \jump{B_{1}}_{i+\halfone,j}^{\mu} \mathbf{S}\left(\average{\mathbf{U}}_{i+\halfone,j}^{\mu} \right) + \frac{1}{2} \jump{B_{1}}_{i-\halfone,j}^{\mu} \mathbf{S}\left(\average{\mathbf{U}}_{i-\halfone,j}^{\mu} \right)
			\right) \\
			& -\frac{1}{\dy}\sum_{\mu=1}^{Q} \omega_{\mu} 
			\left(\frac{1}{2} \jump{B_{2}}_{i,j+\halfone}^{\mu} \mathbf{S}\left(\average{\mathbf{U}}_{i,j+\halfone}^{\mu} \right) + \frac{1}{2} \jump{B_{2}}_{i,j-\halfone}^{\mu} \mathbf{S} \left(\average{\mathbf{U}}_{i,j-\halfone}^{\mu} \right)
			\right),
		\end{aligned}$
}
\end{equation}
where $\jump \cdot$ and $\average{\cdot}$ respectively denote the jump and average of the limiting values at a cell interface: 
\begin{equation*}
	\begin{aligned}
		&\jump{B_{1}}_{i+\halfone,j}^{\mu} := (B_{1})_{i+\halfone,j}^{+,\mu} - (B_{1})_{i+\halfone,j}^{-,\mu}, &
		\qquad
		&\jump{B_{2}}_{i,j+\halfone}^{\mu}:= (B_{2})_{i,j+\halfone}^{\mu,+} - (B_{2})_{i,j+\halfone}^{\mu,-}, \\
		&\average{\mathbf{U}}_{i+\halfone,j}^{\mu} := \frac{1}{2} \left( \mathbf{U}_{i+\halfone,j}^{+,\mu} + \mathbf{U}_{i+\halfone,j}^{-,\mu} \right), &
		\qquad
		&\average{\mathbf{U}}_{i,j+\halfone}^{\mu} := \frac{1}{2} \left( \mathbf{U}_{i,j+\halfone}^{\mu,+} + \mathbf{U}_{i,j+\halfone}^{\mu,-} \right). 
	\end{aligned}
\end{equation*}
Notice that the source term approximation \eqref{Dis:LFSourT2} is novel and different from those proposed in \cite{WuShu2018,WuShu2019}. As a result, the theoretical CFL condition \eqref{LF:CFLCond} for PP property milder than those derived in \cite{WuShu2018,WuShu2019}.   

Define 
\begin{equation}\label{eq:defL}
	\mathcal{L}_{ij}\left(\overline{\mathbf{U}}\right):= -\frac{1}{\dx} \left( \widehat{\mathbf{F}}_{1,i+\halfone,j}-\widehat{\mathbf{F}}_{1,i-\halfone,j} \right)  
-\frac{1}{\dy} \left( \widehat{\mathbf{F}}_{2,i,j+\halfone}-\widehat{\mathbf{F}}_{2,i,j-\halfone} \right) 
+{\mathbf{S}}_{ij},
\end{equation}
then the scheme \eqref{FV:DiscreteForm} can be written as 
	$\frac{\dd \overline{\mathbf{U}}_{ij}}{\dd t} = \mathcal{L}_{ij}\left(\overline{\mathbf{U}}\right)$, 
which can be marched forward in time by some Runge--Kutta or multi-step methods. 
For the scheme \eqref{FV:DiscreteForm} with \eqref{Dis:LFSourT2}, we have the following main theorem on the PP property. 

\begin{theorem}\label{thm:main}
	Assume that $\overline{\mathbf{U}}_{ij}\in G$ for all $i$ and $j$. 
	If the values $\mathbf{U}_{i+\halfone,j}^{\pm,\mu}$ and $\mathbf{U}_{i,j+\halfone}^{\mu,\pm}$ satisfy conditions \eqref{eq:DDF} and \eqref{LF:cond1} for all $i$ and $j$, 
	then the scheme \eqref{FV:DiscreteForm} with \eqref{Dis:LFSourT2} is PP, namely, 
	the updated cell averages satisfy 
	\begin{equation}\label{LF:Update}
		\overline{\mathbf{U}}_{ij}^{\dt} := 
		\overline{\mathbf{U}}_{ij} + \dt \mathcal{L}_{ij}(\overline{\mathbf{U}}) \in G, \qquad \forall i,j, 
	\end{equation}
	under the CFL condition
	\begin{equation}\label{LF:CFLCond}
		0<\dt \left(\frac{\alphaLF_1}{\dx} + \frac{\alphaLF_2}{\dy} \right) < \widehat{\omega}_1,
	\end{equation}
	where 
	\begin{equation}
		\alphaLF_1 \geq \hat{\alpha}_1^{\rm{LF}} +  \beta_1, \qquad 
		\alphaLF_2 \geq \hat{\alpha}_2^{\rm{LF}} +  \beta_2
	\end{equation}
	with  
	\begin{align*}
			&\hat{\alpha}_1^{\rm{LF}}:=\underset{i,j,\mu}{\max} \max \left\{ \alpha_1(\mathbf{U}_{i+\halfone,j}^{-,\mu}, \mathbf{U}_{i-\halfone,j}^{+,\mu}), \alpha_1(\mathbf{U}_{i+\halfone,j}^{+,\mu}, \mathbf{U}_{i-\halfone,j}^{-,\mu}) \right\}  ,\\
			&\hat{\alpha}_2^{\rm{LF}}:=\underset{i,j,\mu}{\max} \max \left\{ \alpha_2(\mathbf{U}_{i,j+\halfone}^{\mu,-}, \mathbf{U}_{i,j-\halfone}^{\mu,+}), \alpha_2(\mathbf{U}_{i,j+\halfone}^{\mu,+}, \mathbf{U}_{i,j-\halfone}^{\mu,-}) \right\}, \\
			&\beta_1:=\underset{i,j,\mu}{\max}  \left\{ \frac{\Big| \jump{B_{1}}_{i+\halfone,j}^{\mu} \Big|}{2\sqrt{\average{\rho}_{i+\halfone,j}^{\mu} }} \right\}, \quad 
			\beta_2:=\underset{i,j,\mu}{\max}  \left\{ \frac{\Big| \jump{B_{2}}_{i,j+\halfone}^{\mu} \Big|}{2\sqrt{\average{\rho}_{i,j+\halfone}^{\mu}}} \right\}.
	\end{align*}
Here, $\alpha_1$ and $\alpha_2$ are defined by 
\begin{align}\nonumber 
	\alpha_i(\mathbf{U}, \widetilde{\mathbf{U}}) = \max \left \{ |v_i|+\mathcal{C}_i, |\widetilde{v}_i|+\widetilde{\mathcal{C}}_i, \frac{\sqrt{\rho}v_i + \sqrt{\widetilde{\rho}}\widetilde{v}_i}{\sqrt{\rho} + \sqrt{\widetilde{\rho}}}+\max\{\mathcal{C}_i, \widetilde{\mathcal{C}}_i\} \right \} + \frac{| \mathbf{B}-\widetilde{\mathbf{B}} |}{\sqrt{\rho} + \sqrt{\widetilde{\rho}}},
	\\ \nonumber
	\mathcal{C}_i = \frac{1}{\sqrt{2}} \left[\mathcal{C}_s^2 + \frac{{| \mathbf{B} |}^2}{\rho} + \sqrt{\left(\frac{{  \mathcal{C}_s^2+ |\mathbf{B} |}^2}{\rho}\right)^2 - \frac{4\mathcal{C}_s^2 B_i^2}{\rho}}\right]^{\frac12},
	\qquad 
	\mathcal{C}_s = \frac{p}{\rho \sqrt{2e}}.
\end{align}
\end{theorem}

The proof of \Cref{thm:main} is very technical and nontrivial, and it will be presented in  \cref{PCP:Sec}. 
We emphasize that the three key ingredients (namely, the DDF projection, the PP limiter, and the approximate source term ${\bf S}_{ij}$ in \eqref{Dis:LFSourT2}) are essential in ensuring the PP property \eqref{LF:Update}. Removing any of them in our method may lead to the loss of the PP property. 

\begin{remark}
The conclusion \eqref{LF:Update} in \Cref{thm:main} demonstrates that the scheme \eqref{FV:DiscreteForm} with \eqref{Dis:LFSourT2} is PP, if the first-order forward Euler method is used for time discretization. 
To achieve high-order accuracy in time, one can use the  high-order accurate strong-stability-preserving (SSP) methods \cite{Gottlieb2009}. Because an SSP method is a convex combination of the Euler forward, the PP property is still valid thanks to the convexity of $G$. For example, if we adopt the SSP third-order Runge--Kutta method, then we obtain a fully discrete, high-order accurate, DDFPP finite volume method:
	\begin{equation}\label{RKTime}
	\begin{aligned}
		\overline{\mathbf{U}}^{(1)}_{ij} &= \overline{\mathbf{U}}^{n}_{ij} + \dt \mathcal{L}_{ij}(\overline{\mathbf{U}}^{n}), \\
		\overline{\mathbf{U}}^{(2)}_{ij} &= \frac{3}{4}\overline{\mathbf{U}}^{n}_{ij} + \frac{1}{4} \left(\overline{\mathbf{U}}^{(1)}_{ij} + \dt \mathcal{L}_{ij}(\overline{\mathbf{U}}^{(1)}) \right), \\
		\overline{\mathbf{U}}^{n+1}_{ij} &= \frac{1}{3}\overline{\mathbf{U}}^{n}_{ij} + \frac{2}{3} \left(\overline{\mathbf{U}}^{(2)}_{ij} + \dt \mathcal{L}_{ij}(\overline{\mathbf{U}}^{(2)}) \right),
	\end{aligned}
\end{equation}
where $\dt$ is the time step-size. 
\end{remark}


\subsection{High-order reconstruction}\label{Rec:Sec}
Reconstructing the point values  $\widehat{\mathbf{U}}_{i+\halfone,j}^{\pm,\mu}$ and $\widehat{\mathbf{U}}_{i,j+\halfone}^{\mu,\pm}$ is a 
 key step in the high-order finite volume method \eqref{FV:DiscreteForm}; see \Cref{fig:DDF-PP}.  
As two typical examples, a second-order accurate linear reconstruction and a fifth-order accurate WENO reconstruction are presented to illustrate the procedure on Cartesian meshes. 

\subsubsection{Second-order linear reconstruction}\label{Rec:LinearRec}
Given the cell averages $\{\overline{\mathbf{U}}_{ij}\}$,
one can reconstruct a piecewise linear function: 
\begin{equation}\label{Rec:2ndPoly}
	\widehat{\mathbf{U}}_{ij}(x,y) = \overline{\mathbf{U}}_{ij} + ( \widehat{\mathbf{U}}_x)_{ij} (x-x_i) +  ( \widehat{\mathbf{U}}_y)_{ij} (y-y_j),  \qquad (x,y) \in I_{ij}.
\end{equation}
The local slopes $(\widehat{\mathbf{U}}_x)_{ij}$ and $(\widehat{\mathbf{U}}_y)_{ij}$ can be defined via the van Albada limiter by 

\begin{equation}
	\resizebox{0.92\hsize}{!}{$
		\begin{aligned}\label{Rec:2ndSlopx}
			(\widehat{\mathbf{U}}_x)_{ij} &=\frac{\left(\Big[\frac{\overline{\mathbf{U}}_{i+1,j} - \overline{\mathbf{U}}_{ij}}{\dx}\Big]^2 + \epsilon_R\right) \circ  \frac{\overline{\mathbf{U}}_{ij} - \overline{\mathbf{U}}_{i-1,j}}{\dx} + \left(\Big[\frac{\overline{\mathbf{U}}_{ij} - \overline{\mathbf{U}}_{i-1,j}}{\dx}\Big]^2 + \epsilon_L\right) \circ \frac{\overline{\mathbf{U}}_{i+1,j} - \overline{\mathbf{U}}_{ij}}{\dx}}{\Big[\frac{\overline{\mathbf{U}}_{ij} - \overline{\mathbf{U}}_{i-1,j}}{\dx}\Big]^2 + \Big[\frac{\overline{\mathbf{U}}_{i+1,j} - \overline{\mathbf{U}}_{ij}}{\dx}\Big]^2 + \epsilon_L + \epsilon_R}, 
		\end{aligned}
		$}
\end{equation}

\begin{equation}
	\resizebox{0.92\hsize}{!}{$
		\begin{aligned}\label{Rec:2ndSlopy}
			(\widehat{\mathbf{U}}_y)_{ij} &=\frac{\left(\Big[\frac{\overline{\mathbf{U}}_{i,j+1} - \overline{\mathbf{U}}_{ij}}{\dy}\Big]^2 + \epsilon_U\right) \circ \frac{\overline{\mathbf{U}}_{ij} - \overline{\mathbf{U}}_{i,j-1}}{\dy} + \left(\Big[\frac{\overline{\mathbf{U}}_{ij} - \overline{\mathbf{U}}_{i,j-1}}{\dy}\Big]^2 + \epsilon_D\right) \circ \frac{\overline{\mathbf{U}}_{i,j+1} - \overline{\mathbf{U}}_{ij}}{\dy}}{\Big[\frac{\overline{\mathbf{U}}_{ij} - \overline{\mathbf{U}}_{i,j-1}}{\dy}\Big]^2 + \Big[\frac{\overline{\mathbf{U}}_{i,j+1} - \overline{\mathbf{U}}_{ij}}{\dy}\Big]^2 + \epsilon_D + \epsilon_U},
		\end{aligned}
		$}
\end{equation}
where $[\mathbf{W}]^2 = \left(W_1^2,\cdots,W_8^2\right)^\top$ for any vector $\mathbf{W}=\left(W_1,\cdots,W_8\right)^\top$, the symbol ``$\circ$'' represents the Hadamard product (element-wise multiplication), and 
we take $\epsilon_L=\epsilon_R=3\dx$ and $\epsilon_D=\epsilon_U=3 \dy$. To achieve a second-order accurate scheme, one can choose the midpoint quadrature rule with $Q=1$ for 
evaluating the numerical fluxes \eqref{FV:Flux} on each cell edge. The limiting values $\widehat{\mathbf{U}}_{i+\halfone,j}^{\pm,1}$ and $\widehat{\mathbf{U}}_{i,j+\halfone}^{1,\pm}$ can be evaluated by using the reconstructed piecewise linear function \eqref{Rec:2ndPoly}: 
\begin{align}\label{Rec:2ndValueX-1}
	& \widehat{\mathbf{U}}_{i+\halfone,j}^{-,1}  = \overline{\mathbf{U}}_{ij} + (\widehat{\mathbf{U}}_x)_{ij} \frac{\dx}{2}, \qquad 
	\widehat{\mathbf{U}}_{i+\halfone,j}^{+,1}  = \overline{\mathbf{U}}_{i+1,j} - (\widehat{\mathbf{U}}_x)_{i+1,j} \frac{\dx}{2},
	\\
	\label{Rec:2ndValueY-1}
	& \widehat{\mathbf{U}}_{i,j+\halfone}^{1,-} = \overline{\mathbf{U}}_{ij} + (\widehat{\mathbf{U}}_y)_{ij} \frac{\dy}{2}, \qquad 
	\widehat{\mathbf{U}}_{i,j+\halfone}^{1,+} = \overline{\mathbf{U}}_{i,j+1} - (\widehat{\mathbf{U}}_y)_{i,j+1} \frac{\dy}{2}.  
\end{align}

\subsubsection{Fifth-order WENO reconstruction}\label{Rec:WENO}
We recall the fifth-order accurate WENO reconstruction on the Cartesian meshes \cite{shu1998essentially,zhang2010}. 
To achieve the fifth-order accuracy, one can utilize the Gauss quadrature with $Q=3$ or the  Gauss--Lobatto quadrature with $Q=4$ for evaluating the numerical fluxes \eqref{FV:Flux} on each cell edge. 
We notice that, if we use the three-point Gauss quadrature rule, negative linear weights would appear 
in the WENO reconstruction for the midpoint \cite{ShiHuShu2002}. Negative weights
must be dealt with care \cite{ShiHuShu2002}, otherwise numerical oscillations and instability may occur. To avoid such risk, we adopt the four-point Gauss-Lobatto quadrature rule in our fifth-order finite volume WENO scheme. 
The two-dimensional (2D) WENO reconstruction procedure is performed in a dimension-by-dimension fashion by the following two steps.  


\begin{enumerate}
	\item[Step 1.] Given the 2D cell averages $\{\overline{\mathbf{U}}_{i,j}\}$, we perform 
	 1D WENO reconstructions to obtain four edge averages: 
	\begin{align*}
		& \left\{\overline{\mathbf{U}}_{i,j} \right\}  ~ \to ~ \left\{\bar{\mathbf{U}}_{i+\frac12,j}^+,  \bar{\mathbf{U}}_{i+\frac12,j}^- \right \}  \quad \mbox{for fixed $j$},
		\\
		& \left\{\overline{\mathbf{U}}_{i,j} \right\}  ~ \to ~ \left\{\bar{\mathbf{U}}_{i,j+\frac12}^+,  \bar{\mathbf{U}}_{i,j+\frac12}^- \right \}  \quad \mbox{for fixed $i$}.
	\end{align*}
		\item[Step 2.] Based on the edge averages, we perform 1D WENO reconstructions to obtain the point values at the Gauss--Lobatto quadrature nodes: 
	\begin{align*}
	&\left\{\bar{\mathbf{U}}_{i+\frac12,j}^+ \right\}  ~ \to ~ 
	\left\{\widehat{\mathbf{U}}_{i+\frac12,j}^{+,\mu} \right\}_{\mu=1}^Q, 
		\quad \left\{\bar{\mathbf{U}}_{i+\frac12,j}^- \right\}  ~ \to ~ 
	\left\{\widehat{\mathbf{U}}_{i+\frac12,j}^{-,\mu} \right\}_{\mu=1}^Q  \quad \mbox{for fixed $i$},
	\\
	&\left\{\bar{\mathbf{U}}_{i,j+\frac12}^+ \right\}  
	~ \to ~  \left\{\widehat{\mathbf{U}}_{i,j+\frac12}^{\mu,+} \right\}_{\mu=1}^Q,   
	\quad 
	\left\{\bar{\mathbf{U}}_{i,j+\frac12}^- \right\}  
	~ \to ~  \left\{\widehat{\mathbf{U}}_{i,j+\frac12}^{\mu,-} \right\}_{\mu=1}^Q 
	\quad \mbox{for fixed $j$}.
\end{align*} 
\end{enumerate} 
In our implementation, we employ the WENO-Z reconstruction \cite{BORGES20083191} in conjunction with local 
characteristic decomposition to improve the performance.

\subsection{Discretely divergence-free projection technique}\label{DivF:Sec}
This subsection proposes the DDF projection, which modifies  $\{\widehat{\mathbf{B}}_{i+\halfone,j}^{\pm,\mu}, \widehat{\mathbf{B}}_{i,j+\halfone}^{\mu,\pm}\}$ to 
$\{\widetilde{\mathbf{B}}_{i+\halfone,j}^{\pm,\mu}, \widetilde{\mathbf{B}}_{i,j+\halfone}^{\mu,\pm}\}$, 
so as to enforce the DDF condition 
\begin{equation}\label{eq:DDF1}
	\nabla_h \cdot \widetilde{\mathbf{B}}_{ij} := \sum_{\mu=1}^{Q} \omega_{\mu}
	\left( \frac{ (\widetilde B_{1})_{i+\halfone,j}^{-,\mu} -   ( \widetilde B_{1})_{i-\halfone,j}^{+,\mu}}{\dx}
	+ \frac{ ( \widetilde B_{2})_{i,j+\halfone}^{\mu,-}  -  (\widetilde B_{2})_{i,j-\halfone}^{\mu,+} }{\dy} \right)=0. 
\end{equation}
The DDF projection technique is very simple, efficient, and easy to implement.


\subsubsection{DDF projection for second-order scheme}\label{DivF:2nd}
We first consider the second-order case to gain some insight on the DDF projection. 
In this case, the reconstructed magnetic field is a linear function within each cell: 
\begin{equation}\label{Rec:2ndPolyB}
	\widehat{\mathbf{B}}_{ij}(x,y) = \overline{\mathbf{B}}_{ij} + ( \widehat{\mathbf{B}}_x)_{ij} (x-x_i) +  ( \widehat{\mathbf{B}}_y)_{ij} (y-y_j),  \qquad (x,y) \in I_{ij}.
\end{equation}
Notice that $Q=1$.  According to \eqref{Rec:2ndValueX-1} and \eqref{Rec:2ndValueY-1}, 
we have  
 \begin{align*}
 	\nabla_h \cdot \widehat{\mathbf{B}}_{ij} & :=  
 	\frac{ (\widehat B_{1})_{i+\halfone,j}^{-,1} -   ( \widehat B_{1})_{i-\halfone,j}^{+,1}}{\dx}
 	+ \frac{ ( \widehat B_{2})_{i,j+\halfone}^{1,-}  -  (\widehat B_{2})_{i,j-\halfone}^{1,+} }{\dy} 
 	\\
 	& = (( \widehat{B}_1)_x)_{ij} + (( \widehat{B}_2)_y)_{ij}   
 	= \nabla \cdot \widehat{\mathbf{B}}_{ij}(x,y), 
 \end{align*}
where the slopes  $(( \widehat{B}_1)_x)_{ij}$ and $(( \widehat{B}_2)_y)_{ij}$ are independently computed by \eqref{Rec:2ndSlopx} and \eqref{Rec:2ndSlopy}, respectively. 
Therefore, in general, the divergence $(( \widehat{B}_1)_x)_{ij} + (( \widehat{B}_2)_y)_{ij}$ is not necessarily zero.  
In order to enforce the DDF condition \eqref{eq:DDF1}, we can modify the slopes $(( \widehat{B}_1)_x)_{ij}$ and $(( \widehat{B}_2)_y)_{ij}$
into 
$(( \widetilde{B}_1)_x)_{ij}$ and $(( \widetilde{B}_2)_y)_{ij}$ such that 
\begin{equation}\label{key113}
	(( \widetilde{B}_1)_x)_{ij} + (( \widetilde{B}_2)_y)_{ij} = 0.
\end{equation}
Define 
$
S_1 := (( \widehat{B}_1)_x)_{ij} - (( \widetilde{B}_1)_x)_{ij}$ and  $S_2 := (( \widehat{B}_2)_y)_{ij} - (( \widetilde{B}_2)_y)_{ij} 
$ 
as the differences between the original and modified slopes. Then the desired DDF condition \eqref{key113} is equivalent to 
$$
S_1 + S_2 = (( \widehat{B}_1)_x)_{ij} + (( \widehat{B}_2)_y)_{ij} =  \nabla_h \cdot \widehat{\mathbf{B}}_{ij}. 
$$
To keep the original accuracy of the scheme, the differences between the original and modified slopes should be as small as possible. Hence we propose to determine $S_1$ and $S_2$ by solving the following optimization problem
 \begin{equation}\label{DivB:2ndOpt}
 	\begin{aligned}
 		&\mbox{\rm minimize} \quad S_1^2 \Delta x^2 + S_2^2  \Delta y^2 , \\
 		&\mbox{\rm subject to} \quad S_1+S_2=\nabla_h \cdot \widehat{\mathbf{B}}_{ij},
 	\end{aligned}
 \end{equation}
which leads to 
\begin{equation}\label{DivB:2ndSlopS}
	S_1 = \frac{\dy^2  \nabla_h \cdot \widehat{\mathbf{B}}_{ij} }{\dx^2 + \dy^2} , \qquad
	S_2 = \frac{\dx^2 \nabla_h \cdot \widehat{\mathbf{B}}_{ij} }{\dx^2 + \dy^2}.
\end{equation}
Therefore, the modified slopes are given by 
\begin{equation}\label{DivB:2ndMod3}
	(( \widetilde{B}_1)_x)_{ij}= (( \widehat{B}_1)_x)_{ij} - \frac{\dy^2  \nabla_h \cdot \widehat{\mathbf{B}}_{ij} }{\dx^2 + \dy^2} , \qquad   (( \widetilde{B}_2)_y)_{ij} = (( \widehat{B}_2)_y)_{ij} - \frac{\dx^2 \nabla_h \cdot \widehat{\mathbf{B}}_{ij} }{\dx^2 + \dy^2}.
\end{equation}
This implies the modified point values are given by 
\begin{equation}\label{DivB:2ndMod2}
	(\widetilde B_{1})_{i\pm \halfone,j}^{\mp,1} =  (\widehat B_{1})_{i\pm \halfone,j}^{\mp,1} \mp A_1, 
\qquad 
(\widetilde B_{2})_{i,j\pm \halfone}^{1,\mp} =  (\widehat B_{2})_{i,j\pm \halfone}^{1,\mp} \mp A_2
\end{equation}
with 
	\begin{equation}\label{DivB:ValuesMod}
	A_1= \frac{\dx}{2} S_1= \frac{ \dx \nabla_h \cdot \widehat{\mathbf{B}}_{ij}  }{2(1+(\dx/\dy)^2)}, \qquad
	A_2= \frac{\dy}{2} S_2= \frac{ \dy \nabla_h \cdot \widehat{\mathbf{B}}_{ij}  }{2(1+(\dy/\dx)^2)}.
\end{equation}
We have the following observations.  

\begin{proposition}\label{prop0}
	The quantities $A_1$ and $A_2$ defined in \eqref{DivB:ValuesMod} are the solution to the following optimization problem
	 \begin{equation}\label{DivB:2ndOpt2}
		\begin{aligned}
			&\mbox{\rm minimize} \quad A_1^2   + A_2^2, \\
			&\mbox{\rm subject to} \quad \frac{2A_1}{\dx} + \frac{2A_2}{\dy}=\nabla_h \cdot \widehat{\mathbf{B}}_{ij}.
		\end{aligned}
	\end{equation}
\end{proposition} 

\begin{proof}
	Substituting the constraint $A_2 = \dy ( \frac12 \nabla_h \cdot \widehat{\mathbf{B}}_{ij} -  \frac{A_1}{\dx})$ into the objective function $A_1^2   + A_2^2$ and then solving the resulting minimization problem with respect to $A_1$, we can obtain the minimum point \eqref{DivB:ValuesMod}.  
\end{proof}

\begin{remark}
Geometrically, the solution to problem \eqref{DivB:2ndOpt2} is the Euclidean projection of $(0,0)$ to the line $\frac{2A_1}{\dx} + \frac{2A_2}{\dy}=\nabla_h \cdot \widehat{\mathbf{B}}_{ij}$ in the two-dimensional plane.    
	Let ${\mathcal P}$ denote the following mapping 
 $$\left\{ (\widehat B_{1})_{i \mp \halfone,j}^{\pm,1}, (\widehat B_{2})_{i,j\mp \halfone}^{1,\pm} \right\} ~~
 \rightarrow ~~ \left\{(\widetilde B_{1})_{i\mp \halfone,j}^{\pm,1}, (\widetilde B_{2})_{i,j\mp \halfone}^{1,\pm} \right\}.$$ 
This mapping is associated with the following matrix 
\begin{equation}
	{\bf P}=
	\begin{pmatrix}
		1-\eta &  \eta & -\mu & \mu\\
		\eta & 1-\eta & \mu &  -\mu \\
		-\mu & \mu & 1-\zeta & \zeta \\
		\mu & -\mu & \zeta & 1- \zeta
	\end{pmatrix},
\end{equation} 
where $\eta = \frac{1}{2(1+(\dx/\dy)^2)}$, $\mu = \frac{\dx}{\dy} \eta$, and $\zeta=\frac{1}{2(1+(\dy/\dx)^2)}$. 
We observe that
$${\bf P}^2 = {\bf P}, \qquad {\mathcal P}^2 = {\mathcal P},$$
thus ${\mathcal P}$ is a projection operator. 
Hence we call the modification \eqref{DivB:2ndMod2} ``{\em DDF projection}''. 
\end{remark}

\begin{proposition}\label{prop1}
	For the second-order reconstruction \eqref{Rec:2ndPolyB}, the DDF modification of the slopes \eqref{DivB:2ndMod3} is equivalent to the DDF projection of the point values \eqref{DivB:2ndMod2}.
\end{proposition} 

\begin{proof}
Clearly, the DDF modification of the slopes \eqref{DivB:2ndMod3} implies that the associated point values are  modified into \eqref{DivB:2ndMod2} with \eqref{DivB:ValuesMod}. 
Conversely, if we modify the point values as \eqref{DivB:2ndMod2} with \eqref{DivB:ValuesMod}, then we can reconstruct a linear function by using the modified point values. We can verify that the slopes of the new  linear function satisfy \eqref{DivB:2ndMod3}. 
\end{proof}

\Cref{prop1} demonstrates that we only need to modify the reconstructed point values via  \eqref{DivB:2ndMod2} at the cell interfaces, to enforce the DDF constraint \eqref{eq:DDF1}.

\subsubsection{DDF projection for higher-order schemes}\label{DivF:5th}
We now discuss the DDF projection technique to enforce the condition \eqref{eq:DDF1} for $k$th-order finite volume schemes with $k\ge 3$. 
Unlike the second-order scheme, for higher-order schemes such as the fifth-order WENO scheme, the reconstruction directly gives only the point values at the cell-interface without explicitly building any approximation polynomials. Therefore, the DDF condition  \eqref{eq:DDF1} cannot be enforced  by the DF techniques (cf.~\cite{Balsara2004,Balsara2009,xu2016divergence,liu2021new}) based on approximation polynomials. 

Motivated by the DDF projection technique \eqref{DivB:2ndMod2} in the second-order case, we propose the following DDF projection on higher-order reconstructed point values at the cell-interface: 
\begin{equation}\label{DivB:highOrderMod1}
	\begin{aligned}
			(\widetilde B_{1})_{i\pm \halfone,j}^{\mp,\mu} =  (\widehat B_{1})_{i\pm \halfone,j}^{\mp,\mu} \mp A_1, 
		\qquad 
		(\widetilde B_{2})_{i,j\pm \halfone}^{\mu,\mp} =  (\widehat B_{2})_{i,j\pm \halfone}^{\mu,\mp} \mp A_2, \qquad 1\leq \mu \leq Q, 
	\end{aligned}
\end{equation}
where $A_1$ and $A_2$ are defined by 
\begin{equation}\label{DivB:highOrderValuesMod1}
	A_1=\frac{  \dx  \nabla_h \cdot \widehat{\mathbf{B}}_{ij} }{2(1+(\dx/\dy)^2)}, \qquad
	A_2=\frac{ \dy  \nabla_h \cdot \widehat{\mathbf{B}}_{ij} }{2(1+(\dy/\dx)^2)} 
\end{equation}
with 
$$
\nabla_h \cdot \widehat{\mathbf{B}}_{ij} :=  	 \sum_{\mu=1}^{Q} \omega_{\mu}
\left( \frac{ (\widehat B_{1})_{i+\halfone,j}^{-,\mu} -   ( \widehat B_{1})_{i-\halfone,j}^{+,\mu}}{\dx}
+ \frac{ ( \widehat B_{2})_{i,j+\halfone}^{\mu,-}  -  (\widehat B_{2})_{i,j-\halfone}^{\mu,+} }{\dy} \right). 
$$

\begin{theorem}\label{prop2}
The modified point values \eqref{DivB:highOrderMod1} satisfy the DDF condition \eqref{eq:DDF1}. 
\end{theorem} 

\begin{proof}
Substituting the modified values \eqref{DivB:highOrderMod1} into the formulation of $\nabla_h \cdot \widetilde{\mathbf{B}}_{ij}$  in \eqref{eq:DDF1}, we obtain 
\begin{equation*}
	\resizebox{1\hsize}{!}{$
		\begin{aligned}
			\nabla_h \cdot \widetilde{\mathbf{B}}_{ij} & = \sum_{\mu=1}^{Q} \omega_{\mu}
			\left(  \frac{ \left[ (\widehat B_{1})_{i+\halfone,j}^{-,\mu} - A_1 \right] - \left[  ( \widehat B_{1})_{i-\halfone,j}^{+,\mu} +A_1\right]  }{\dx}
			+ \frac{   \left[ ( \widehat B_{2})_{i,j+\halfone}^{\mu,-} - A_2 \right]  - \left[  (\widehat B_{2})_{i,j-\halfone}^{\mu,+} +A_2 \right] }{\dy} \right)
			\\
			& = \nabla_h \cdot \widehat{\mathbf{B}}_{ij}  - \left(\frac{2A_1}{\dx} + \frac{2A_2}{\dy} \right) \sum_{\mu=1}^{Q} \omega_{\mu} = \nabla_h \cdot \widehat{\mathbf{B}}_{ij} - \frac{  \nabla_h \cdot \widehat{\mathbf{B}}_{ij} }{1+(\dx/\dy)^2} -  \frac{  \nabla_h \cdot \widehat{\mathbf{B}}_{ij} }{1+(\dy/\dx)^2} =  0.  
		\end{aligned}
		$}
\end{equation*}
The proof is completed. 
\end{proof}

Now, we raise a natural and important question: Does the DDF projection  
\eqref{DivB:highOrderMod1} keep the original high-order accuracy of the reconstructed point values? 
This question is answered by the following theorem. 

\begin{theorem}
Assume that the exact magnetic field, denoted by ${\bf B}=(B_1,B_2,B_3)$, is a $C^k$ function. Assume that the reconstructed point values are $k$th-order approximations to the exact point values, {i.e.}, we have 
\begin{equation}\label{key433}
	(\widehat B_{1})_{i\pm \halfone,j}^{\mp,\mu} - B_1(x_{i\pm \halfone}, y_j^\mu) = {\mathcal O} (h^k), \qquad 
(\widehat B_{2})_{i,j\pm \halfone}^{\mu,\mp} - B_2( x_i^\mu, y_{j\pm \halfone} ) = {\mathcal O} (h^k),
\end{equation}
where $h=\max \{\Delta x, \Delta y\}$. Then the DDF modified point values are also $k$th-order approximations to the exact point values, namely, we have 
 $$
 (\widetilde B_{1})_{i\pm \halfone,j}^{\mp,\mu} - B_1(x_{i\pm \halfone}, y_j^\mu) = {\mathcal O} (h^k), \qquad 
 (\widetilde B_{2})_{i,j\pm \halfone}^{\mu,\mp} - B_2( x_i^\mu, y_{j\pm \halfone} ) = {\mathcal O} (h^k),
 $$
 This means the DDF projection \eqref{DivB:highOrderMod1} maintains the original high-order accuracy of the reconstruction.  
\end{theorem}

\begin{proof}
Based on the accuracy of the adopted $Q$-point quadrature rule, we have 
 \begin{align*}
 &	\left| \sum_{\mu=1}^Q \omega_\mu B_1 ( x_{i\pm \halfone}, y_j^\mu ) - \frac1{\Delta y} \int_{y_{j-\frac12}}^{y_{j+\halfone}} B_1 ( x_{i\pm \halfone}, y  ) {\rm d} y \right| \le {\mathcal O}(h^k),
 \\
  &	\left| \sum_{\mu=1}^Q \omega_\mu B_2 ( x_i^\mu, y_{j\pm \halfone} ) - \frac1{\Delta x} \int_{x_{i-\frac12}}^{x_{i+\halfone}} B_2 ( x, y_{j\pm \halfone}  ) {\rm d} x \right| \le {\mathcal O}(h^k).
 \end{align*}
Recall that the exact magnetic field is divergence-free. 
With the help of the divergence theorem and the triangle inequality, we obtain 
{\small \begin{align*}
	&\left| \nabla_h \cdot \widehat{\mathbf{B}}_{ij} \right| = \left| \nabla_h \cdot \widehat{\mathbf{B}}_{ij} -  \frac{1}{\dx\dy}\iint_{ I_{ij}} \nabla \cdot \mathbf{B} \dd x \dd y  \right| 
	\\
	& = \left| \nabla_h \cdot \widehat{\mathbf{B}}_{ij} -  \frac{1}{\dx\dy}\int_{\partial I_{ij}} \mathbf{B} \cdot \mathbf{n}_{\partial I_{ij}} \dd s \right| 
	\\
	& \le \left| \nabla_h \cdot \widehat{\mathbf{B}}_{ij} -  \sum_{\mu=1}^{Q} \omega_{\mu}
	\left( \frac{  B_1 ( x_{i + \halfone}, y_j^\mu ) -  B_1 ( x_{i - \halfone}, y_j^\mu )  }{\dx}
	+ \frac{  B_2 ( x_i^\mu, y_{j + \halfone} )  -   B_2 ( x_i^\mu, y_{j - \halfone} ) }{\dy} \right) \right|
	\\
	& \quad  +  \frac1{\dx} \left| \sum_{\mu=1}^Q \omega_\mu B_1 ( x_{i + \halfone}, y_j^\mu ) - \frac1{\Delta y} \int_{y_{j-\frac12}}^{y_{j+\halfone}} B_1 ( x_{i + \halfone}, y  ) {\rm d} y \right|
	\\
	& \quad +  \frac1{\dx} \left| \sum_{\mu=1}^Q \omega_\mu B_1 ( x_{i - \halfone}, y_j^\mu ) - \frac1{\Delta y} \int_{y_{j-\frac12}}^{y_{j+\halfone}} B_1 ( x_{i - \halfone}, y  ) {\rm d} y \right|
	\\
	& \quad +  \frac1{\dy} 	\left| \sum_{\mu=1}^Q \omega_\mu B_2 ( x_i^\mu, y_{j+ \halfone} ) - \frac1{\Delta x} \int_{x_{i-\frac12}}^{x_{i+\halfone}} B_2 ( x, y_{j+ \halfone}  ) {\rm d} x \right|
	\\
	& \quad + \frac1{\dy} 	\left| \sum_{\mu=1}^Q \omega_\mu B_2 ( x_i^\mu, y_{j - \halfone} ) - \frac1{\Delta x} \int_{x_{i-\frac12}}^{x_{i+\halfone}} B_2 ( x, y_{j- \halfone}  ) {\rm d} x \right|
	\\
	& \le \frac{1}{\dx} \sum_{\mu=1}^Q  \omega_\mu \left( \left| (\widehat B_{1})_{i + \halfone,j}^{-,\mu} - B_1(x_{i + \halfone}, y_j^\mu) \right| + \left| (\widehat B_{1})_{i - \halfone,j}^{+,\mu} - B_1(x_{i - \halfone}, y_j^\mu) \right|  \right) 
	\\
	& \quad + \frac{1}{\dy} \sum_{\mu=1}^Q  \omega_\mu \left( \left| (\widehat B_{2})_{i,j + \halfone}^{\mu,-} - B_2( x_i^\mu, y_{j + \halfone} ) \right| + \left|(\widehat B_{2})_{i,j- \halfone}^{\mu,+} - B_2( x_i^\mu, y_{j-\halfone} ) \right|  \right)  + {\mathcal O}(h^{k-1}) 
	\\
	& = {\mathcal O}(h^{k-1}).
\end{align*}}
It then follows from \eqref{DivB:highOrderValuesMod1} that 
\begin{equation}\label{key3123}
	\left| A_1 \right| \le \Delta x {\mathcal O}(h^{k-1}) = {\mathcal O}(h^{k}), 
	\qquad \left| A_2 \right| \le \Delta y {\mathcal O}(h^{k-1}) = {\mathcal O}(h^{k}).
\end{equation}
According to \eqref{key3123}, \eqref{key433}, and \eqref{DivB:highOrderMod1}, we have 
\begin{align*}
	&(\widetilde B_{1})_{i\pm \halfone,j}^{\mp,\mu} =  B_1(x_{i\pm \halfone}, y_j^\mu) + {\mathcal O} (h^k) \mp A_1 = B_1(x_{i\pm \halfone}, y_j^\mu) + {\mathcal O} (h^k), \\
&(\widetilde B_{2})_{i,j\pm \halfone}^{\mu,\mp} =   B_2( x_i^\mu, y_{j\pm \halfone} ) + {\mathcal O} (h^k) \mp A_2 = B_2( x_i^\mu, y_{j\pm \halfone} ) + {\mathcal O} (h^k). 
\end{align*}
The proof is completed. 
\end{proof}

The numerical results shown in \Cref{Ex:Vortex} will further confirm that the proposed DDF projection does not destroy the high-order accuracy of our schemes. 



\begin{remark}
	The DDF projection \eqref{DivB:highOrderMod1} only involves the point values $(\widehat B_{1})_{i\pm \halfone,j}^{\mp,\mu}$ and $(\widehat B_{2})_{i,j\pm \halfone}^{\mu,\mp}$, and does not require any approximation polynomials in the reconstruction step. 
	Hence the DDF projection \eqref{DivB:highOrderMod1} is very easy to implement. 
\end{remark}

\subsection{Positivity-preserving limiter}\label{PCP:limiter}

In this subsection, we design a simple PP limiter to enforce the condition \eqref{LF:cond1}, as illustrated in \Cref{fig:DDF-PP}. 
Assume that $\overline {\bf U}_{ij }\in G$ for all $i$ and $j$. 
For each cell $I_{ij}$, 
given the point values $\{\widetilde{\mathbf{U}}_{i+\halfone,j}^{\pm,\mu}, \widetilde{\mathbf{U}}_{i,j+\halfone}^{\mu,\pm}\}$, 
the PP limiting procedure modifies them to $\{{\mathbf{U}}_{i+\halfone,j}^{\pm,\mu}, {\mathbf{U}}_{i,j+\halfone}^{\mu,\pm}\}$, which satisfy \eqref{LF:cond1}, via the following two steps.


\noindent
{\bf Step 1}. 
 First, modify the density:
	\begin{equation}
		\dbtilde{\rho}_{i\pm\halfone,j}^{\mp,\mu}=\theta_1 \left( \widetilde \rho_{i\pm\halfone,j}^{\mp,\mu} - \bar{\rho}_{ij}\right) + \bar{\rho}_{ij}, \qquad
		\dbtilde{\rho}_{i,j\pm\halfone}^{\mu,\mp}=\theta_1 \left( \widetilde \rho_{i,j\pm\halfone}^{\mu,\mp} - \bar{\rho}_{ij}\right) + \bar{\rho}_{ij}
	\end{equation}
	with
	\begin{equation*}
		\theta_1 := \min \left\{ \left| \frac{\bar{\rho}_{ij} -\epsilon_1}{\bar{\rho}_{ij} -\rho_{\min}} \right|, 1 \right\},
		\quad
		\rho_{\rm min} =
		\begin{cases}
		\min \left\{ \widetilde \rho_{i\pm\halfone,j}^{\mp,\mu}, \widetilde \rho_{i,j\pm\halfone}^{\mu,\mp} \right\},  &\mbox{if } k=2,
		\\
		 \min \left\{ \widetilde \rho_{i\pm\halfone,j}^{\mp,\mu}, \widetilde \rho_{i,j\pm\halfone}^{\mu,\mp}, \rho_{ij}^{\star}  \right\},  &\mbox{if } k\ge 3,
		 \end{cases}
	\end{equation*}
	where we take $\epsilon_1 = \min\{10^{-13}, \bar{\rho}_{ij} \}$ to avoid the influence of the round-off errors, and $\rho_{ij}^{\star}$ is the first component of ${\bf \Pi}_{ij}$ defined by  
	\begin{align*}
		& \widetilde {\bf \Pi}_{ij} := \frac{\overline{\mathbf{U}}_{ij} - \frac{\widehat{\omega}_1}{\lambda}  \left[ \lambda_1 \left( \widetilde {\bf \Pi}_{i-\halfone,j}^{+} + \widetilde {\bf \Pi}_{i+\halfone,j}^{-} \right) + \lambda_2 \left( \widetilde {\bf \Pi}_{i,j-\halfone}^{+} + \widetilde {\bf \Pi}_{i,j+\halfone}^{-} \right) \right]}{1-2\widehat{\omega}_1},
		\\
		& \widetilde {\bf \Pi}_{i\pm\halfone,j}^{\mp} := \sum_{\mu=1}^{Q} \omega_{\mu} \widetilde{\mathbf{U}}_{i\pm\halfone,j}^{\mp,\mu}, \qquad
		\widetilde {\bf \Pi}_{i,j\pm\halfone}^{\mp} := \sum_{\mu=1}^{Q} \omega_{\mu} \widetilde{\mathbf{U}}_{i,j\pm\halfone}^{\mu,\mp}.
	\end{align*}
	This step modifies $ \{ \widetilde{\mathbf{U}}_{i+\halfone,j}^{\pm,\mu}, \widetilde{\mathbf{U}}_{i,j+\halfone}^{\mu,\pm} \}$ to  
	$\{ \dbtilde{\mathbf{U}}_{i+\halfone,j}^{\pm,\mu}, \dbtilde{\mathbf{U}}_{i,j+\halfone}^{\mu,\pm} \}$ with $\dbtilde{\mathbf{U}}:= (\dbtilde{\rho}, \widetilde{\mathbf{m}},\widetilde{\mathbf{B}}, \widetilde E)^{\top}$.

\noindent	
{\bf Step 2}. Modify the point values $\{ \dbtilde{\mathbf{U}}_{i+\halfone,j}^{\pm,\mu}, \dbtilde{\mathbf{U}}_{i,j+\halfone}^{\mu,\pm} \}$ to $\{ {\mathbf{U}}_{i+\halfone,j}^{\pm,\mu}, {\mathbf{U}}_{i,j+\halfone}^{\mu,\pm} \}$ by 
	\begin{equation}\label{eq:PPlimited}
		{\mathbf{U}}_{i\pm\halfone,j}^{\mp,\mu}=\theta_2 \left(\dbtilde{\mathbf{U}}_{i\pm\halfone,j}^{\mp,\mu} - \overline{\mathbf{U}}_{ij}\right) + \overline{\mathbf{U}}_{ij}, \qquad
		{\mathbf{U}}_{i,j\pm\halfone}^{\mu,\mp}=\theta_2 \left(\dbtilde{\mathbf{U}}_{i,j\pm\halfone}^{\mu,\mp} - \overline{\mathbf{U}}_{ij}\right) + \overline{\mathbf{U}}_{ij}
	\end{equation}
	with
	\begin{equation*}
		\resizebox{1\hsize}{!}{$
			\theta_2 = \min \left\{ \left| \frac{\mathcal{E}(\overline{\mathbf{U}}_{ij}) -\epsilon_2}{\mathcal{E}(\overline{\mathbf{U}}_{ij}) -{\mathcal{E}}_{\min}} \right|, 1 \right\},
			~
			{\mathcal{E}}_{\min} = 
			\begin{cases} 
				\min \left\{\mathcal{E}(\dbtilde{\mathbf{U}}_{i\pm\halfone,j}^{\mp,\mu}), \mathcal{E}(\dbtilde{\mathbf{U}}_{i,j\pm\halfone}^{\mu,\mp}) \right\},  &\mbox{if } k=2,
				\\
				\min \left\{\mathcal{E}(\dbtilde{\mathbf{U}}_{i\pm\halfone,j}^{\mp,\mu}), \mathcal{E}(\dbtilde{\mathbf{U}}_{i,j\pm\halfone}^{\mu,\mp}), \mathcal{E} \left( \dbtilde{\bf \Pi}_{ij} \right) \right\},  &\mbox{if } k\ge 3,
			\end{cases}
			$}
	\end{equation*}
	where $\epsilon_2 = \min\{10^{-13}, \mathcal{E}(\overline{\mathbf{U}}_{ij}) \} $, and $\dbtilde{\bf \Pi}_{ij}$ is given by 
	\begin{align*}
		&\dbtilde {\bf \Pi}_{ij} := \frac{\overline{\mathbf{U}}_{ij} - \frac{\widehat{\omega}_1}{\lambda}  \left[ \lambda_1 \left( \dbtilde {\bf \Pi}_{i-\halfone,j}^{+} + \dbtilde {\bf \Pi}_{i+\halfone,j}^{-} \right) + \lambda_2 \left( \dbtilde {\bf \Pi}_{i,j-\halfone}^{+} + \dbtilde {\bf \Pi}_{i,j+\halfone}^{-} \right) \right]}{1-2\widehat{\omega}_1},
		\\
			&	\dbtilde {\bf \Pi}_{i\pm\halfone,j}^{\mp} := \sum_{\mu=1}^{Q} \omega_{\mu} \dbtilde{\mathbf{U}}_{i\pm\halfone,j}^{\mp,\mu}, \qquad
		\dbtilde {\bf \Pi}_{i,j\pm\halfone}^{\mp} := \sum_{\mu=1}^{Q} \omega_{\mu} \dbtilde{\mathbf{U}}_{i,j\pm\halfone}^{\mu,\mp}.
	\end{align*}

The above PP limiter is motivated by the simplified maximum-principle-preserving limiter in \cite{zhang2011b}. 
As demonstrated by \Cref{Ex:Vortex}, such a PP limiter does not destroy the high-order accuracy; see also \cite{zhang2010,zhang2011b,ZHANG2017301} for some theoretical justification. 

\begin{theorem}
Assume that $\overline {\bf U}_{ij }\in G$ for all $i$ and $j$.	Then the PP limited point values $\left\{ {\mathbf{U}}_{i+\halfone,j}^{\pm,\mu}, {\mathbf{U}}_{i,j+\halfone}^{\mu,\pm} \right\}$ defined by \eqref{eq:PPlimited}  maintain the DDF condition \eqref{eq:DDF} 
and simultaneously satisfy the PP conditions \eqref{LF:cond1}. 
\end{theorem}

\begin{proof}
	Thanks to the linearity of the discrete divergence operator $\nabla_h \cdot$, we have 
	\begin{equation*}
		\nabla_h \cdot {\bf B}_{ij} = \theta_2 \nabla_h \cdot \dbtilde{\bf B}_{ij} + (1-\theta_2) \nabla_h \cdot \overline {\bf B}_{ij} 
		=  \theta_2 \nabla_h \cdot \dbtilde{\bf B}_{ij} = \theta_2 \nabla_h \cdot \widetilde{\bf B}_{ij} = 0. 
	\end{equation*}
Hence the DDF condition \eqref{eq:DDF} is preserved by the PP limiter. In the following, we will verify the conditions \eqref{LF:cond1}. 
Note that 
\begin{align*}
 \dbtilde{\rho}_{i\pm\halfone,j}^{\mp,\mu} 
	&= \theta_1  \left(\widetilde{\rho}_{i\pm\halfone,j}^{\mp,\mu} - \overline{\rho}_{ij}\right) + \overline{\rho}_{ij} 
	 \ge \theta_1  \left(\rho_{\min} - \overline{\rho}_{ij}\right) + \overline{\rho}_{ij}
	 \\ & = -\theta_1 \left| \rho_{\min} - \overline{\rho}_{ij} \right| + \left( \overline{\rho}_{ij} - \epsilon_1 \right) + \epsilon_1 
\ge \epsilon_1 > 0, 
\end{align*}
which implies 
$
	\rho_{i\pm\halfone,j}^{\mp,\mu} =\theta_2\dbtilde{\rho}_{i\pm\halfone,j}^{\mp,\mu}    + (1-\theta_2) \overline{\rho}_{ij} > 0. 
$ 
	Because the function ${\mathcal E}({\bf U})$ is concave with respect to ${\bf U}$ when $\rho > 0$, applying Jensen's inequality gives 
	\begin{align*}
		{\mathcal E}( {\mathbf{U}}_{i\pm\halfone,j}^{\mp,\mu} )  & = 
		{\mathcal E} \Big( \theta_2 \dbtilde{\mathbf{U}}_{i\pm\halfone,j}^{\mp,\mu} + (1-\theta_2) \overline{\mathbf{U}}_{ij}  \Big) 
		 \ge \theta_2 {\mathcal E} \Big(  \dbtilde{\mathbf{U}}_{i\pm\halfone,j}^{\mp,\mu} \Big) 
		+(1-\theta_2) {\mathcal E} \Big(  \overline{\mathbf{U}}_{ij}  \Big) 
		\\
		& \ge \theta_2 {\mathcal E}_{\min} + (1-\theta_2) {\mathcal E} \Big(  \overline{\mathbf{U}}_{ij}  \Big) =  \theta_2 \left( {\mathcal E}_{\min} - {\mathcal E} (  \overline{\mathbf{U}}_{ij}  ) \right) +  {\mathcal E} (  \overline{\mathbf{U}}_{ij}  )
		\\
		& \ge - \theta_2 \left| {\mathcal E}_{\min} - {\mathcal E} (  \overline{\mathbf{U}}_{ij}  ) \right| + ({\mathcal E} (  \overline{\mathbf{U}}_{ij}  ) - \epsilon_2 ) + \epsilon_2 
		 \ge \epsilon_2 >0. 
	\end{align*}
	Therefore, we have ${\mathbf{U}}_{i\pm\halfone,j}^{\mp,\mu} \in G$. Similar arguments give ${\mathbf{U}}_{i,j+\halfone}^{\mu,\pm}\in G$. For $k\ge 3$, we observe that 
	\begin{equation*}
		{\bf \Pi}_{i\pm\halfone,j}^{\mp} = \theta_2 \Big( \dbtilde {\bf \Pi}_{i\pm\halfone,j}^{\mp} -  \overline{\mathbf{U}}_{ij} \Big) + \overline{\mathbf{U}}_{ij}, 
		\qquad   {\bf \Pi}_{i,j\pm\halfone}^{\mp} = \theta_2 
		\Big( \dbtilde {\bf \Pi}_{i,j\pm\halfone}^{\mp} -  \overline{\mathbf{U}}_{ij} \Big) + \overline{\mathbf{U}}_{ij}, 
	\end{equation*}
which yield 
\begin{equation}\label{key5433}
	{\bf \Pi}_{ij} = \theta_2 \Big( \dbtilde{\bf \Pi}_{ij} -  \overline{\mathbf{U}}_{ij} \Big) + \overline{\mathbf{U}}_{ij} = \theta_2 \dbtilde{\bf \Pi}_{ij} + (1-\theta_2) \overline{\mathbf{U}}_{ij}.
\end{equation}
Let ${\bf n}_1 = (1,0,0,0,0,0,0,0)^\top$. Similarly, we have 
\begin{align*}
\dbtilde	{\bf \Pi}_{ij} \cdot {\bf n}_1 = \theta_1 \Big( \widetilde{\bf \Pi}_{ij} \cdot {\bf n}_1 -  \overline{\mathbf{U}}_{ij} \cdot {\bf n}_1 \Big) + \overline{\mathbf{U}}_{ij} \cdot {\bf n}_1 = \theta_1 \Big( \rho_{ij}^{\star} -  \overline{\rho}_{ij}   \Big) + \overline{\rho}_{ij}  > 0,
\end{align*}
which together with \eqref{key5433} implies 
$
{\bf \Pi}_{ij} \cdot {\bf n}_1 = \theta_2 \dbtilde{\bf \Pi}_{ij} \cdot {\bf n}_1 + (1-\theta_2) \overline{\rho}_{ij} > 0.
$ 
Applying Jensen's inequality to ${\mathcal E}({\bf U})$, we obtain 
\begin{align*}
	{\mathcal E} ({\bf \Pi}_{ij}) &= {\mathcal E} \Big(  \theta_2 \dbtilde{\bf \Pi}_{ij} + (1-\theta_2) \overline{\mathbf{U}}_{ij} \Big) 
	\ge  \theta_2{\mathcal E} (  \dbtilde{\bf \Pi}_{ij}  ) + (1-\theta_2){\mathcal E} (   \overline{\mathbf{U}}_{ij} ) 
	\\
	& \ge \theta_2 {\mathcal E}_{\min} + (1-\theta_2) {\mathcal E} (  \overline{\mathbf{U}}_{ij}  ) 
 \ge \epsilon_2 >0. 
\end{align*}
Therefore, we have ${\bf \Pi}_{ij} \in G$ when the order of accuracy $k\ge 3$. In summary, the PP conditions \eqref{LF:cond1} hold. The proof is completed. 
\end{proof}


\section{Rigorous analysis of positivity-preserving property} \label{PCP:Sec}
In this section, we rigorously analyze the PP property of the proposed finite volume method and give the proof of \Cref{thm:main}.  
The analysis and proof are very technical and difficult. 
There are two main challenges. First, the point values $\{ {\mathbf{U}}_{i+\halfone,j}^{\pm,\mu}, {\mathbf{U}}_{i,j+\halfone}^{\mu,\pm}\}$ 
in our finite volume method are strongly coupled by the 
important DDF condition \eqref{eq:DDF}, this makes the PP analysis very complicated. In particular, such a strong coupling invalidates the classic convex decomposition technique \cite{zhang2010,zhang2010b}, which is based on splitting a high-order multidimensional scheme into convex combination of formally  fist-order one-dimensional PP schemes.  
The second challenge lies in the high nonlinearity of the MHD system: particularly, the flux  ${\bf F}_i({\bf U})$ and the internal energy $\mathcal{E}(\mathbf{U})$ are highly nonlinear functions of the conservative vector ${\bf U}$. In order to analyze the positivity of $\mathcal{E}(\mathbf{U})$, one has to substitute the high-order scheme into the complicated function  $\mathcal{E}(\mathbf{U})$ to judge the positivity of the computed internal energy $\mathcal{E}$.      
In the following, we will introduce the GQL approach \cite{WuShu2021GQL} to overcome these difficulties. 


\subsection{Geometric quasilinearization} 
We first give a brief review of the GQL approach, which {\em equivalently} transfers the intractable {\em nonlinear} constraint  $\mathcal{E}(\mathbf{U})>0$ in \eqref{eq:G} into a family of {\em linear} constraints. More specifically, the GQL approach provides the following  
equivalent linear representation of the admissible set $G$, which will be the foundation stone of our PP analysis and proof of \Cref{thm:main}.

\begin{lemma}[GQL representation \cite{Wu2017a}] \label{Lemma:Gstar}
	The admissible state set $G$ defined in \eqref{eq:G} is equivalent to 
	\begin{equation}\label{eq:Gs}
		G_* = \left\{
		\mathbf{U}=(\rho, \mathbf{m}, \mathbf{B}, E)^\top:~ \mathbf{U}\cdot \mathbf{n}_1>0,~ \mathbf{U}\cdot \mathbf{n}^{*} + \frac{| \mathbf{B}^* |^2}{2} >0~  \forall \mathbf{v}^*, \mathbf{B}^*\in \mathbb{R}^3
		\right\},
	\end{equation}	
	where $\mathbf{n}_1 =(1,0,0,0,0,0,0,0)^\top$,  $\mathbf{n}^{*}=\left(\frac{|\mathbf{v}^*|^2}{2}, -\mathbf{v}^*, -\mathbf{B}^*, 1\right)^{\top}$, and the variables $\{\mathbf{v}^*, \mathbf{B}^*\}$ are called the extra free auxiliary variables in the GQL framework \cite{WuShu2021GQL} which are introduced  in exchange for linearity.
\end{lemma}

The first proof of \Cref{Lemma:Gstar} can be found in \cite{Wu2017a}, while its geometric interpretation was later given in \cite{WuShu2021GQL}. It is worth noting that, 
although a few extra variables $\{\mathbf{v}^*, \mathbf{B}^*\}$ are introduced in \eqref{eq:Gs}, all the constraints in $G_*$ become linear, which greatly facilitates the PP analysis.  
In fact, 
such a linear equivalent form (called GQL representation) can be constructed for any convex sets, as shown in \cite{WuShu2021GQL}.

 

\begin{lemma}[see \cite{Wu2017a}] \label{Lemma:LF}
	Given any two admissible states $\mathbf{U}, \widetilde{\mathbf{U}} \in G$, we have 
	\begin{equation}\label{eq:keyieq}
		\left(\mathbf{U} -\frac{\mathbf{F}_\ell(\mathbf{U})}{\alpha} + \widetilde{\mathbf{U}}+\frac{\mathbf{F}_\ell(\widetilde{\mathbf{U}})}{\alpha} \right) \cdot \mathbf{n}^* + {| \mathbf{B}^* |}^2 + \frac{B_\ell-\widetilde{B}_\ell}{\alpha}(\mathbf{v}^* \cdot \mathbf{B}^*) >0,
	\end{equation} 
	 for any free auxiliary variables $\mathbf{v}^*, \mathbf{B}^*\in \mathbb{R}^3$ and $\alpha>\alpha_\ell(\mathbf{U}, \widetilde{\mathbf{U}})$, where $\alpha_\ell$ is defined in \Cref{thm:main} and $\ell \in \{1,2,3\}$.
\end{lemma}

As a direct consequence of \Cref{Lemma:LF}, we have the following estimates. 

\begin{corollary}
	For any $\mathbf{U}, \widetilde{\mathbf{U}} \in G$,  $\mathbf{v}^* , \mathbf{B}^* \in \mathbb{R}^3$, and $\ell\in \{1,2,3\}$, it holds that   
	\begin{equation}
		\resizebox{0.58\hsize}{!}{$
			\begin{aligned}
				\label{LF:pf-eq1}
				-\left( \mathbf{F}_\ell (\mathbf{U})-\mathbf{F}_\ell(\widetilde{\mathbf{U}}) \right) \cdot \mathbf{n}_1  > -\alpha_\ell(\mathbf{U}, \widetilde{\mathbf{U}}) (\mathbf{U}+ \widetilde{\mathbf{U}}) \cdot \mathbf{n}_1, 
			\end{aligned}
			$}
	\end{equation}
	\begin{equation}
		\resizebox{0.92\hsize}{!}{$
			\begin{aligned}
				\label{LF:pf-eq2}
				-\left( \mathbf{F}_\ell(\mathbf{U})-\mathbf{F}_\ell(\widetilde{\mathbf{U}}) \right) \cdot \mathbf{n}^* \ge -\alpha_\ell(\mathbf{U}, \widetilde{\mathbf{U}}) \left( (\mathbf{U}+ \widetilde{\mathbf{U}}) \cdot \mathbf{n}^* + |\mathbf{B}^*|^2 \right) - (B_\ell-\tilde{B}_\ell) (\mathbf{v}^* \cdot \mathbf{B}^*).
			\end{aligned}
			$}
	\end{equation}
\end{corollary}

\begin{proof}
	The inequality \eqref{LF:pf-eq1} can be derived as follows: 
	$$
	-\left( \mathbf{F}_\ell (\mathbf{U})-\mathbf{F}_\ell(\widetilde{\mathbf{U}}) \right) \cdot \mathbf{n}_1 = \widetilde{\rho} \widetilde{v}_\ell - \rho v_\ell 
	\ge  - \max \{ |v_\ell| ,  \left|\widetilde{v}_\ell \right| \} ( \rho + \widetilde{\rho} ) > -\alpha_\ell(\mathbf{U}, \widetilde{\mathbf{U}}) (\mathbf{U}+ \widetilde{\mathbf{U}}) \cdot \mathbf{n}_1.
	$$
	The inequality \eqref{LF:pf-eq2} follows from \eqref{eq:keyieq} by reformulation and taking $\alpha \to \alpha_\ell(\mathbf{U}, \widetilde{\mathbf{U}})$. 
\end{proof}

The following lemma will be useful in studying the effect of approximate source term \eqref{Dis:LFSourT2} on the PP property. 

\begin{lemma}[see \cite{WuShu2018}] \label{Lemma:SourT}
	For any $\mathbf{U} \in G$, any $\mathbf{v}^* , \mathbf{B}^* \in \mathbb{R}^3$, and $\xi \in \mathbb{R}$, we have 
	\begin{equation}
		-\xi \left( \mathbf{S}(\mathbf{U})\cdot \mathbf{n}^* \right) \geq \xi (\mathbf{v}^* \cdot \mathbf{B}^*) - \frac{|\xi|}{\sqrt{\rho}} \left( \mathbf{U} \cdot \mathbf{n}^* + \frac{|\mathbf{B}^*|^2}{2} \right).
	\end{equation}
\end{lemma}

\subsection{Proof of \Cref{thm:main}}

We are ready to rigorously prove \Cref{thm:main}.

\begin{proof}
	Under conditions \eqref{LF:cond1}, the GQL representation in \Cref{Lemma:Gstar} implies 
	\begin{align}\label{eq:329}
		\mathbf{U}_{i+\halfone,j}^{\pm,\mu} \cdot {\bf n}_1 >0, \qquad 
		\mathbf{U}_{i+\halfone,j}^{\pm,\mu}  \cdot \mathbf{n}^* + \frac{|\mathbf{B}^*|^2}{2} >0,
		\\ \label{eq:330}
		\mathbf{U}_{i,j+\halfone}^{\mu,\pm} \cdot {\bf n}_1 >0, \qquad 
		\mathbf{U}_{i,j+\halfone}^{\mu,\pm}  \cdot \mathbf{n}^* + \frac{|\mathbf{B}^*|^2}{2} >0,
	\end{align}
	and when the accuracy order $k \ge 3$, we have $0<\widehat{\omega}_1 \le \frac16$ and  
\begin{equation}
	\label{eq:331}
	\overline{\mathbf{U}}_{ij}  \cdot \mathbf{n}_1 > \frac{\widehat{\omega}_1}{\lambda}  \left[ \lambda_1 \left( {\bf \Pi}_{i-\halfone,j}^{+} + {\bf \Pi}_{i+\halfone,j}^{-} \right) + \lambda_2 \left( {\bf \Pi}_{i,j-\halfone}^{+} + {\bf \Pi}_{i,j+\halfone}^{-} \right) \right] \cdot \mathbf{n}_1,
\end{equation}
\begin{equation}
	\resizebox{0.92\hsize}{!}{$
		\begin{aligned}
			\label{eq:332}
			\overline{\mathbf{U}}_{ij}  \cdot \mathbf{n}^* + \frac{|\mathbf{B}^*|^2}{2} > \frac{\widehat{\omega}_1}{\lambda}  \left[ \lambda_1 \left( {\bf \Pi}_{i-\halfone,j}^{+} + {\bf \Pi}_{i+\halfone,j}^{-} \right) + \lambda_2 \left( {\bf \Pi}_{i,j-\halfone}^{+} + {\bf \Pi}_{i,j+\halfone}^{-} \right) \right] \cdot \mathbf{n}^* +\widehat{\omega}_1  |\mathbf{B}^*|^2.
		\end{aligned} 
		$}
\end{equation}
If $k=2$, then $\widehat{\omega}_1 = \frac12$ and 
{\small$
\overline{\mathbf{U}}_{ij} = \frac{\widehat{\omega}_1}{\lambda}  \left[ \lambda_1 \left( {\bf \Pi}_{i-\halfone,j}^{+} + {\bf \Pi}_{i+\halfone,j}^{-} \right) + \lambda_2 \left( {\bf \Pi}_{i,j-\halfone}^{+} + {\bf \Pi}_{i,j+\halfone}^{-} \right) \right],
$ }
which yields 
\begin{equation}
	\label{eq:333}
	\overline{\mathbf{U}}_{ij}  \cdot \mathbf{n}_1  = \frac{\widehat{\omega}_1}{\lambda}  \left[ \lambda_1 \left( {\bf \Pi}_{i-\halfone,j}^{+} + {\bf \Pi}_{i+\halfone,j}^{-} \right) + \lambda_2 \left( {\bf \Pi}_{i,j-\halfone}^{+} + {\bf \Pi}_{i,j+\halfone}^{-} \right) \right] \cdot \mathbf{n}_1,
\end{equation}
\begin{equation}
	\resizebox{0.92\hsize}{!}{$
		\begin{aligned}
			\label{eq:334}
			\overline{\mathbf{U}}_{ij}  \cdot \mathbf{n}^* + \frac{|\mathbf{B}^*|^2}{2} &= \frac{\widehat{\omega}_1}{\lambda}  \left[ \lambda_1 \left( {\bf \Pi}_{i-\halfone,j}^{+} + {\bf \Pi}_{i+\halfone,j}^{-} \right) + \lambda_2 \left( {\bf \Pi}_{i,j-\halfone}^{+} + {\bf \Pi}_{i,j+\halfone}^{-} \right) \right] \cdot \mathbf{n}^* +\widehat{\omega}_1  |\mathbf{B}^*|^2. 
		\end{aligned} 
		$}
\end{equation}
Combining \eqref{eq:331}--\eqref{eq:332} for $k\ge 3$ with \eqref{eq:333}--\eqref{eq:334} for $k=2$, we obtain for any $k\ge 2$ that 
\begin{equation}
	\label{eq:335}
	\overline{\mathbf{U}}_{ij}  \cdot \mathbf{n}_1  \ge  \frac{\widehat{\omega}_1}{\lambda}  \left[ \lambda_1 \left( {\bf \Pi}_{i-\halfone,j}^{+} + {\bf \Pi}_{i+\halfone,j}^{-} \right) + \lambda_2 \left( {\bf \Pi}_{i,j-\halfone}^{+} + {\bf \Pi}_{i,j+\halfone}^{-} \right) \right] \cdot \mathbf{n}_1,
\end{equation}
\begin{equation}
	\resizebox{0.92\hsize}{!}{$
		\begin{aligned}
			\label{eq:336}
			\overline{\mathbf{U}}_{ij}  \cdot \mathbf{n}^* + \frac{|\mathbf{B}^*|^2}{2}  \ge \frac{\widehat{\omega}_1}{\lambda}  \left[ \lambda_1 \left( {\bf \Pi}_{i-\halfone,j}^{+} + {\bf \Pi}_{i+\halfone,j}^{-} \right) + \lambda_2 \left( {\bf \Pi}_{i,j-\halfone}^{+} + {\bf \Pi}_{i,j+\halfone}^{-} \right) \right] \cdot \mathbf{n}^* +\widehat{\omega}_1  |\mathbf{B}^*|^2
		\end{aligned} 
		$}
\end{equation}
for any free auxiliary variables $\mathbf{v}^* , \mathbf{B}^* \in \mathbb{R}^3$. 
From \eqref{FV:InterfacePi} and \eqref{eq:329}--\eqref{eq:330}, we observe that 
	\begin{equation}\label{LF:cond1_1}
		{\bf \Pi}_{i + \halfone,j}^{\pm} \cdot \mathbf{n}_1 >0, \quad   {\bf \Pi}_{i,j + \halfone}^{\pm}  \cdot \mathbf{n}_1 > 0 \quad \forall i,j,
	\end{equation}
	\begin{equation}\label{LF:cond1_3}
		{\bf \Pi}_{i + \halfone,j}^{\pm} \cdot \mathbf{n}^* + \frac{|\mathbf{B}^*|^2}{2} >0, \quad   {\bf \Pi}_{i,j + \halfone}^{\pm}  \cdot \mathbf{n}^* + \frac{|\mathbf{B}^*|^2}{2} > 0 \quad \forall \mathbf{v}^* , \mathbf{B}^* \in \mathbb{R}^3\quad \forall i,j. 
	\end{equation}
	The estimates \eqref{eq:335}--\eqref{LF:cond1_3} will be useful in showing the PP property \eqref{LF:Update} of the updated cell average $\overline{\mathbf{U}}_{ij}^{\dt} := 
	\overline{\mathbf{U}}_{ij} + \dt \mathcal{L}_{ij}(\overline{\mathbf{U}})$. 
	Substituting the numerical fluxes \eqref{FV:Flux} with \eqref{eq:LFFlux} into \eqref{eq:defL}, we can reformulate $\overline{\mathbf{U}}_{ij}^{\dt} $ as 
	\begin{equation}\label{LF:UpdateRe}
		\begin{aligned}
			\overline{\mathbf{U}}_{ij}^{\dt}
			= &\overline{\mathbf{U}}_{ij} + \frac{1}{2} \lambda_1 \left( {\bf \Pi}_{i+\halfone,j}^+ + {\bf \Pi}_{i-\halfone,j}^- - {\bf \Pi}_{i+\halfone,j}^- - {\bf \Pi}_{i-\halfone,j}^+ \right) \\
			&+ \frac{1}{2} \lambda_2 \left( {\bf \Pi}_{i,j+\halfone}^+ + {\bf \Pi}_{i,j-\halfone}^- - {\bf \Pi}_{i,j+\halfone}^- - {\bf \Pi}_{i,j-\halfone}^+ \right) + {\bf \Pi}_{\rm{{F}}} + \dt {\bf S}_{ij},
		\end{aligned}
	\end{equation}
	where $\lambda_1 = \frac{\alphaLF_1 \dt }{\dx}$, $\lambda_2 = \frac{\alphaLF_2 \dt }{\dx}$, and
	\begin{equation*}
		\begin{aligned}
			{\bf \Pi}_{\rm{{F}}} := &-\frac{1}{2} \frac{\dt}{\dx} \sum_{\mu=1}^{Q} \omega_{\mu} \left[ \left( \mathbf{F}_1(\mathbf{U}_{i+\halfone,j}^{-,\mu}) - \mathbf{F}_1(\mathbf{U}_{i-\halfone,j}^{+,\mu})  \right)   +  \left( \mathbf{F}_1(\mathbf{U}_{i+\halfone,j}^{+,\mu}) - \mathbf{F}_1(\mathbf{U}_{i-\halfone,j}^{-,\mu})  \right)   \right] \\
			& -\frac{1}{2} \frac{\dt}{\dy} \sum_{\mu=1}^{Q} \omega_{\mu} \left[ \left( \mathbf{F}_2(\mathbf{U}_{i,j+\halfone}^{\mu,-}) - \mathbf{F}_2(\mathbf{U}_{i,j-\halfone}^{\mu,+})  \right)   +  \left( \mathbf{F}_2(\mathbf{U}_{i,j+\halfone}^{\mu,+}) - \mathbf{F}_2(\mathbf{U}_{i,j-\halfone}^{\mu,-})  \right)   \right].
		\end{aligned}
	\end{equation*}
	Let us first prove that $\overline{\mathbf{U}}_{ij}^{\dt} \cdot \mathbf{n}_1>0$.
	Thanks to the inequality \eqref{LF:pf-eq1}, we derive 
	\begin{align} \nonumber
			{\bf \Pi}_{\rm{{F}}} \cdot \mathbf{n}_1 > 
			& -\frac{1}{2} \frac{\alphaLF_1 \dt}{\dx} \sum_{\mu=1}^{Q} \omega_{\mu}  \left( \mathbf{U}_{i+\halfone,j}^{-,\mu} + \mathbf{U}_{i-\halfone,j}^{+,\mu}  +  \mathbf{U}_{i+\halfone,j}^{+,\mu} + \mathbf{U}_{i-\halfone,j}^{-,\mu}  \right)  \cdot \mathbf{n}_1 \\ \nonumber
			& -\frac{1}{2} \frac{\alphaLF_2 \dt}{\dy} \sum_{\mu=1}^{Q} \omega_{\mu} \left( \mathbf{U}_{i,j+\halfone}^{\mu,-} + \mathbf{U}_{i,j-\halfone}^{\mu,+} + \mathbf{U}_{i,j+\halfone}^{\mu,+} + \mathbf{U}_{i,j-\halfone}^{\mu,-}  \right) \cdot \mathbf{n}_1
		 \\  \nonumber
			= &-\frac{1}{2} \lambda_1 \left(   {\bf \Pi}_{i+\halfone,j}^- + {\bf \Pi}_{i-\halfone,j}^+ + {\bf \Pi}_{i+\halfone,j}^+ + {\bf \Pi}_{i-\halfone,j}^- \right) \cdot \mathbf{n}_1 
			 \\ \label{eq:511}
			&- \frac{1}{2} \lambda_2 \left(  {\bf \Pi}_{i,j+\halfone}^- + {\bf \Pi}_{i,j-\halfone}^+ + {\bf \Pi}_{i,j+\halfone}^+ + {\bf \Pi}_{i,j-\halfone}^-  \right)  \cdot \mathbf{n}_1 .
	\end{align}
	Note that the first component of ${\bf S}({\bf U})$ is zero, yielding  
	${\bf S}_{ij} \cdot \mathbf{n}_1 =0$. 
	Combining it with \eqref{LF:UpdateRe}--\eqref{eq:511} gives 
	\begin{align}
		\overline{\mathbf{U}}_{ij}^{\dt} \cdot \mathbf{n}_1 > \overline{\mathbf{U}}_{ij} \cdot {\bf n}_1 
		- \lambda_1 \left(  {\bf \Pi}_{i+\halfone,j}^- + {\bf \Pi}_{i-\halfone,j}^+ \right) \cdot {\bf n}_1
		- \lambda_2 \left(  {\bf \Pi}_{i,j+\halfone}^- + {\bf \Pi}_{i,j-\halfone}^+ \right) \cdot {\bf n}_1. 
	\end{align}
	This together with \eqref{eq:335} implies 
	\begin{equation*}
			\begin{aligned}
				\overline{\mathbf{U}}_{ij}^{\dt} \cdot \mathbf{n}_1
				>  
				(\widehat{\omega}_1 - \lambda) \left[\frac{\lambda_1}{\lambda}  \left({\bf \Pi}_{i-\halfone,j}^{+} + {\bf \Pi}_{i+\halfone,j}^{-} \right) \cdot \mathbf{n}_1 
				+ \frac{\lambda_2}{\lambda} \left({\bf \Pi}_{i,j-\halfone}^{+} + {\bf \Pi}_{i,j+\halfone}^{-} \right) \cdot \mathbf{n}_1 \right] > 0,
			\end{aligned}
	\end{equation*}
	where the last inequality follows from \eqref{LF:cond1_1} and the CFL condition \eqref{LF:CFLCond}. 
	Next, let us prove $\overline{\mathbf{U}}_{ij}^{\dt} \cdot \mathbf{n}^* + \frac{|\mathbf{B}^*|^2}{2}>0$ for any 
	 free auxiliary variables $\mathbf{v}^* , \mathbf{B}^* \in \mathbb{R}^3$. 
	 It follows from \eqref{LF:UpdateRe} that 
	 \begin{equation}\label{LF:UpdateRe2}
	 	\resizebox{0.92\hsize}{!}{$
	 		\begin{aligned}
	 			\overline{\mathbf{U}}_{ij}^{\dt} \cdot \mathbf{n}^* &+ \frac{|\mathbf{B}^*|^2}{2}
	 			= \overline{\mathbf{U}}_{ij} \cdot \mathbf{n}^* + \frac{|\mathbf{B}^*|^2}{2} + \frac{1}{2} \lambda_1 \left( {\bf \Pi}_{i+\halfone,j}^+ + {\bf \Pi}_{i-\halfone,j}^- - {\bf \Pi}_{i+\halfone,j}^- - {\bf \Pi}_{i-\halfone,j}^+ \right) \cdot \mathbf{n}^* \\
	 			&+ \frac{1}{2} \lambda_2 \left( {\bf \Pi}_{i,j+\halfone}^+ + {\bf \Pi}_{i,j-\halfone}^- - {\bf \Pi}_{i,j+\halfone}^- - {\bf \Pi}_{i,j-\halfone}^+ \right) \cdot \mathbf{n}^* + {\bf \Pi}_{\rm{{F}}}\cdot \mathbf{n}^* + \dt {\bf S}_{ij} \cdot \mathbf{n}^*.
	 		\end{aligned}
	 		$}
	 \end{equation}
 	 We now estimate the lower bounds for ${\bf \Pi}_{\rm{{F}}}\cdot \mathbf{n}^*$ and $\dt {\bf S}_{ij} \cdot \mathbf{n}^*$, respectively. 
	 Thanks to the inequality \eqref{LF:pf-eq2}, we have 
	 \begin{align} \nonumber
	 	{\bf \Pi}_{\rm{{F}}} \cdot  \mathbf{n}^*    \ge 
			& -\frac{1}{2} \frac{\hat{\alpha}_1^{\rm{LF}} \dt}{\dx} \sum_{\mu=1}^{Q} \omega_{\mu}  \left[ \left( \mathbf{U}_{i+\halfone,j}^{-,\mu} + \mathbf{U}_{i-\halfone,j}^{+,\mu}  +  \mathbf{U}_{i+\halfone,j}^{+,\mu} + \mathbf{U}_{i-\halfone,j}^{-,\mu}  \right)  \cdot \mathbf{n}^*  + 2 |\mathbf{B}^*|^2 \right]
			\\ \nonumber
			&-\frac{1}{2} \frac{ \dt}{\dx} \sum_{\mu=1}^{Q} \omega_{\mu}  \left[ (B_1)_{i+\halfone,j}^{-,\mu} -  (B_1)_{i-\halfone,j}^{+,\mu} + (B_1)_{i+\halfone,j}^{+,\mu} - (B_1)_{i-\halfone,j}^{-,\mu}  \right] ({\bf v}^* \cdot {\bf B}^*)
			\\ \nonumber
			& -\frac{1}{2} \frac{\hat{\alpha}_2^{\rm{LF}} \dt}{\dy} \sum_{\mu=1}^{Q} \omega_{\mu}  \left[ \left( \mathbf{U}_{i,j+\halfone}^{\mu,-} + \mathbf{U}_{i,j-\halfone}^{\mu,+} + \mathbf{U}_{i,j+\halfone}^{\mu,+} + \mathbf{U}_{i,j-\halfone}^{\mu,-}  \right) \cdot \mathbf{n}^* + 2 |\mathbf{B}^*|^2 \right] 
			\\ \nonumber
			& -\frac{1}{2} \frac{ \dt}{\dy} \sum_{\mu=1}^{Q} \omega_{\mu}  \left[ (B_2)_{i,j+\halfone}^{\mu,-} -  (B_2)_{i,j-\halfone}^{\mu,+} + (B_2)_{i,j+\halfone}^{\mu,+} - (B_2)_{i,j-\halfone}^{\mu,-}  \right] ({\bf v}^* \cdot {\bf B}^*) 
			\\ \nonumber
			= &-\frac{1}{2} \frac{\hat{\alpha}_1^{\rm{LF}} \dt}{\dx} \left[ \left( {\bf \Pi}_{i+\halfone,j}^+ + \Pi_{i-\halfone,j}^- + {\bf \Pi}_{i+\halfone,j}^- + {\bf \Pi}_{i-\halfone,j}^+ \right) \cdot \mathbf{n}^* + 2 |\mathbf{B}^*|^2 \right] 
			\\ \nonumber
			&- \frac{1}{2} \frac{\hat{\alpha}_2^{\rm{LF}} \dt}{\dx} \left[ \left({\bf \Pi}_{i,j+\halfone}^+ + {\bf \Pi}_{i,j-\halfone}^- + {\bf \Pi}_{i,j+\halfone}^- + {\bf \Pi}_{i,j-\halfone}^+ \right)  \cdot \mathbf{n}^* + 2 |\mathbf{B}^*|^2 \right] 
			\\ \label{eq:PiFn}
			& - \dt ( \nabla_h \cdot \average{\mathbf{B}}_{ij} ) (\mathbf{v}^* \cdot \mathbf{B}^*),
	\end{align}
	where the ``averaged'' discrete divergence operator $\nabla_h \cdot \average{\mathbf{B}}_{ij}$ is defined by
\begin{equation}\label{Dis:DivOpe}
	\resizebox{0.92\hsize}{!}{$
		\begin{aligned}
			\nabla_h \cdot \average{\mathbf{B}}_{ij} := \sum_{\mu=1}^{Q} \omega_{\mu} \left(\frac{\average{B_{1}}_{i+\halfone,j}^{\mu} -  \average{B_{1}}_{i-\halfone,j}^{\mu}}{\dx} +\frac{\average{B_{2}}_{i,j+\halfone}^{\mu} - \average{B_{2}}_{i,j-\halfone}^{\mu}}{\dy}  \right).
		\end{aligned}
		$}
\end{equation}
Thanks to Lemma \ref{Lemma:SourT}, we have the following estimate relative to the source term
\begin{equation*}
	\resizebox{1\hsize}{!}{$
		\begin{aligned}
			- \jump{{B}_{1}}_{i+\halfone,j}^{\mu} \mathbf{S} \left( \average{\mathbf{U}}_{i+\halfone,j}^{\mu} \right) \cdot \mathbf{n}^*  
			& \geq  \jump{{B}_{1}}_{i+\halfone,j}^{\mu} (\mathbf{v}^* \cdot \mathbf{B}^*) - \frac{\Big| \jump{B_{1}}_{i+\halfone,j}^{\mu} \Big|}{\sqrt{\average{\rho}_{i+\halfone,j}^{\mu}}} \left( \average{\mathbf{U}}_{i+\halfone,j}^{\mu} \cdot \mathbf{n}^* + \frac{|\mathbf{B}^*|^2}{2} \right) \\
			& \geq \jump{{B}_{1}}_{i+\halfone,j}^{\mu} (\mathbf{v}^* \cdot \mathbf{B}^*) - 2\beta_1 \left( \average{\mathbf{U}}_{i+\halfone,j}^{\mu} \cdot \mathbf{n}^* + \frac{|\mathbf{B}^*|^2}{2} \right).
		\end{aligned}
	$}
\end{equation*} 
	Similarly, one can obtain 
	\begin{equation*}
		\resizebox{1\hsize}{!}{$
		\begin{aligned}
			- \jump{B_{2}}_{i,j+\halfone}^{\mu}  \mathbf{S} \left( \average{\mathbf{U}}_{i,j+\halfone}^{\mu} \right) \cdot \mathbf{n}^*  
			\geq \jump{{B}_{2}}_{i,j+\halfone}^{\mu} (\mathbf{v}^* \cdot \mathbf{B}^*) - 2\beta_2 \left( \average{\mathbf{U}}_{i,j+\halfone}^{\mu} \cdot \mathbf{n}^* + \frac{|\mathbf{B}^*|^2}{2} \right), 
		\end{aligned}
		$}
	\end{equation*}
	Combining these estimates with \eqref{Dis:LFSourT2}, we obtain 
	\begin{align} \nonumber
			\dt {\bf S}_{ij} \cdot \mathbf{n}^*  \geq & 
			\frac{\dt}{\dx} \sum_{\mu=1}^Q \omega_\mu 
			\left( \frac12 \jump{{B}_{1}}_{i+\halfone,j}^{\mu} + \frac12 \jump{{B}_{1}}_{i-\halfone,j}^{\mu}  \right) (\mathbf{v}^* \cdot \mathbf{B}^*) \\ \nonumber
			& - \frac{\beta_1 \dt}{\dx} \sum_{\mu=1}^{Q} \left[ \left( \average{\mathbf{U}}_{i+\halfone,j}^{\mu} + \average{\mathbf{U}}_{i-\halfone,j}^{\mu} \right) \cdot \mathbf{n}^* + |\mathbf{B}^*|^2 \right] 
			\\ \nonumber
			& 
			+\frac{\dt}{ \dy} \sum_{\mu=1}^Q \omega_\mu 
			\left( \frac12 \jump{{B}_{2}}_{i,j+\halfone}^{\mu} + \frac12 \jump{{B}_{2}}_{i,j-\halfone}^{\mu}  \right) (\mathbf{v}^* \cdot \mathbf{B}^*)
			\\ \nonumber
			& - \frac{\beta_2 \dt}{\dy} \sum_{\mu=1}^{Q} \left[ \left( \average{\mathbf{U}}_{i,j+\halfone}^{\mu} + \average{\mathbf{U}}_{i,j-\halfone}^{\mu} \right) \cdot \mathbf{n}^* + |\mathbf{B}^*|^2 \right] 
			\\ \nonumber
			 = &   - \frac{\beta_1 \dt}{2\dx} \left[ \left( {\bf \Pi}_{i+\halfone,j}^+ + {\bf \Pi}_{i-\halfone,j}^- + {\bf \Pi}_{i+\halfone,j}^- + {\bf \Pi}_{i-\halfone,j}^+ \right) \cdot \mathbf{n}^* + 2 |\mathbf{B}^*|^2 \right] 
			\\  \nonumber
			&-   \frac{\beta_2 \dt}{2\dy} \left[ \left( {\bf \Pi}_{i,j+\halfone}^+ + {\bf \Pi}_{i,j-\halfone}^- + {\bf \Pi}_{i,j+\halfone}^- + {\bf \Pi}_{i,j-\halfone}^+ \right)  \cdot \mathbf{n}^* + 2 |\mathbf{B}^*|^2 \right] 
			\\ \label{LF:EstSourT}
			&+ \frac{\dt}{2}  \widehat{\jump{\mathbf{B}}}_{ij} (\mathbf{v}^* \cdot \mathbf{B}^*),  
	\end{align}
	where
$
		\widehat{\jump{\mathbf{B}}}_{ij} := 
		\sum_{\mu=1}^{Q} \omega_{\mu} \left( 
		\frac{ \jump{{B}_{1}}_{i+\halfone,j}^{\mu} + \jump{{B}_{1}}_{i-\halfone,j}^{\mu}}{\dx} 
		+ \frac{\jump{{B}_{2}}_{i,j+\halfone}^{\mu} + \jump{{B}_{2}}_{i,j-\halfone}^{\mu}}{\dy}
		\right).
$ 
	Substituting the estimates in \eqref{eq:336}, \eqref{eq:PiFn}, and \eqref{LF:EstSourT}  
	into 
	 \eqref{LF:UpdateRe2}, we obtain 
	\begin{align*}
			\overline{\mathbf{U}}_{ij}^{\dt}& \cdot \mathbf{n}^* + \frac{|\mathbf{B}^*|^2}{2}  
			\\
			 \ge  &
			  \frac{1}{2} \left( \lambda_1 - (\hat{\alpha}_1^{\rm{LF}} + \beta_1 )\frac{\dt}{\dx}  \right) \left[ \left( {\bf \Pi}_{i+\halfone,j}^+ + {\bf \Pi}_{i-\halfone,j}^- \right) \cdot \mathbf{n}^* + |\mathbf{B}^*|^2 \right] \\
			& + \frac{1}{2} \left( \lambda_2 - (\hat{\alpha}_2^{\rm{LF}} + \beta_2 ) \frac{\dt}{\dy}  \right) \left[ \left( {\bf \Pi}_{i,j+\halfone}^+ + {\bf \Pi}_{i,j-\halfone}^- \right) \cdot \mathbf{n}^* + |\mathbf{B}^*|^2 \right] \\
			& + \frac{1}{2} \left( \frac{2\widehat{\omega}_1 \lambda_1}{\lambda}  -  \left( \lambda_1 + (\hat{\alpha}_1^{\rm{LF}} + \beta_1 )\frac{\dt}{\dx} \right)   \right)  \left[ \left( {\bf \Pi}_{i+\halfone,j}^- + {\bf \Pi}_{i-\halfone,j}^+  \right) \cdot \mathbf{n}^* + |\mathbf{B}^*|^2 \right] \\
			& + \frac{1}{2} \left( \frac{2\widehat{\omega}_1 \lambda_2}{\lambda}  - \left( \lambda_2 + (\hat{\alpha}_2^{\rm{LF}} + \beta_2 )\frac{\dt}{\dy} \right)  \right)  \left[ \left( {\bf \Pi}_{i,j+\halfone}^- + {\bf \Pi}_{i,j-\halfone}^+  \right) \cdot \mathbf{n}^* + |\mathbf{B}^*|^2 \right] 
			\\
			& + \dt (\mathbf{v}^* \cdot \mathbf{B}^*) \left( \frac{1}{2} \widehat{\jump{\mathbf{B}}}_{ij} -  \nabla_h \cdot \average{\mathbf{B}}_{ij} \right)
			\\
			 \overset{\eqref{LF:CFLCond}}{>} &  \dt (\mathbf{v}^* \cdot \mathbf{B}^*) \left( \frac{1}{2} \widehat{\jump{\mathbf{B}}}_{ij} -  \nabla_h \cdot \average{\mathbf{B}}_{ij} \right) \\
			= & \dt (\mathbf{v}^* \cdot \mathbf{B}^*) \left( - \nabla_h \cdot \mathbf{B}_{ij} \right) 
			\overset{\eqref{eq:DDF}}{=} 0,
	\end{align*}
	where we have used the CFL condition \eqref{LF:CFLCond} which ensures that 
	$\lambda_1 \geq  (\hat{\alpha}_1^{\rm{LF}} + \beta_1 )\frac{\dt}{\dx}$,
	$\lambda_2 \geq (\hat{\alpha}_2^{\rm{LF}} + \beta_2 ) \frac{\dt}{\dx}$,
	$\frac{2\widehat{\omega}_1 \lambda_1}{\lambda}  > 2\lambda_1 \geq  \lambda_1 + (\hat{\alpha}_1^{\rm{LF}} + \beta_1 )\frac{\dt}{\dx}$, 
	and
	$\frac{2\widehat{\omega}_2 \lambda_2}{\lambda}  > 2\lambda_2 \geq  \lambda_2 + (\hat{\alpha}_2^{\rm{LF}} + \beta_2 )\frac{\dt}{\dx}$. 
	In summary, we have $\overline{\mathbf{U}}_{ij}^{\dt} \in G_* = G$. 
	The proof is completed.
\end{proof}


\section{Numerical experiments}\label{Numericalexperiments}
In this section, we give several benchmark or challenging numerical examples to validate the robustness, effectiveness, and accuracy of our DDFPP finite volume schemes. 
We focus on 
the second-order and fifth-order DDFPP schemes with the third-order SSP Runge–Kutta time discretization \eqref{RKTime}. 
In all our tests, we use the ideal EOS $e=p/(\rho(\gamma-1))$ and set the CFL number as $0.3$. 

\begin{expl}[Vortex problem with low pressure]\label{Ex:Vortex}
	We first consider a smooth vortex problem \cite{Christlieb,WuShu2018} to examine the accuracy of the proposed DDFPP schemes. The initial condition is chosen as $(\rho, \mathbf{v}, \mathbf{B}, p) = (1,1+\delta v_1, 1+\delta v_2, 0, \delta B_1, \delta B_2, 0, 1+\delta p)$ with the perturbations
	$
		(\delta v_1, \delta v_2) = \frac{\mu}{\sqrt{2\pi}} e^{0.5(1-r^2)}(-y,x)$,  
		$(\delta B_1, \delta B_2) = \frac{\mu}{2\pi} e^{0.5(1-r^2)}(-y,x)$,  
		$\delta p = -\frac{\mu^2(1+r^2)}{8 \pi^2} e^{1-r^2}$, 
	where  $r^2= {x}^2 + {y}^2$ and the vortex strength $\mu=5.389489439$. The lowest thermal pressure is about $5.3\times 10^{-12}$ in the vortex center, so that our PP limiting procedure is necessary for this example.  
	In our test, the computational domain is taken as $[-10,10]^2$, the boundary conditions are  periodic, and the adiabatic index is $\gamma=5/3$. We simulate this problem until $t=0.05$.	 
	\Cref{tab:Ex-vortex1} lists the $l^1$ errors and the corresponding convergence rates in the velocity, magnetic field, and pressure for our second-order and fifth-order DDFPP schemes at different grid resolutions. The results indicate that the expected convergence rates are achieved, verifying that our DDF projection and PP limiter do not affect the accuracy. 

		\begin{table}[htbp] 
		\centering
		\caption{ \Cref{Ex:Vortex}: $l^1$ errors at $t=0.05$ and the corresponding convergence rates for our second-order and fifth-order DDFPP schemes with  $ \dx = \dy = 20/N $. 
		}
		\label{tab:Ex-vortex1}
		\setlength{\tabcolsep}{2mm}{
			\begin{tabular}{cllclclc}
				\toprule[1.5pt]
				\multirow{2}{*}{ Method } &
				\multirow{2}{*}{$ N$} &
				\multicolumn{2}{c}{$v_2$} &
				\multicolumn{2}{c}{$B_1$} & 
				\multicolumn{2}{c}{$p$} \\
				\cmidrule(r){3-4} \cmidrule(r){5-6} \cmidrule(l){7-8}
				&
				& $ l^{1} $ error & order & $ l^{1} $ error & order & $ l^{1} $ error & order \\
				
				\midrule[1.5pt]
				
				\multirow{7}{*}{ 2nd-order }
				
				&20  & 1.79e{-}4&{--} & 1.19e{-}4&{--} & 1.60e{-}4&{--} \\
				&40  & 5.29e{-}5&1.75 & 3.62e{-}5&1.72 & 6.27e{-}5&1.35 \\
				&80  & 1.60e{-}5&1.72 & 1.06e{-}5&1.77 & 2.27e{-}5&1.46 \\
				&160 & 4.04e{-}6&1.99 & 2.82e{-}6&1.91 & 5.84e{-}6&1.96 \\
				&320 & 9.04e{-}7&2.16 & 6.32e{-}7&2.16 & 1.22e{-}6&2.26 \\
				&640 & 1.54e{-}7&2.55 & 1.07e{-}7&2.57 & 2.17e{-}7&2.49 \\
				&1280& 2.64e{-}8&2.54 & 1.77e{-}8&2.59 & 4.07e{-}8&2.42 \\

				\midrule[1.5pt]
				
				\multirow{7}{*}{ 5th-order }
				& 20  & 1.06e{-}3 &--   & 7.21e{-}4 &--   & 9.91e{-}4 &--   \\
				& 40  & 1.57e{-}4 &2.76 & 9.10e{-}5 &2.99 & 1.26e{-}4 &2.97 \\
				& 80  & 1.80e{-}5 &3.13 & 1.01e{-}5 &3.17 & 1.61e{-}5 &2.97 \\
				& 160 & 9.12e{-}7 &4.30 & 4.49e{-}7 &4.49 & 6.36e{-}7 &4.66 \\
				& 320 & 3.11e{-}8 &4.87 & 1.51e{-}8 &4.90 & 1.89e{-}8 &5.08 \\
				& 640 & 8.10e{-}10&5.26 & 4.61e{-}10&5.03 & 5.68e{-}10&5.05 \\
				& 1280& 2.50e{-}11&5.01 & 1.61e{-}11&4.84 & 2.04e{-}11&4.80 \\

				\bottomrule[1.5pt]
			\end{tabular}
		}
	\end{table}

\end{expl}

\begin{expl}[Orszag--Tang problem]\label{Ex:OT}
	This is a benchmark test \cite{jiang1999high,WuJiangShu2022}, and it is performed to verify the effectiveness, DDF property, and high resolution of our schemes. The initial state is taken as 
	\begin{equation*}
		(\rho, \mathbf{v}, \mathbf{B}, p) =  (\gamma^2, -\sin y, \sin x, 0, -\sin y, \sin (2x), 0, \gamma)
	\end{equation*}
	with $\gamma = 5/3$. The computational domain is $[0,2\pi]^2$ and divided into $200\times 200$ uniform cells with periodic boundary conditions.  
	 \Cref{fig:Ex-OT2} compares the results obtained by the fifth-order finite volume schemes with and without DDF projection, respectively. We observe that the solution computed without DDF projection has obvious oscillations and indicates very severe numerical instability.  This validates the necessity of the proposed DDF projection technique. 
	For further comparison, we also 
	evaluate the discrete divergence error
	$
	\varepsilon_{div} := \max_{i,j} \left| \nabla_h \cdot {\bf B}_{ij} \right| 
	$
	and present its evolution in \Cref{fig:Ex-OT-gdivB}. As we can see, the discrete divergence error grows very fast if the DDF projection is not used, while using the DDF projection can enforce the discrete divergence error $\varepsilon_{div}$ at the level of round-off error. 
	It is worth mentioning that if we do not use the DDF projection and PP limiter, 
	we observe that the fifth-order WENO code would break down at $t=3.11$ due to nonphysical solution. 
	This further confirms the importance of DDF projection and PP limiter for robust simulation.

%
%
%
%

	\begin{figure}[htbp]
	\centering
		\begin{subfigure}{0.32\textwidth}
				\includegraphics[width=\textwidth]{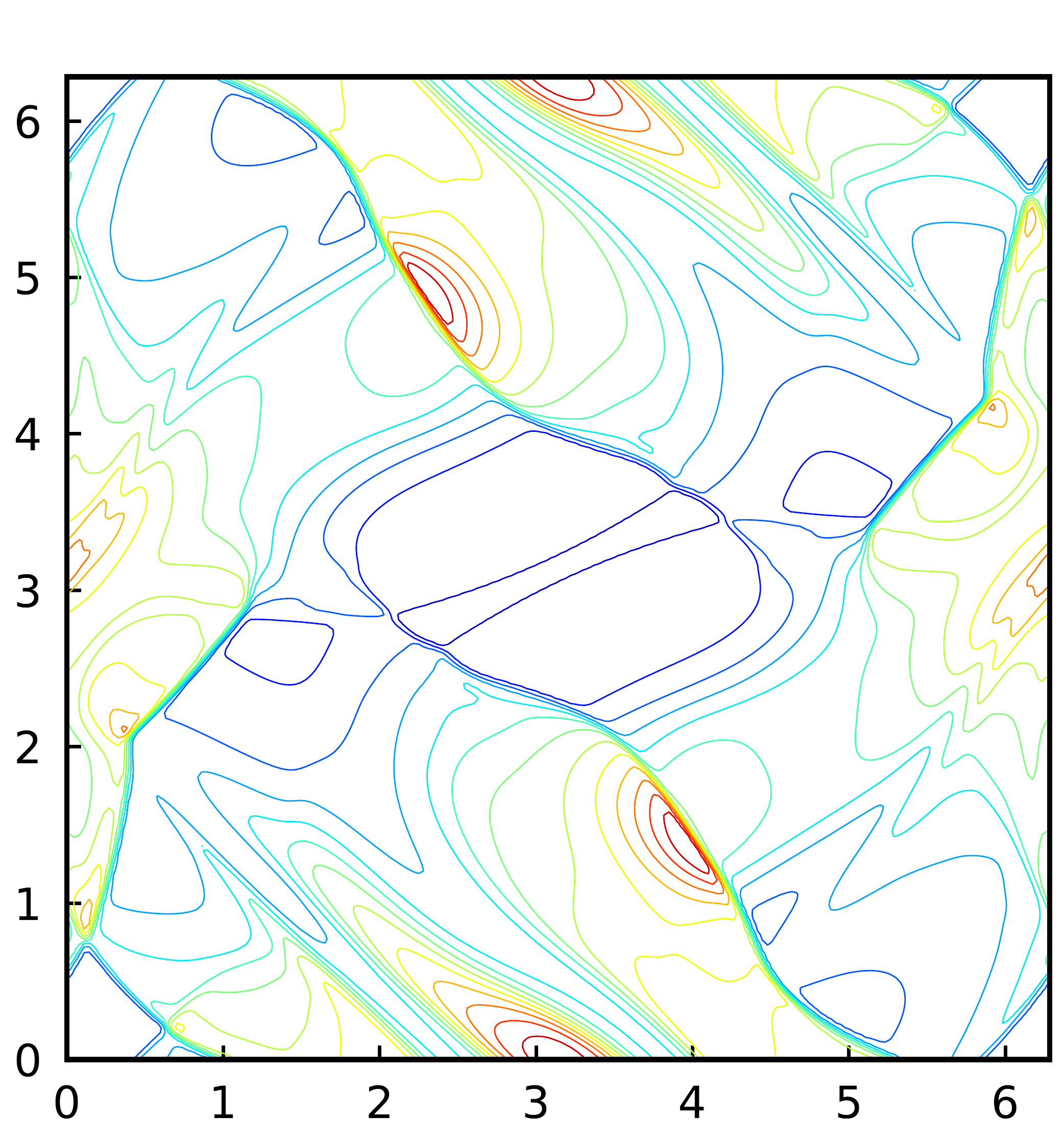}
				\caption{With DDF at $t=2$}
			\end{subfigure}
		\hfill	
	\begin{subfigure}{0.32\textwidth}
		\centering
		\includegraphics[width=\textwidth]{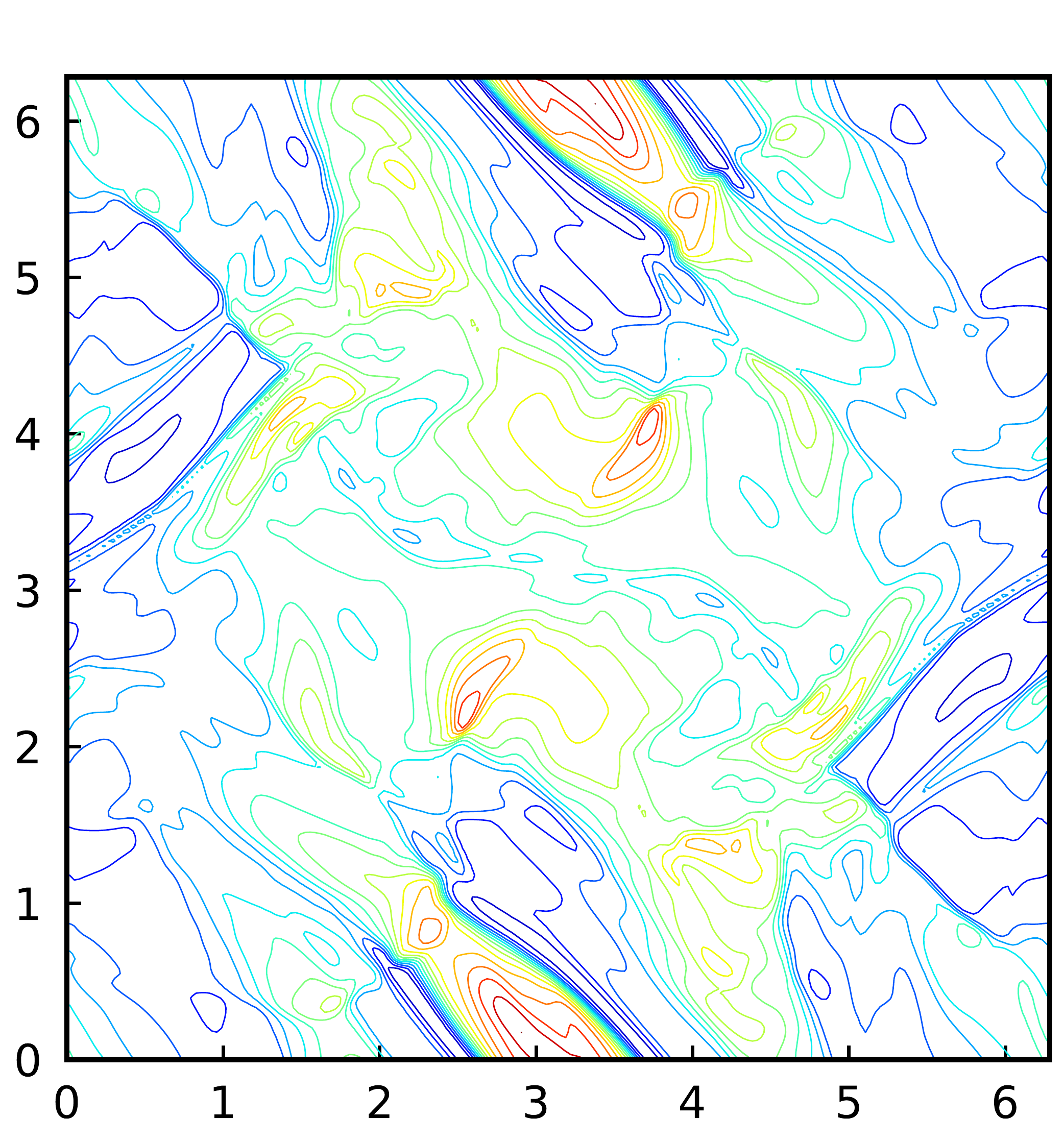}
					\caption{With DDF at $t=4$}
	\end{subfigure}
	\hfill
	\begin{subfigure}{0.32\textwidth}
		\centering
		\includegraphics[width=\textwidth]{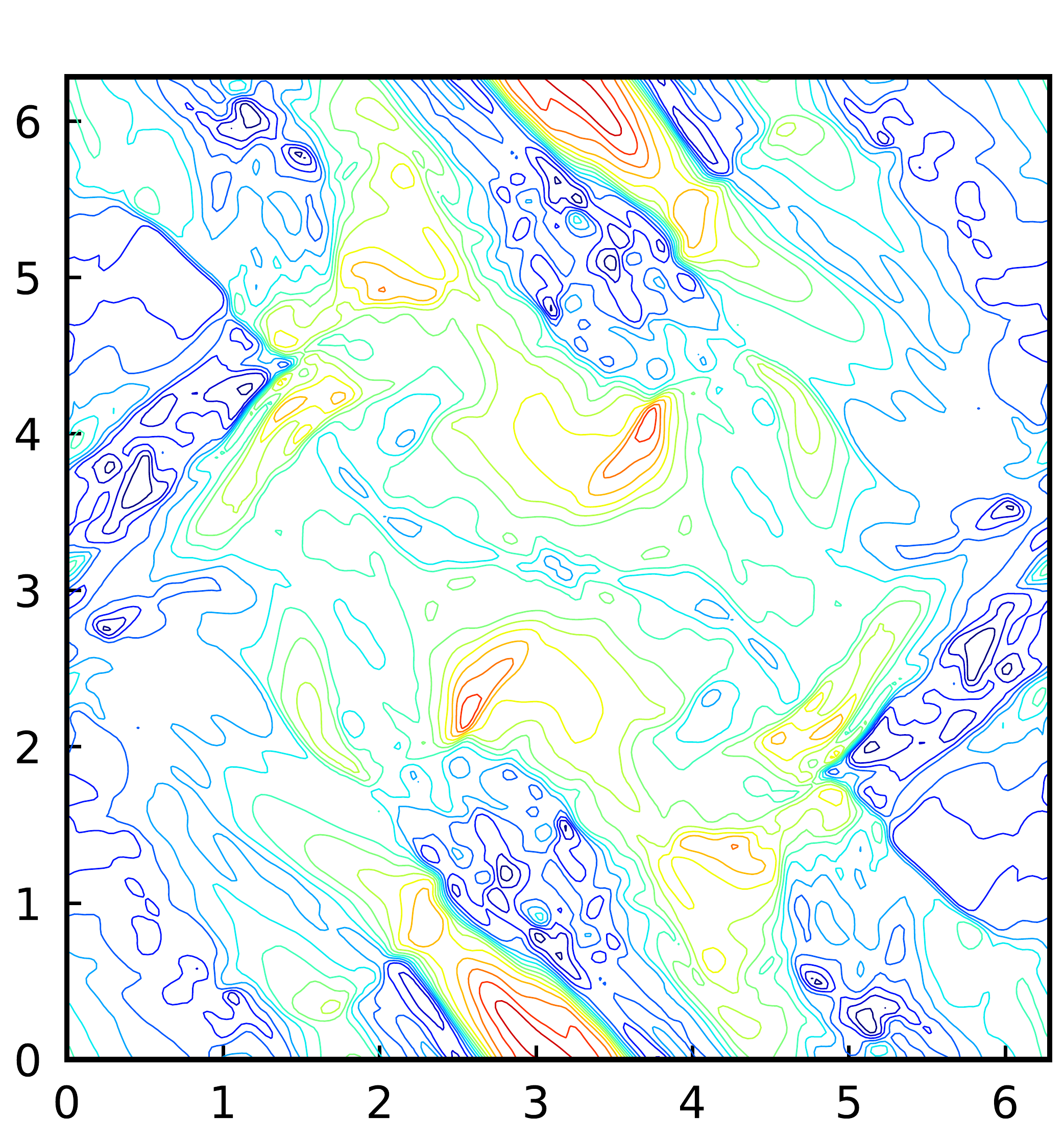}
					\caption{Without DDF at $t=4$}
	\end{subfigure}
	
	\caption{ \Cref{Ex:OT}: Density contours computed with or without DDF projection. 
	}
	\label{fig:Ex-OT2}
\end{figure}

\end{expl}

\begin{figure}[htbp]
	\centering
	\begin{subfigure}{0.32\textwidth}
		\includegraphics[width=\textwidth]{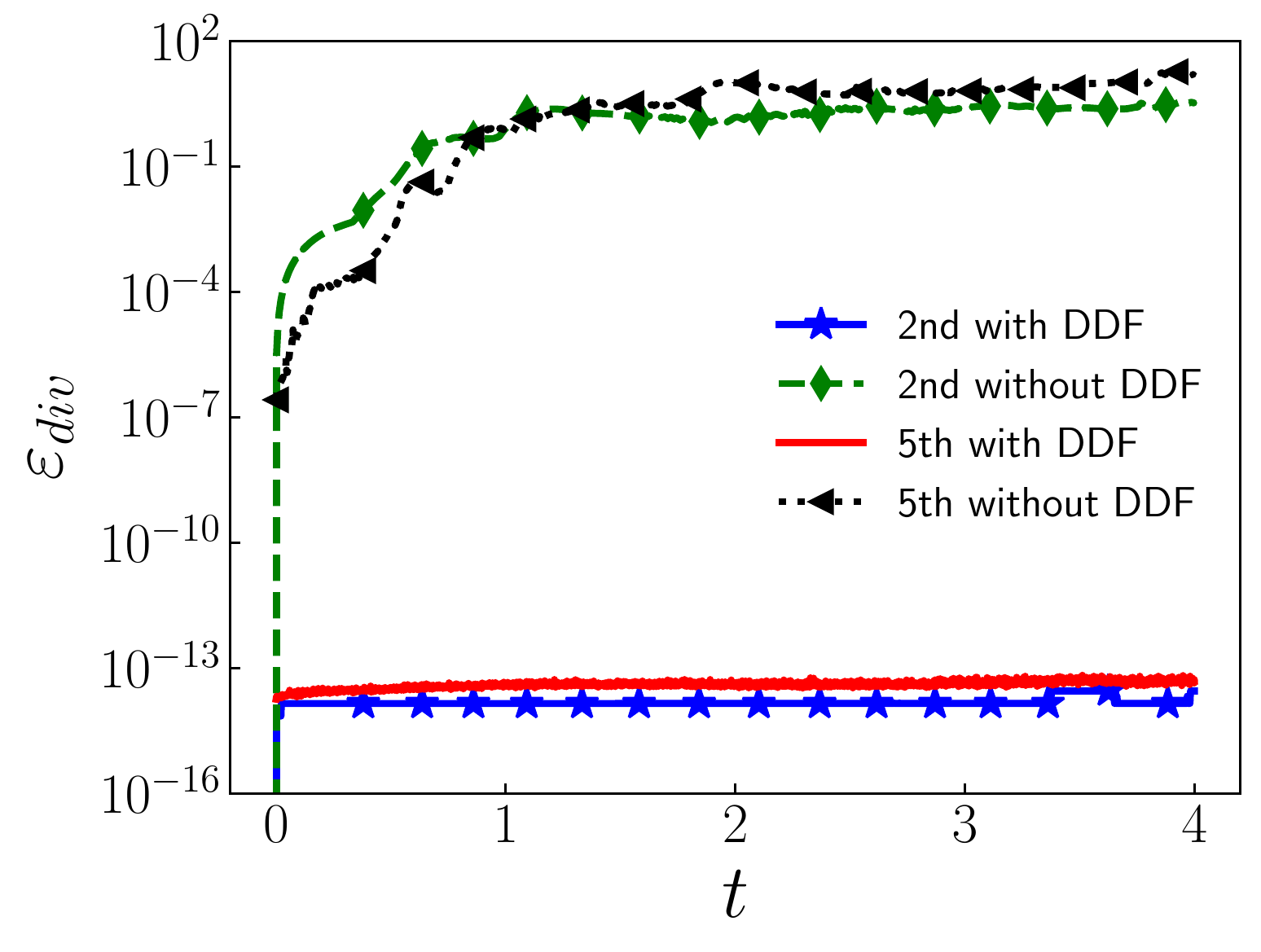}
		\caption{ \Cref{Ex:OT}}
		\label{fig:Ex-OT-gdivB}
	\end{subfigure}
	\hfill
	\begin{subfigure}{0.32\textwidth}
		\includegraphics[width=\textwidth]{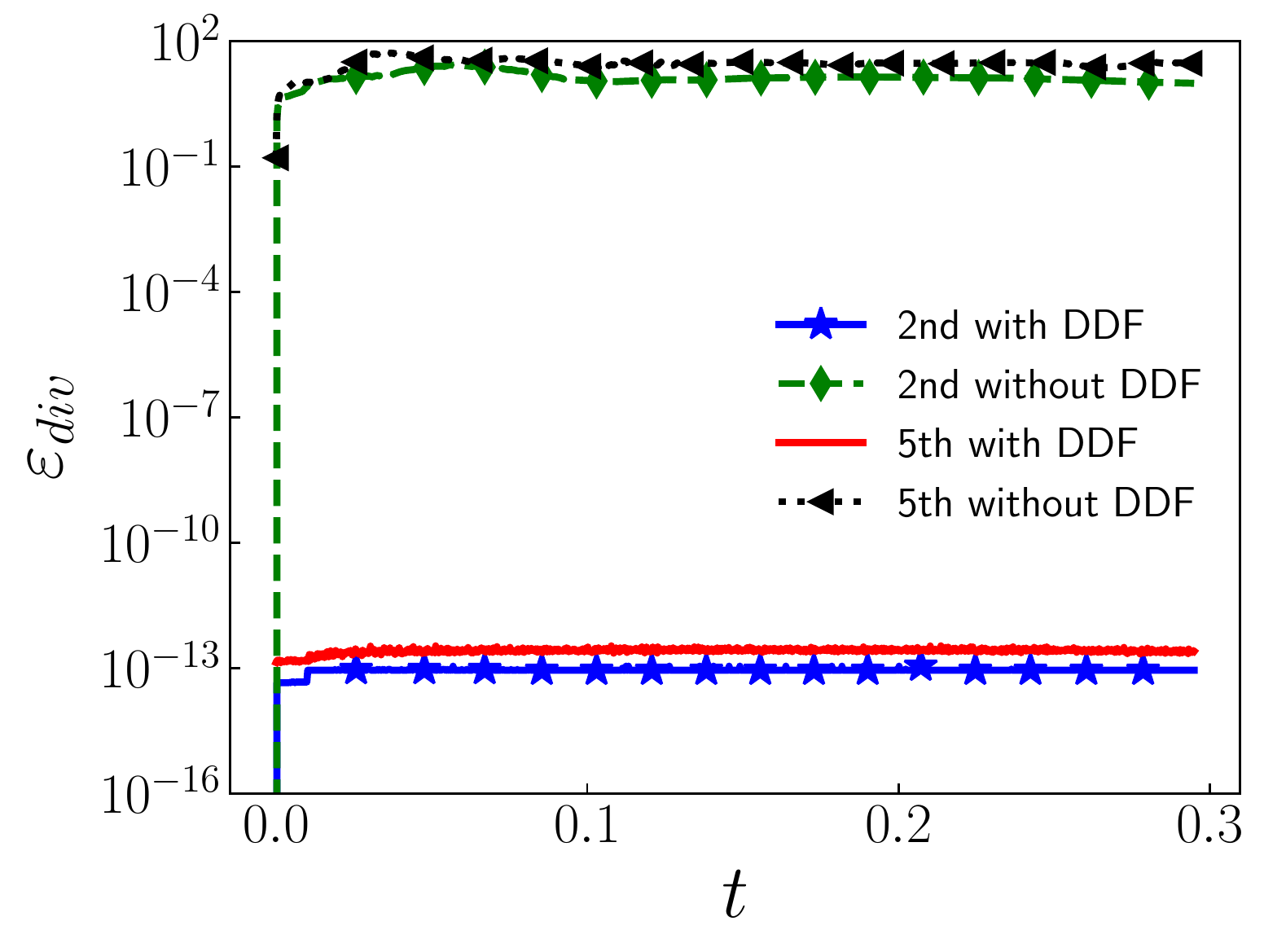}
		\caption{ \Cref{Ex:Rotor}}
		\label{fig:Ex-Rotor-gdivB}
	\end{subfigure}
	\hfill
	\begin{subfigure}{0.32\textwidth}
		\includegraphics[width=\textwidth]{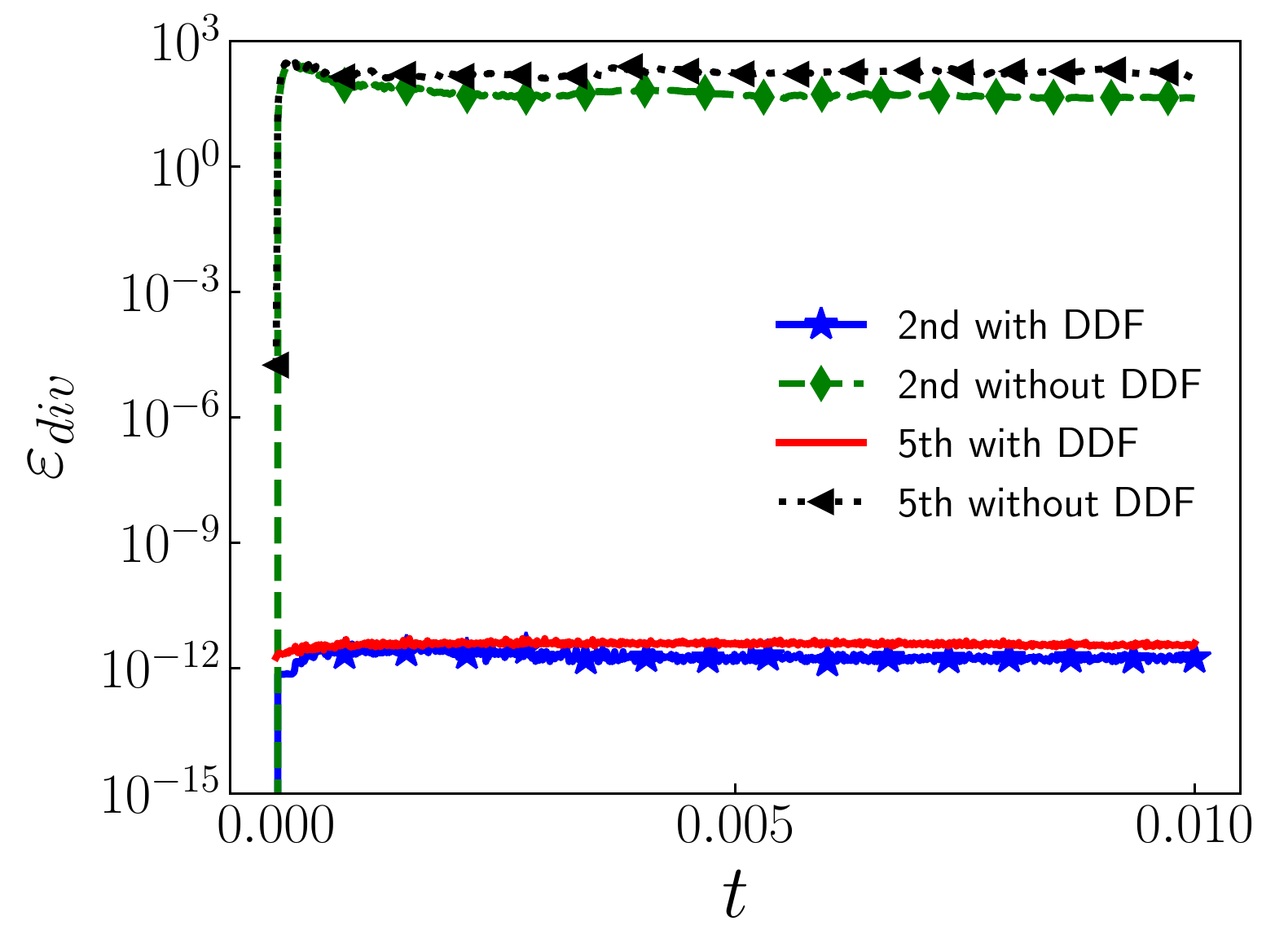}
		\caption{ \Cref{Ex:Blast}}
		\label{fig:Ex-BL-gdivB}
	\end{subfigure}
	
	\caption{The time evolution of the divergence errors $\varepsilon_{div}:= \max_{i,j} \left| \nabla_h \cdot {\bf B}_{ij} \right| $.
	}
	\label{fig:Ex-OT-Rotor-Blast}
\end{figure}

\begin{expl}[Rotor problem]\label{Ex:Rotor}
	In this example, we simulate the benchmark rotor problem, which describes a dense fluid disk rotating in an ambient fluid with $\gamma=5/3$. The initial conditions are set as   
	{\small \begin{equation*}
		(\rho, \mathbf{v}, \mathbf{B}, p) = 
		\begin{cases}
			(10, -(y-0.5)/r_1, (x-0.5)/r_1, 0, 2.5/\sqrt{4\pi}, 0, 0, 0.5), &  r \leq r_1, \\
			(1+9 \phi, -\phi(y-0.5)/r, \phi(x-0.5)/r, 0, 2.5/\sqrt{4\pi}, 0, 0, 0.5), &  r_1< r \leq r_2, \\
			(1,0,0,0, 2.5/\sqrt{4\pi}, 0, 0, 0.5 ), &  r_2<r, \\
		\end{cases}
	\end{equation*}}
	where $r=\sqrt{(x-0.5)^2+(y-0.5)^2}$, $r_1=0.1$, $r_2=0.115$, and $\phi=(r_2-r)/(r_2-r_0)$. The computational domain is taken as $[0,1]^2$ with outflow boundary conditions. The simulation is performed until $t=0.295$ on the uniform mesh of $400\times 400$ cells.  
	 \Cref{fig:Ex-Rotor} displays the numerical results obtained by the proposed second-order and fifth-order DDFPP schemes. 
	 We see that our results agree well with those reported in \cite{WuJiangShu2022}. 
	 The time evolution of the discrete divergence error $\varepsilon_{div}$ is shown \Cref{fig:Ex-Rotor-gdivB}. 
	One can observe that the divergence errors are kept at about $10^{-13}$ if the DDF projection is applied, while the errors grow quickly if the DDF projection is not used.  
	 A further comparison is shown in \Cref{fig:Ex-Ro2} for the fifth-order finite volume WENO schemes with and without DDF projection.  One can see the numerical instability and nonphysical structures from the solution computed without DDF projection, while the waves are correctly resolved when the DDF projection is used. 
	
	\begin{figure}[htbp]
		\centering
		\begin{subfigure}{0.32\textwidth}
			\includegraphics[width=\textwidth]{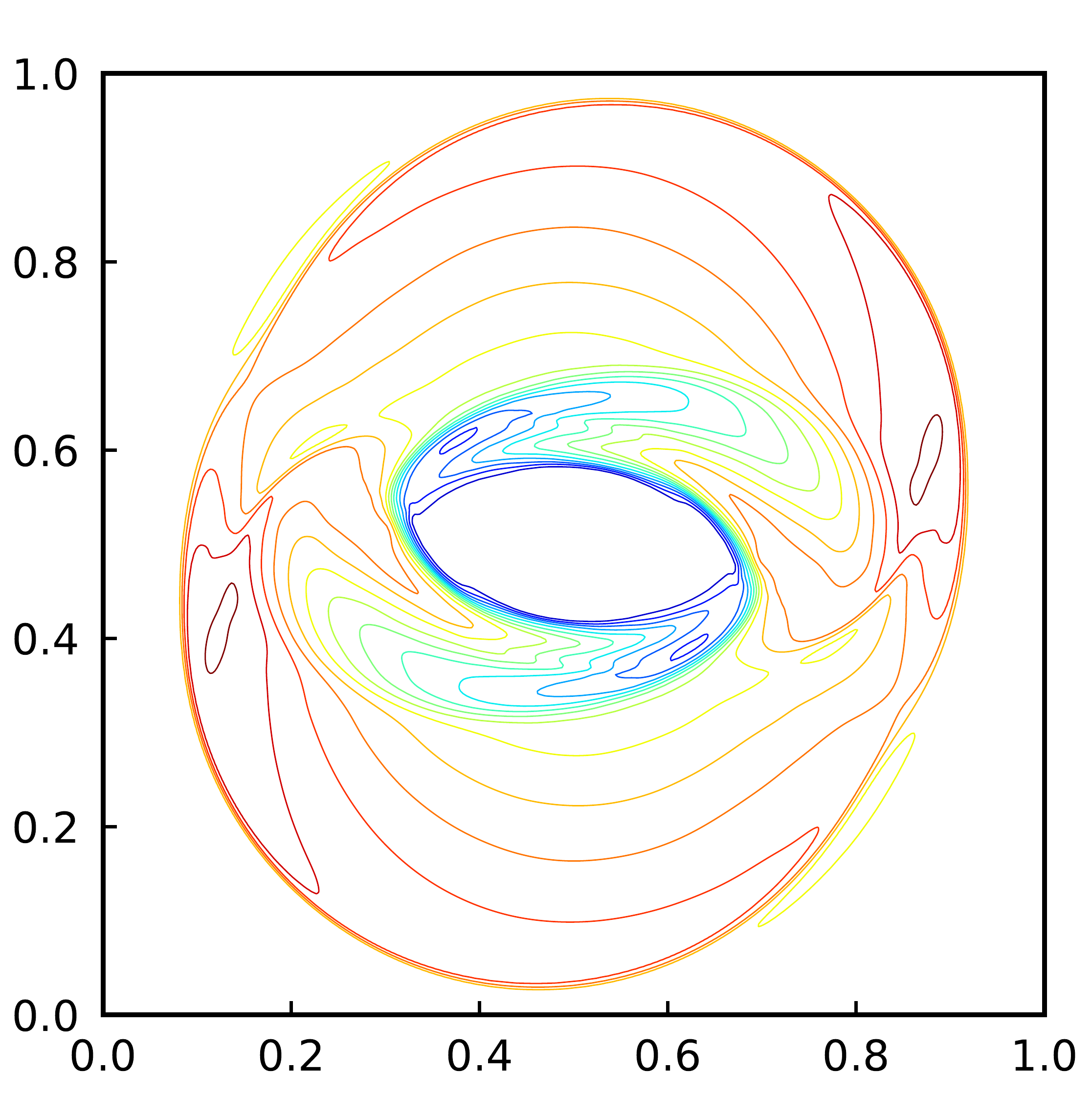}
		\end{subfigure}
		\hfill
		\begin{subfigure}{0.32\textwidth}
			\includegraphics[width=\textwidth]{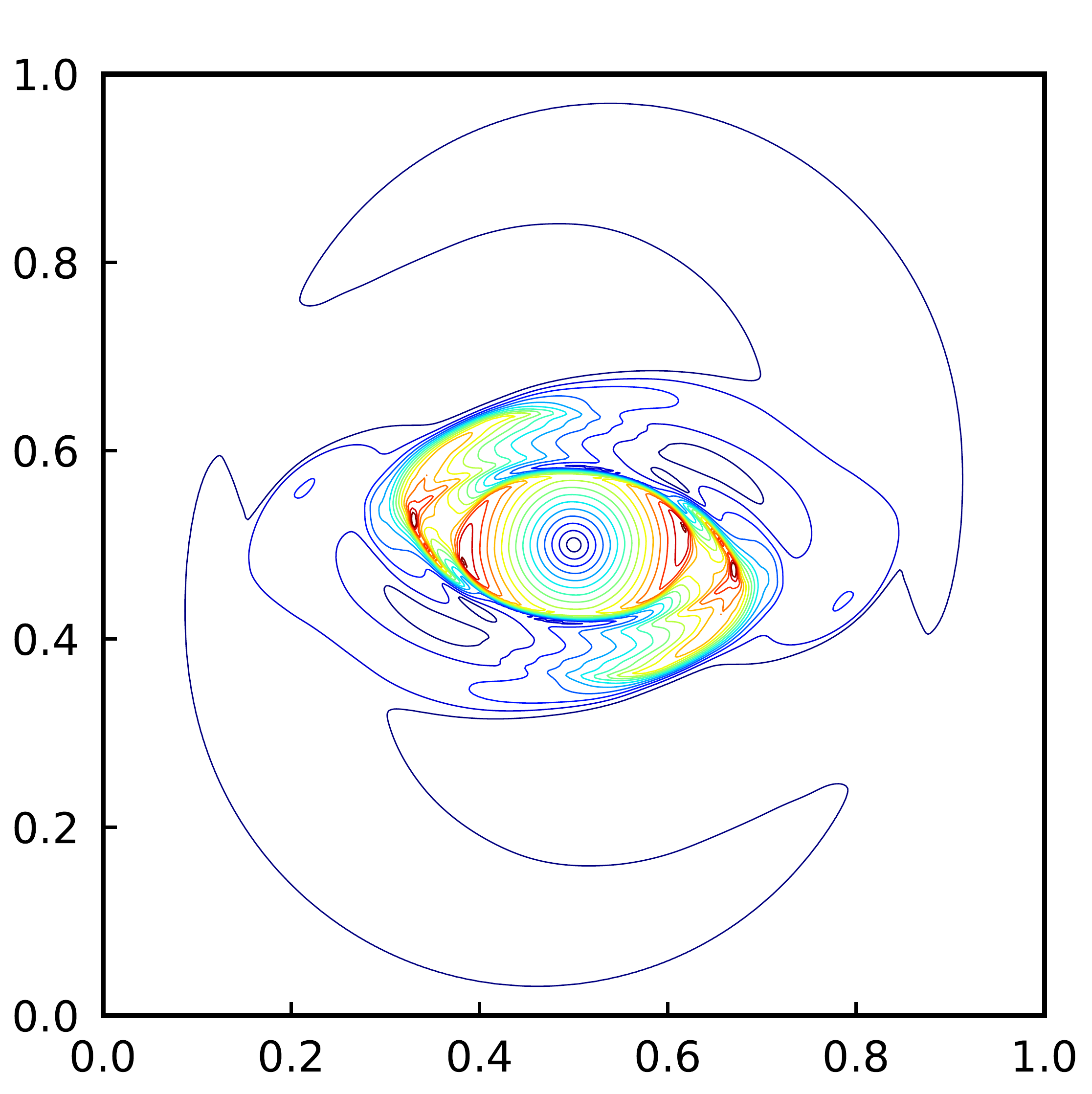}
		\end{subfigure}
		\hfill
		\begin{subfigure}{0.32\textwidth}
			\includegraphics[width=\textwidth]{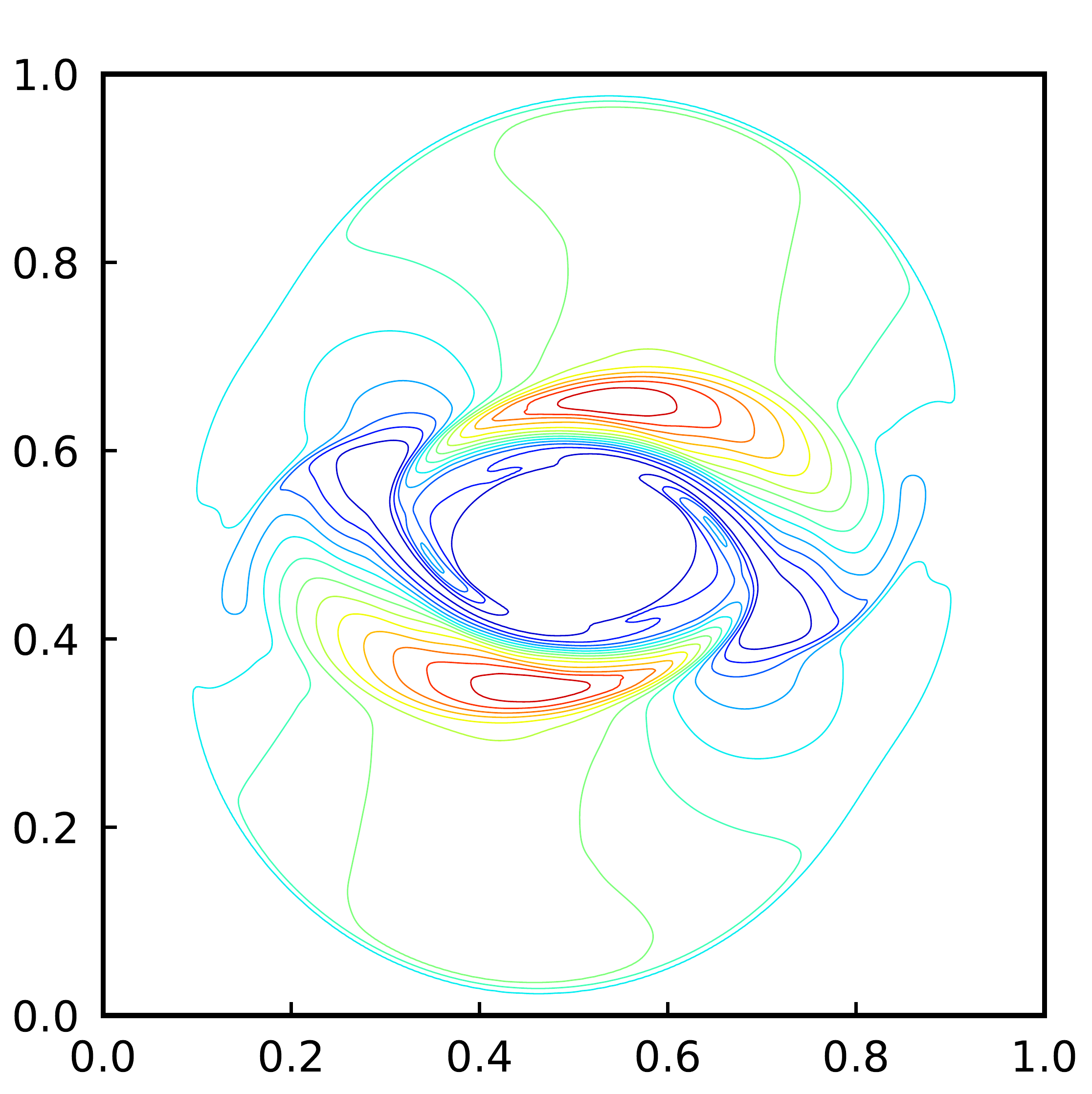}
		\end{subfigure}

		\begin{subfigure}{0.32\textwidth}
			\includegraphics[width=\textwidth]{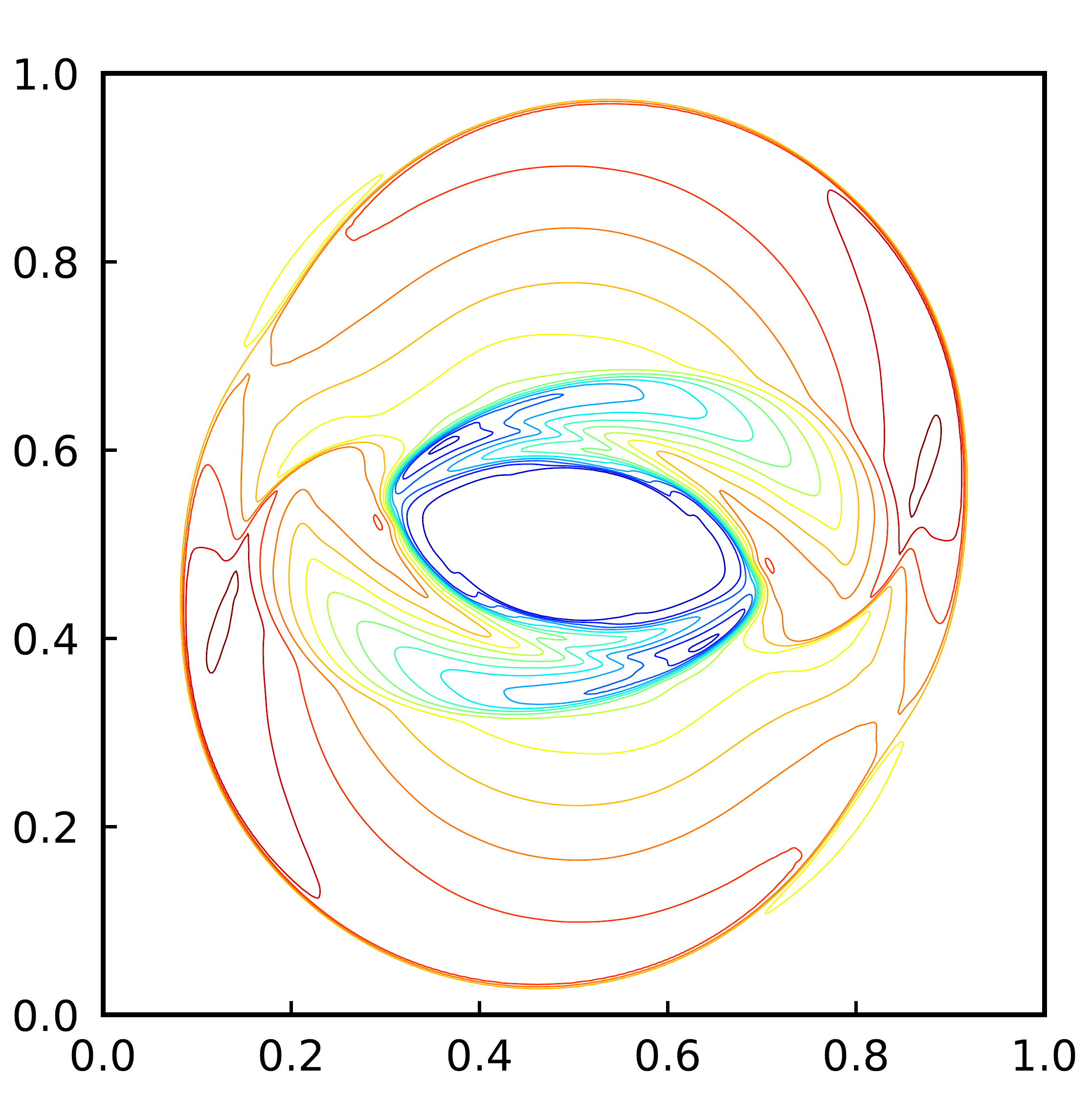}
		\end{subfigure}
		\hfill
		\begin{subfigure}{0.32\textwidth}
			\includegraphics[width=\textwidth]{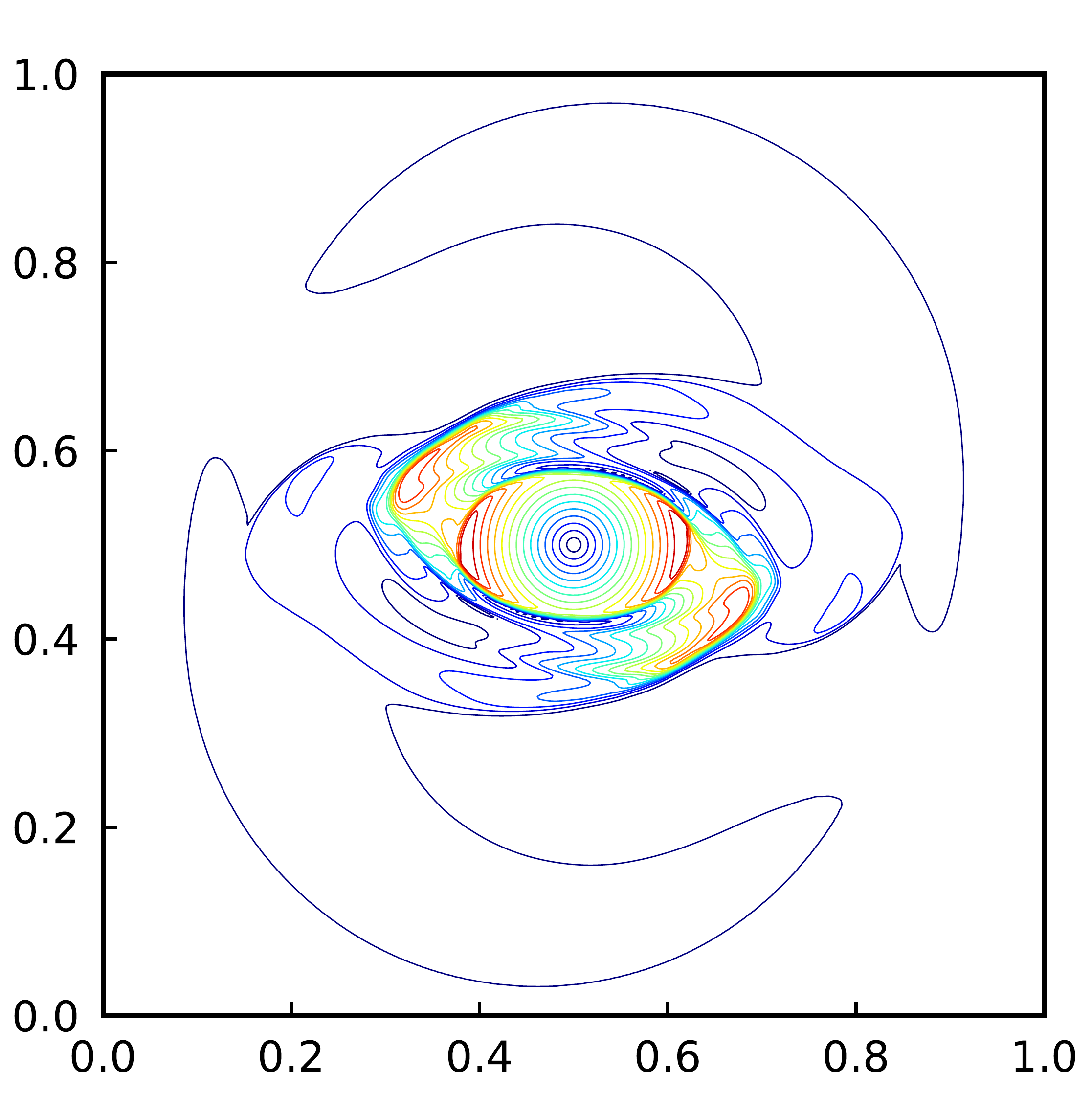}
		\end{subfigure}
		\hfill
		\begin{subfigure}{0.32\textwidth}
			\includegraphics[width=\textwidth]{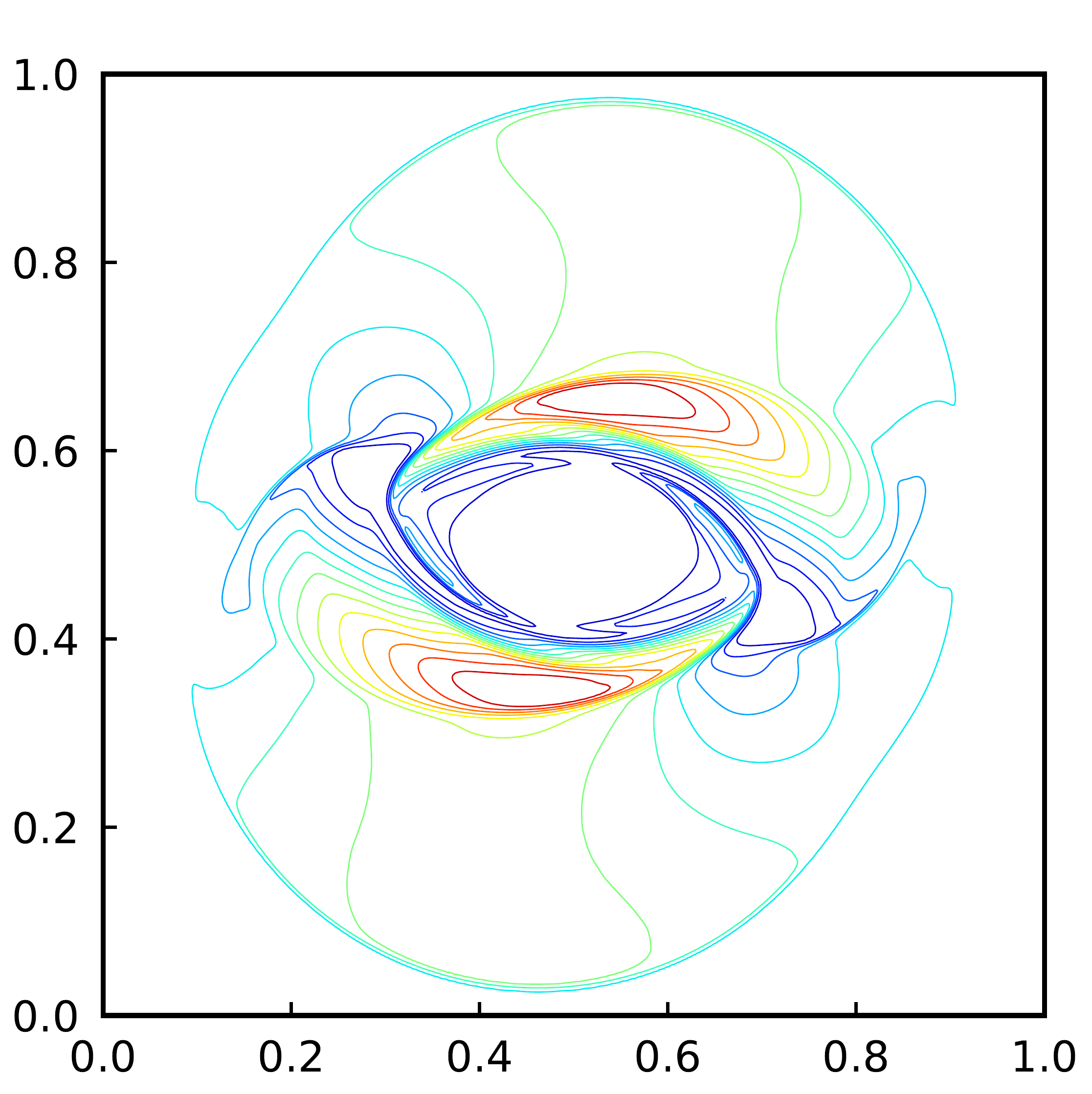}
		\end{subfigure}
		
		\caption{Contours of thermal pressure (left), Mach number (middle), and magnetic pressure (right). 
			Top: second-order DDFPP. Bottom: fifth-order DDFPP.
		}
		\label{fig:Ex-Rotor}
	\end{figure} 
	
		\begin{figure}[htbp]
		\centering
		\begin{subfigure}{0.48\textwidth}
			\centering
			\includegraphics[width=0.8\textwidth]{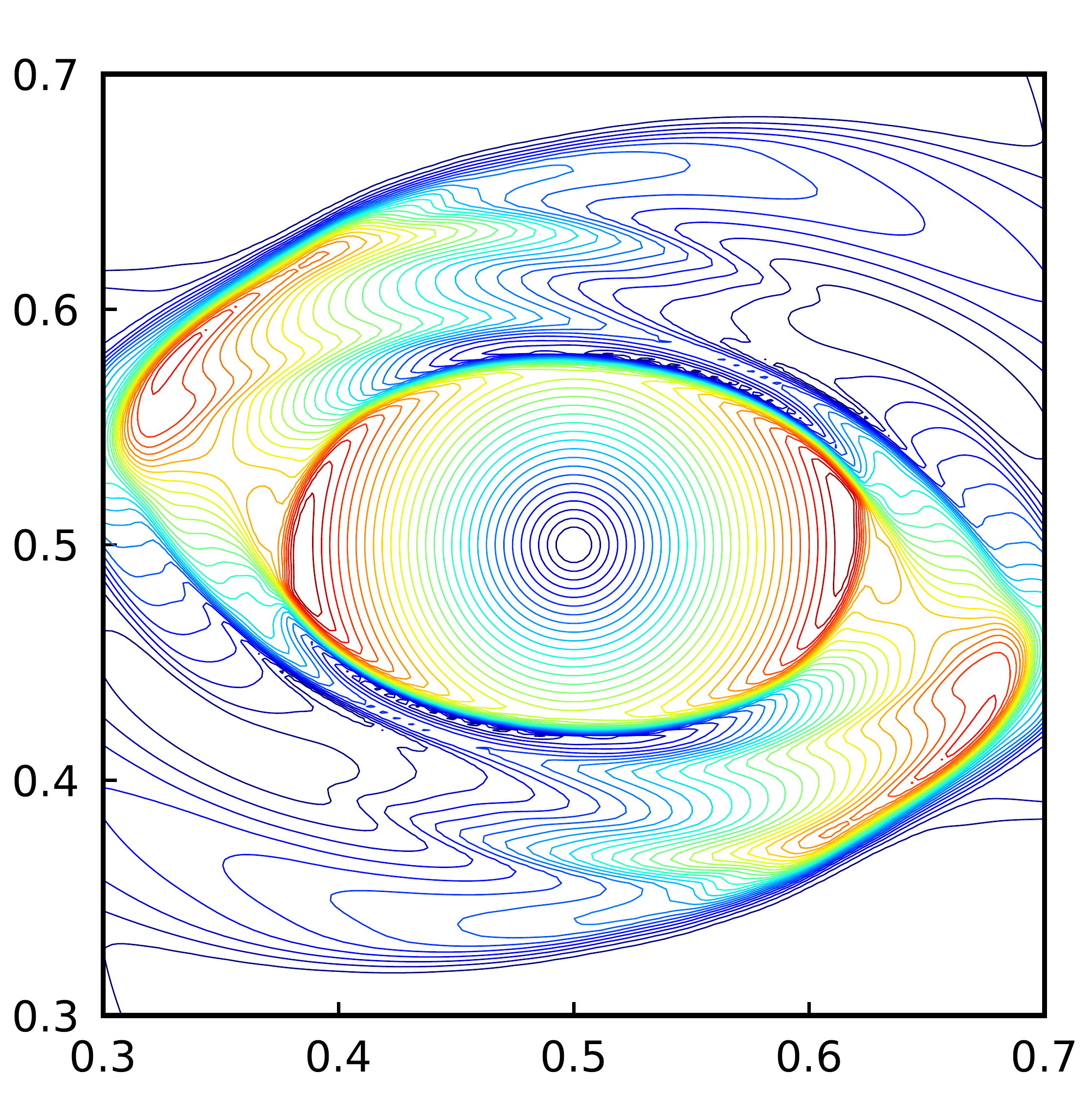}
			\caption{With DDF projection}
		\end{subfigure}
		\begin{subfigure}{0.48\textwidth}
			\centering
			\includegraphics[width=0.8\textwidth]{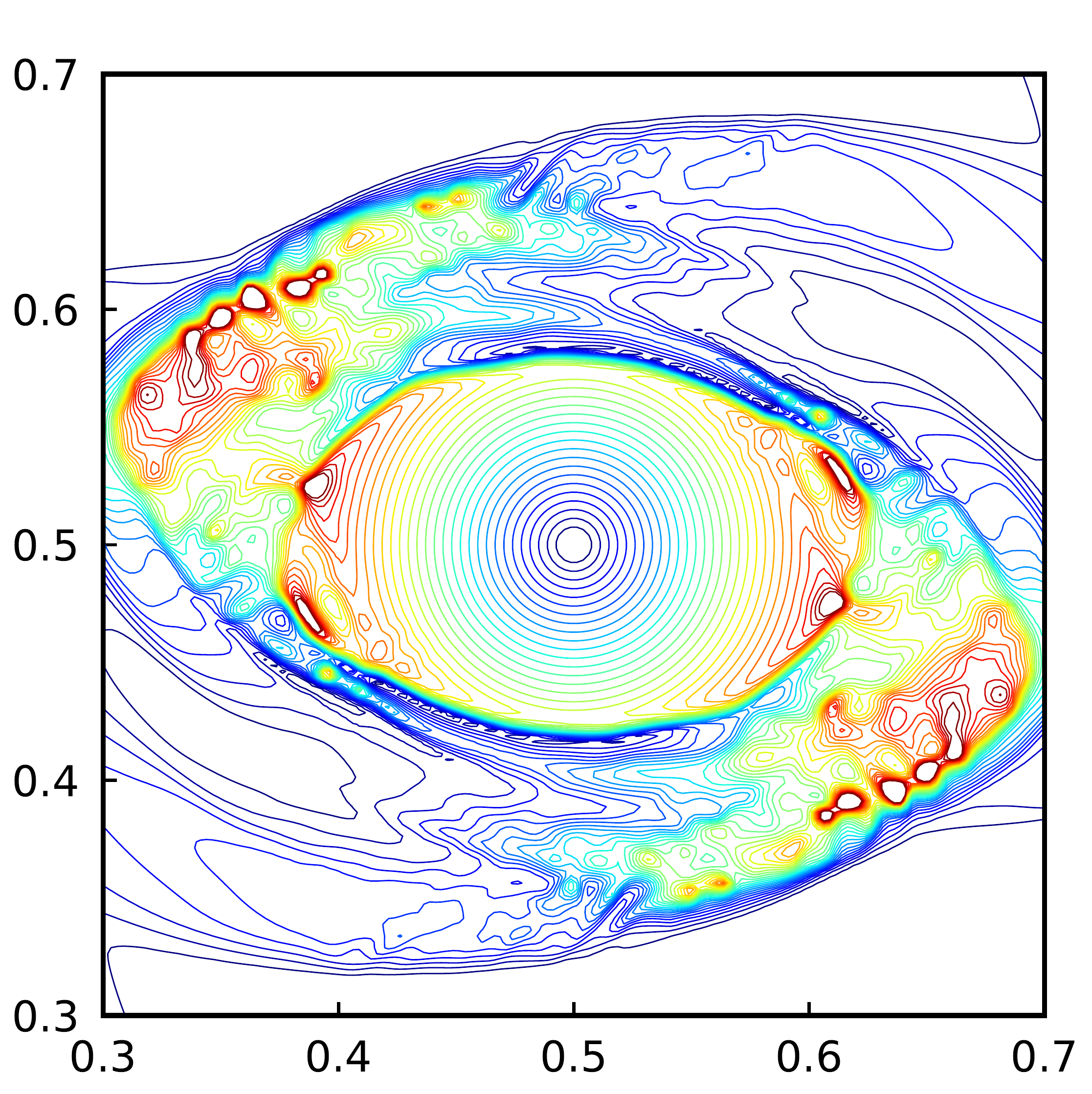}
			\caption{Without DDF projection}
		\end{subfigure}
		
		\caption{ \Cref{Ex:Rotor}: Zoom-in central part of Mach number at $t=0.295$. 
		}
		\label{fig:Ex-Ro2}
	\end{figure}

\end{expl}

\begin{expl}[Blast problem]\label{Ex:Blast}
	We now study a challenging blast problem \cite{BalsaraSpicer1999} with the adiabatic index $\gamma = 1.4$ to demonstrate the PP property of our numerical schemes.  The computational domain $[-0.5,0.5]^2$ is initially full of stationary fluid with piecewise constant pressure $p$, which has a strong circular jump on $x^2+y^2 = 0.1^2$. The initial conditions are given by
	\begin{equation*}
		(\rho, \mathbf{v}, \mathbf{B}) = \left(1, 0, 0, 0, \frac{100}{\sqrt{4\pi}}, 0, 0\right), \qquad
		p = 
		\begin{cases}
			10^3, & \sqrt{x^2+y^2} \leq 0.1, \\
			0.1, & \text{otherwise}.
		\end{cases}
	\end{equation*}
		The computational domain is partitioned into $200 \times 200$ uniform cells. 
		\Cref{fig:Ex-Blast1} displays the numerical results obtained by the proposed second-order 
		and fifth-order DDFPP schemes, in comparison with the results given by the fifth-order WENO scheme without DDF projection. 
		We can see that the flow patterns are well captured by the proposed DDFPP schemes and agree with those computed in 
		\cite{Christlieb,WuShu2018,WuShu2019,WuShu2020NumMath,WuJiangShu2022}. 
		Furthermore, the fifth-order DDFPP scheme exhibits higher resolution than the second-order one, and both DDFPP schemes are very robust. 
		However, without DDF projection, the fifth-order WENO scheme produces obvious oscillations and nonphysical structures 
		in the numerical solution.  
		\Cref{fig:Ex-BL-gdivB} presents the time evolution of the discrete divergence error $\varepsilon_{div}$. 
		As we can see, the DDF condition is enforced to machine accuracy when the DDF projection is employed, 
		but the divergence errors grow very fast if we do not perform the DDF projection.  
		We also notice that if we turn off the PP limiter, the code would break down due to negative pressure.

	\begin{figure}[htbp]
		\centering
		\begin{subfigure}{0.32\textwidth}
			\includegraphics[width=\textwidth]{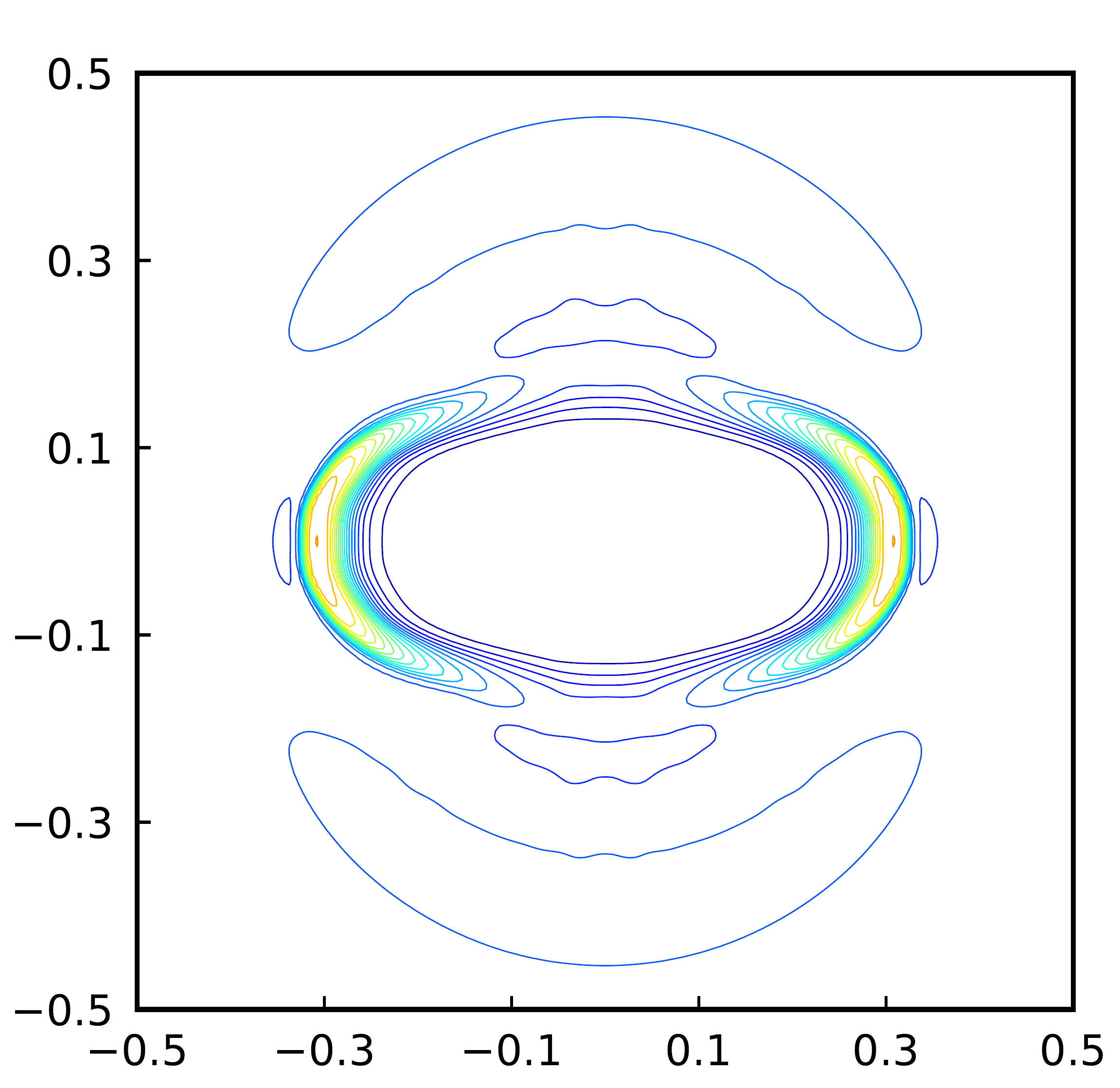}
		\end{subfigure}
		\hfill
		\begin{subfigure}{0.32\textwidth}
			\includegraphics[width=\textwidth]{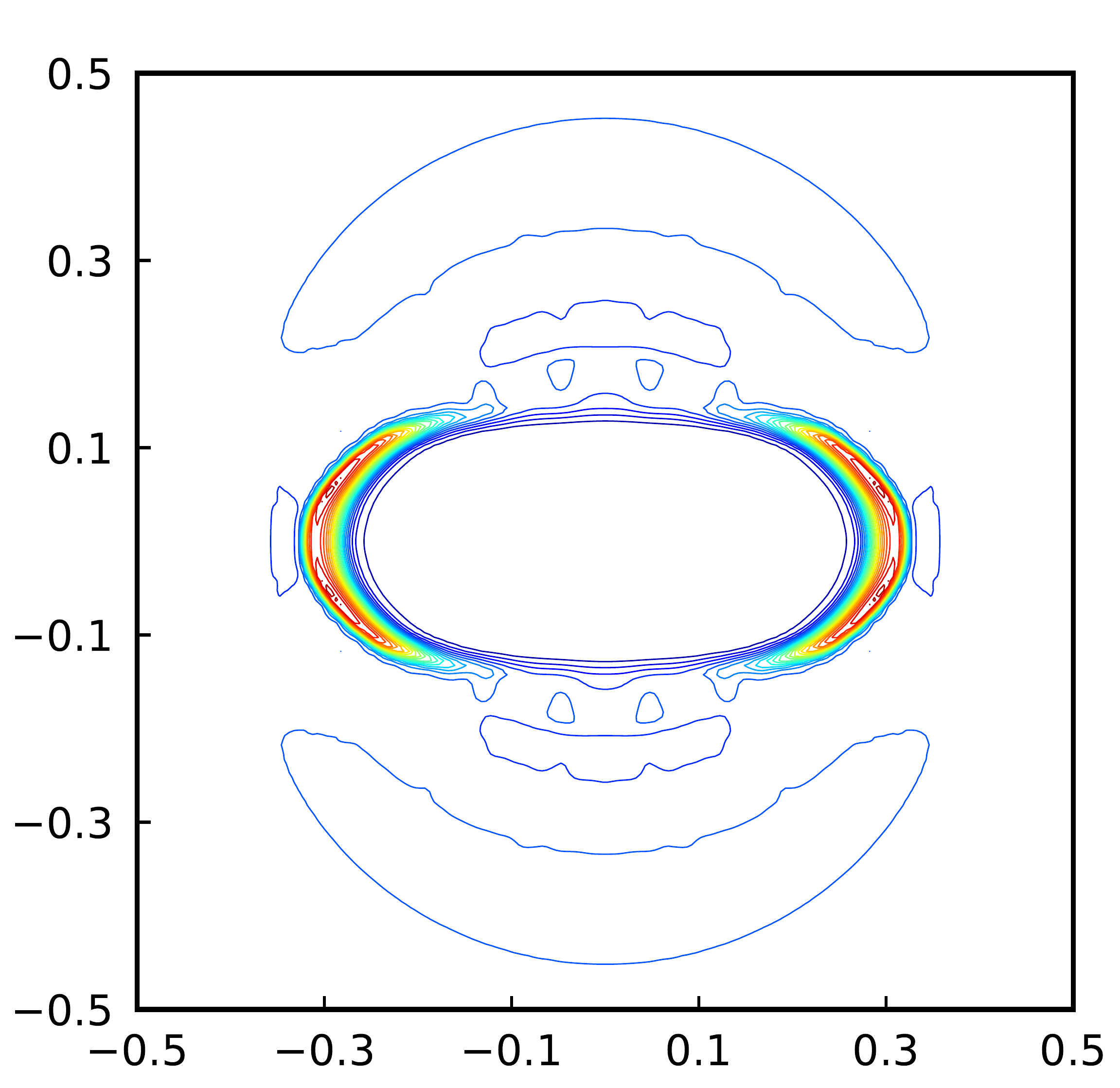}
		\end{subfigure}
		\hfill
	\begin{subfigure}{0.32\textwidth}
		\includegraphics[width=\textwidth]{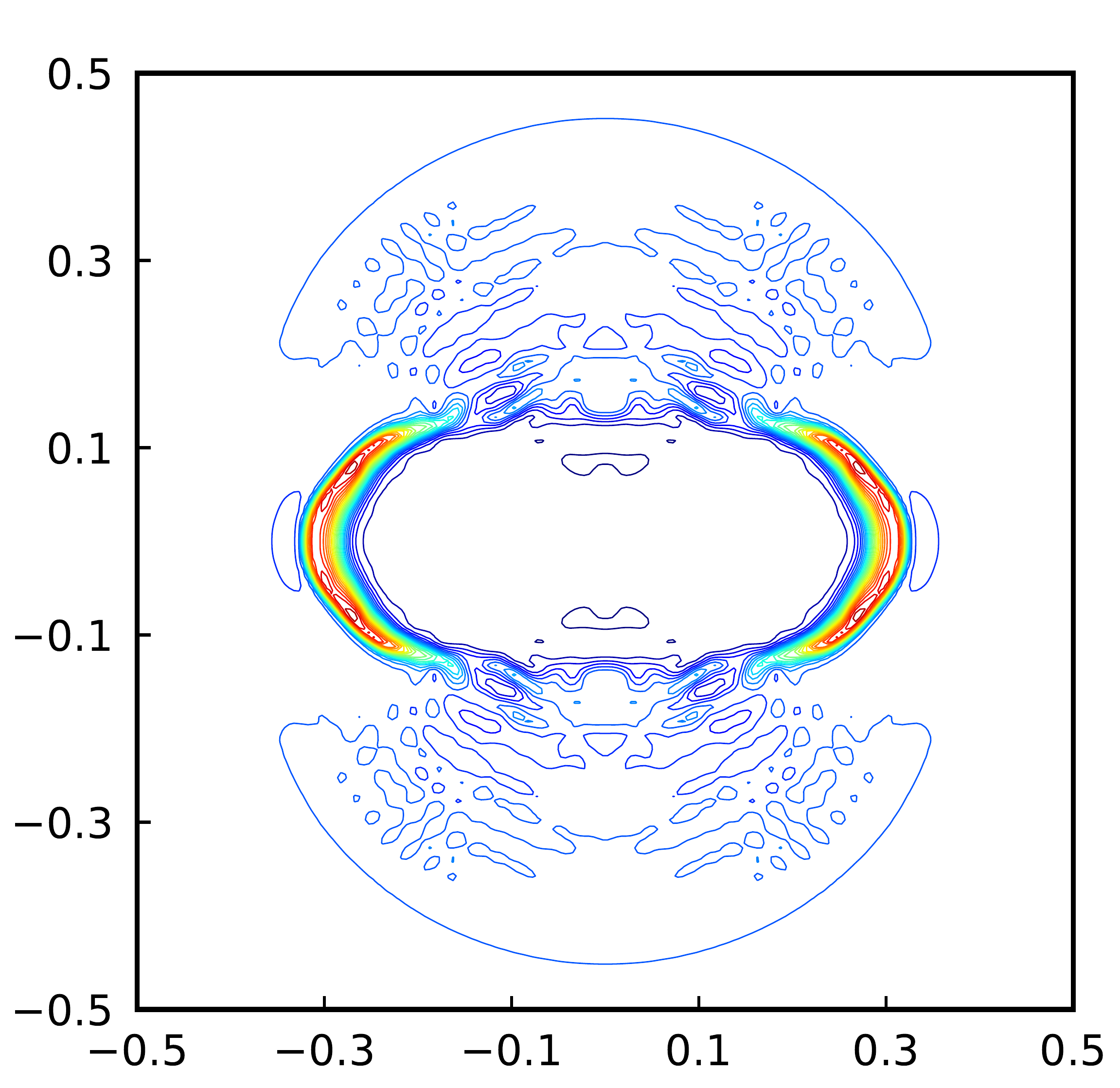}
	\end{subfigure}

		\caption{ \Cref{Ex:Blast}: The contour plots of the density. Twenty-five equally spaced contour lines from 0.2 to 4.54. 
			Left: second-order DDFPP scheme. Middle: fifth-order DDFPP scheme.  Right: fifth-order WENO scheme {\em without DDF projection}.
		}
		\label{fig:Ex-Blast1}
	\end{figure} 
	 
\end{expl}

\begin{expl}[Astrophysical jet]\label{Ex:Jet}
	The last example simulates a high Mach number jet problem in a strong magnetic field  \cite{WuShu2018,WuJiangShu2022}. The adiabatic index is set to be $\gamma=1.4$. Initially, the domain $[-0.5,0.5]\times[0,1.5]$ is full of the ambient plasma with 
	$(\rho, \mathbf{v}, \mathbf{B}, p) = (0.1\gamma, 0, 0, 0, 0, \sqrt{200}, 0, 1).$ 
	A high-speed jet initially locates at $x\in [-0.05,0.05]$ and $y=0$, it is injected in $y$-direction of the bottom boundary with the inflow jet condition
	$(\rho, \mathbf{v}, \mathbf{B}, p) = (\gamma, 0, 800, 0, 0, \sqrt{200}, 0, 1).$  
	In our test, the computational domain is taken as $[0,0.5]\times [0,1.5]$ and divided into $200\times 600$ cells. For the left boundary $x=0$, the reflecting boundary condition is imposed. The outflow conditions are applied on other boundaries. 
	\Cref{fig:Ex-Jet1} shows the schlieren images of density logarithm $\log{\rho}$ obtained by our second-order and fifth-order DDFPP schemes. We can observe that the numerical results are comparable to those simulated in \cite{WuShu2018,WuShu2019,WuJiangShu2022}, while the fifth-order DDFPP scheme has a higher resolution than the second-order one. 
	This is a very challenging test, as  
	negative pressure can be easily produced in the simulation, due to the high Mach number and the presence of magnetic field. 
	If the PP limiter is turned off, we observe that the code would blow up quickly. 
	The proposed DDFPP schemes exhibit good robustness and high resolution for this problem. 
	
%
	
	\begin{figure}[htbp]
		\centering
		\begin{subfigure}{0.32\textwidth}
			\includegraphics[width=\textwidth]{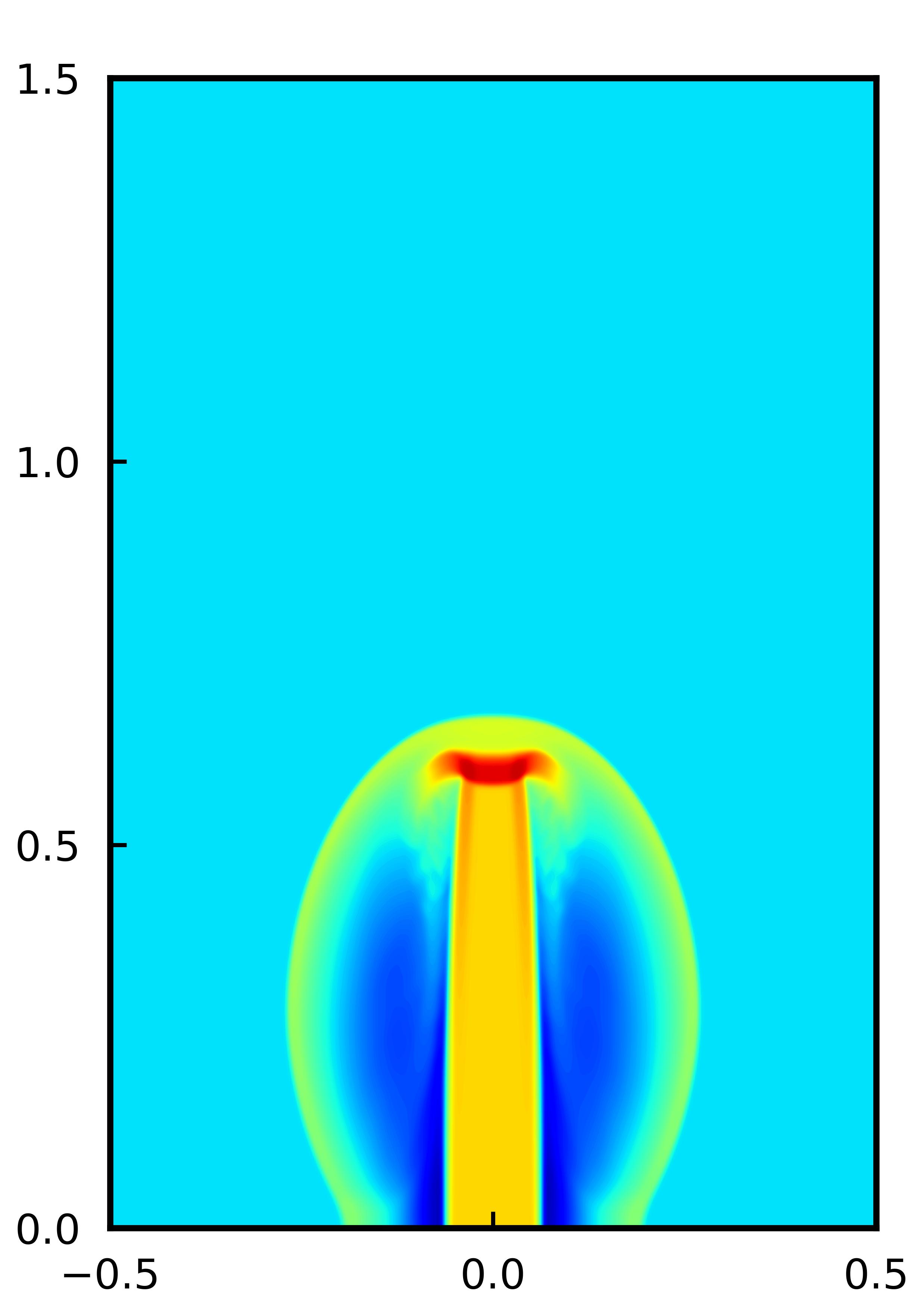}
		\end{subfigure}
		\hfill
		\begin{subfigure}{0.32\textwidth}
			\includegraphics[width=\textwidth]{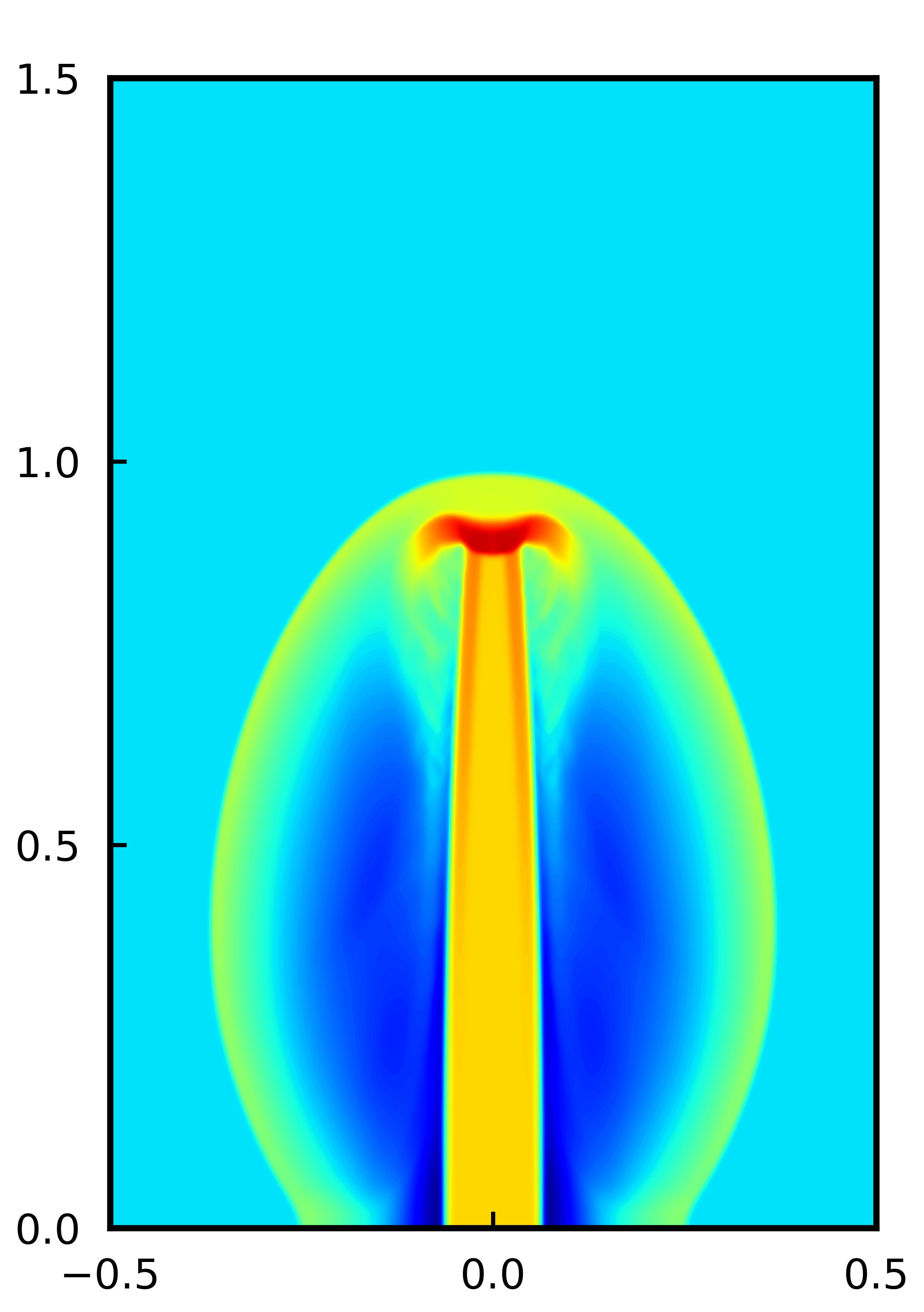}
		\end{subfigure}
		\hfill
		\begin{subfigure}{0.32\textwidth}
			\includegraphics[width=\textwidth]{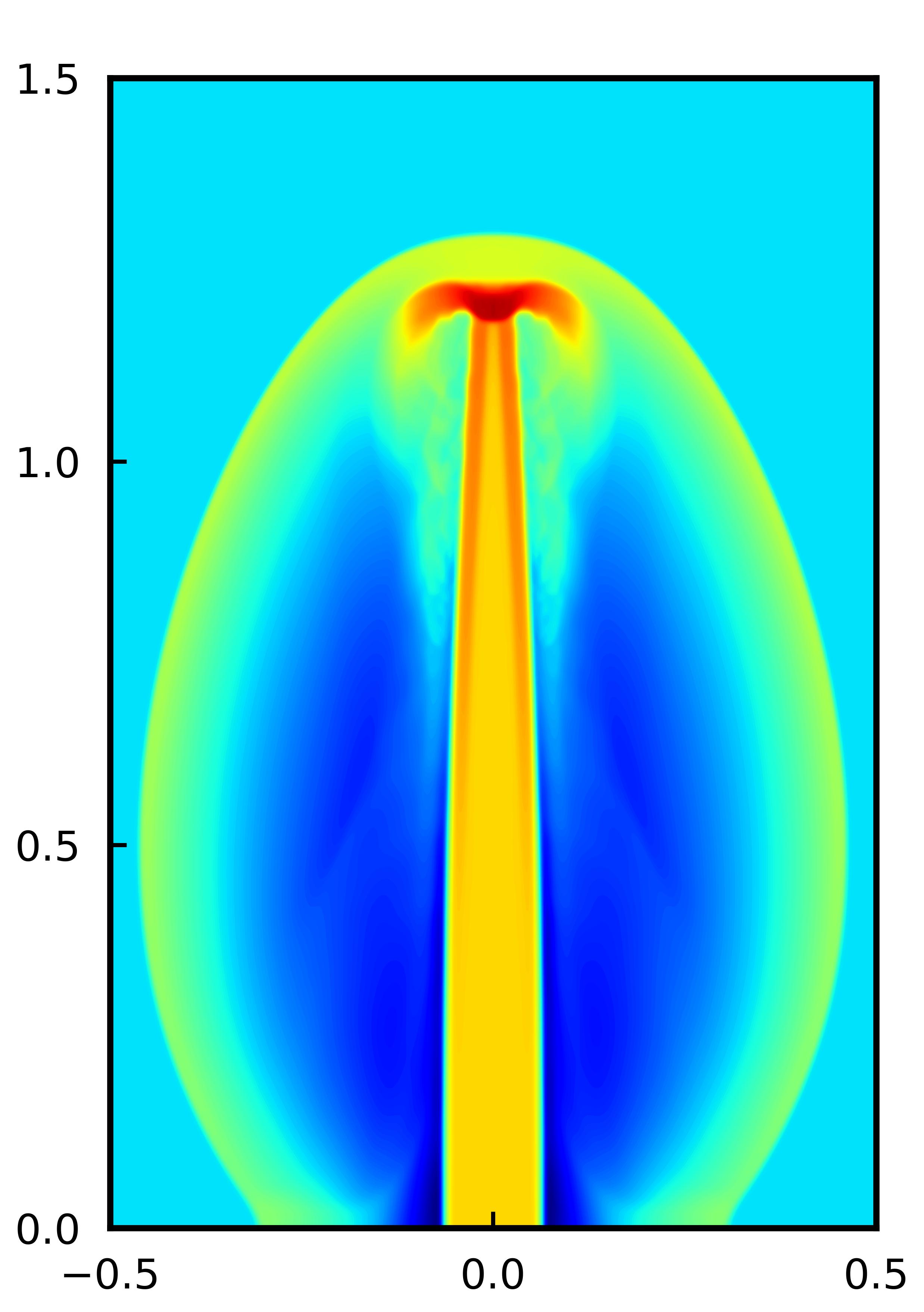}
		\end{subfigure}

		\begin{subfigure}{0.32\textwidth}
			\includegraphics[width=\textwidth]{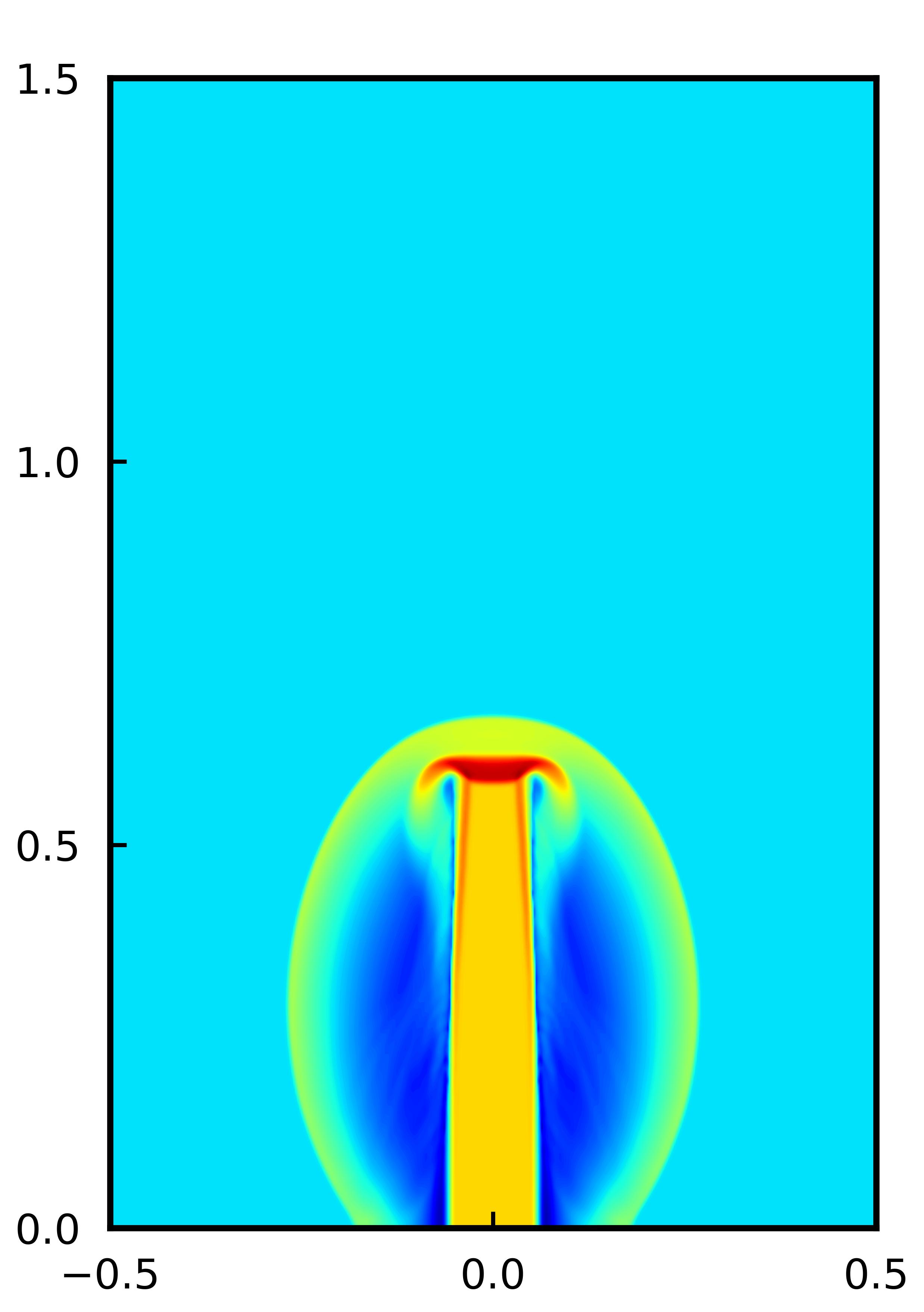}
		\end{subfigure}
		\hfill
		\begin{subfigure}{0.32\textwidth}
			\includegraphics[width=\textwidth]{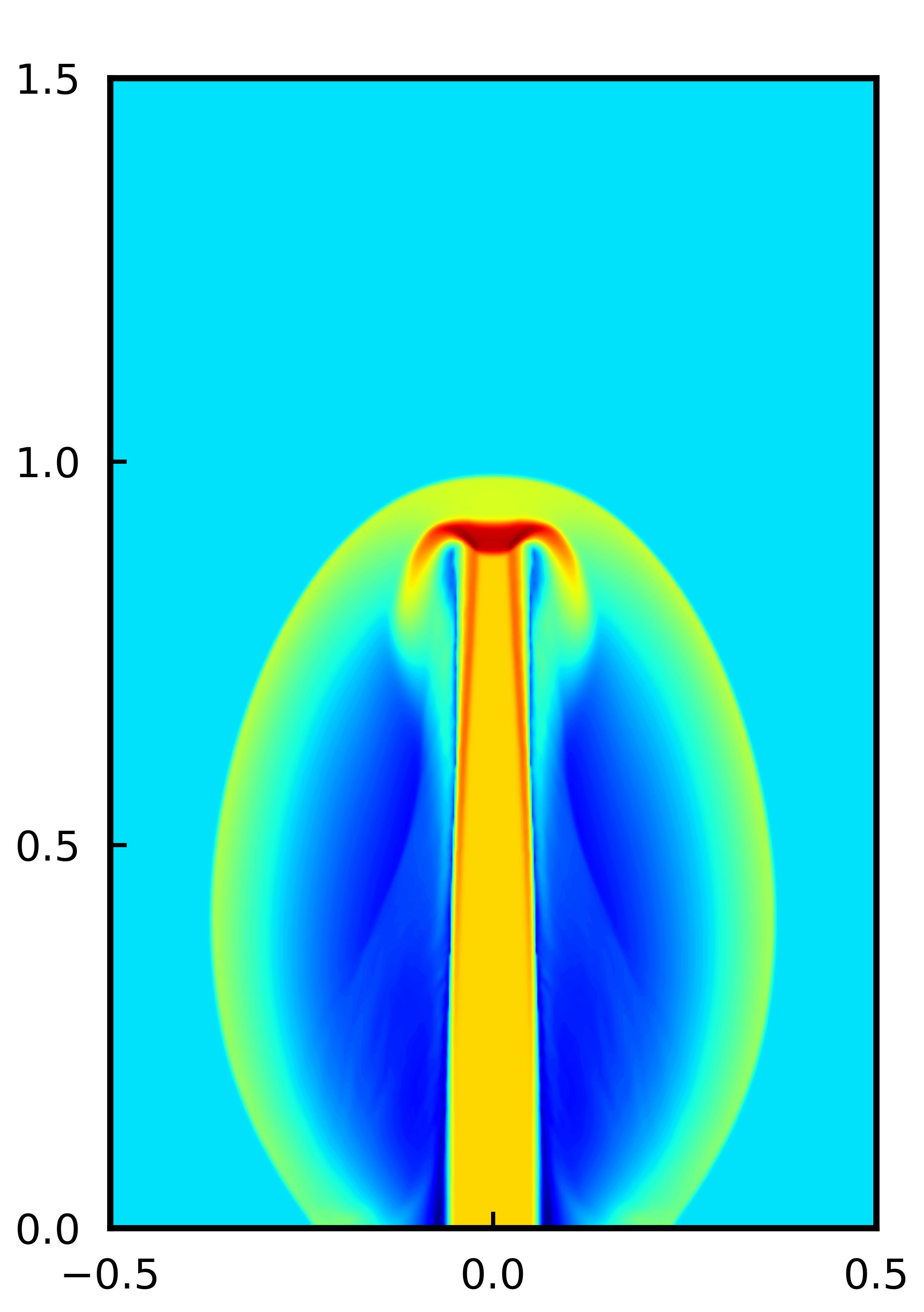}
		\end{subfigure}
		\hfill
		\begin{subfigure}{0.32\textwidth}
			\includegraphics[width=\textwidth]{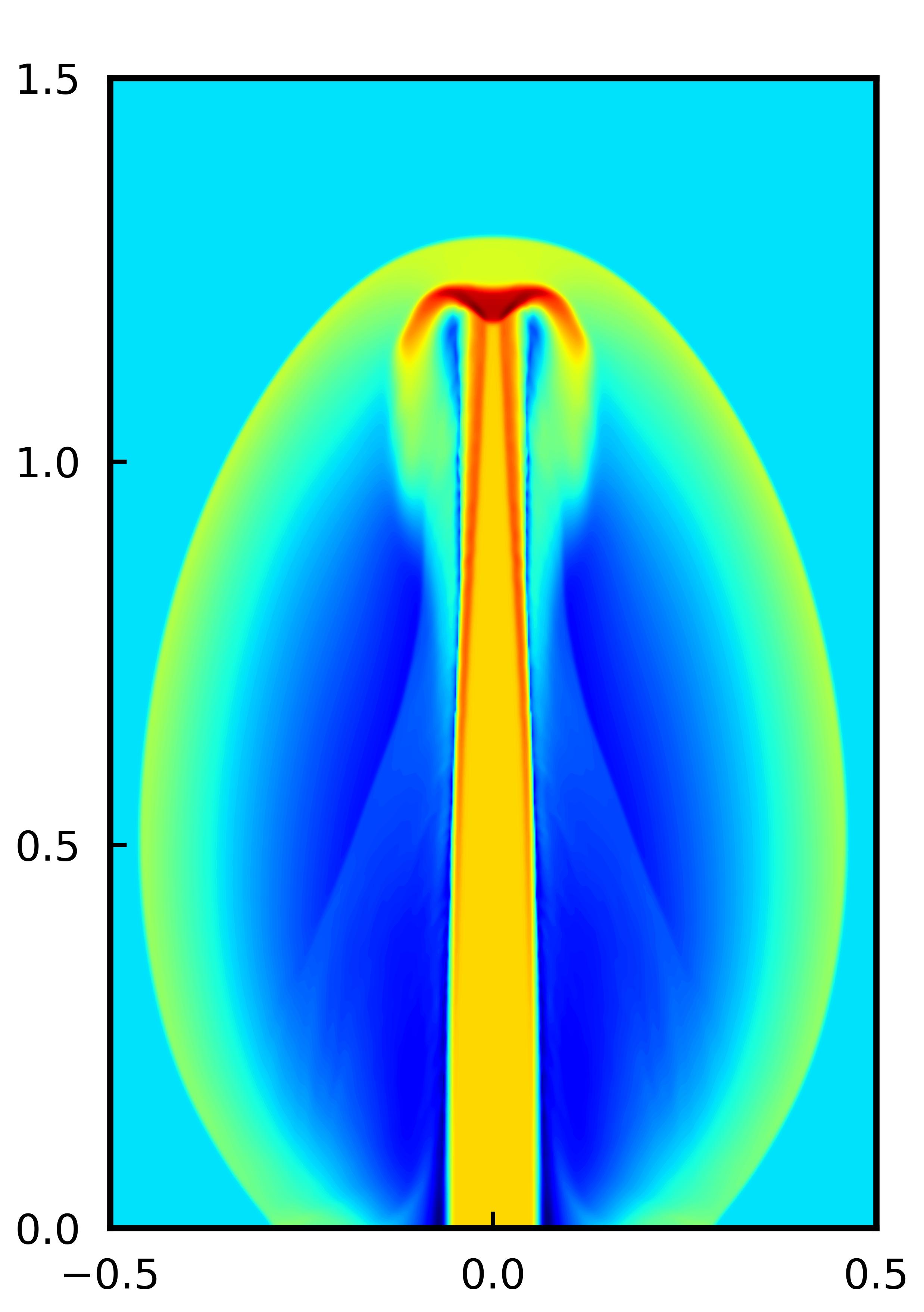}
		\end{subfigure}
		
		\caption{ \Cref{Ex:Jet}: Density logarithm at $t=0.001, 0.0015$, and $0.002$ (from left to right). 
			Top: second-order DDFPP scheme. Bottom: fifth-order DDFPP scheme.
		}
		\label{fig:Ex-Jet1}
	\end{figure} 
	
\end{expl}

\section{Conclusions}\label{Conclusions}

In this paper, we have developed a new high-order robust finite volume method for solving the ideal MHD equations, which govern the behavior of magnetized fluids in astrophysical and laboratory plasmas. The method stands out due to its ability to maintain both a DDF constraint and the PP property simultaneously. We have proposed a novel discrete projection technique that enforces the DDF condition by projecting the reconstructed point values at the cell interface into a DDF space, without using any approximation polynomials. This projection method is highly efficient, easy to implement, and particularly suitable for standard high-order finite volume WENO methods, which typically return only the point values in the reconstruction rather than approximation polynomials. 
We have also developed a new finite volume framework for constructing provably PP schemes for the ideal MHD system. 
The framework incorporates the  discrete projection, 
a suitable approximation to the Godunov-Powell source terms, and a simple PP limiter. 
We have provided rigorous analysis of the PP property of the proposed finite volume method, showing that the DDF condition and suitable approximation to the source terms eliminate the effect of magnetic divergence terms on the PP property. The analysis is challenging due to the nonlinearity of  $\mathcal{E}(\mathbf{U})$ and the relationship between the DDF and PP properties.  
We have overcome the challenges by using the GQL approach, which transfers the nonlinear constraint into a family of linear ones.
Finally, we have demonstrated the effectiveness of the proposed method through several benchmark and demanding numerical experiments. The numerical results have shown that the proposed method is robust and accurate, and confirmed the importance of the proposed DDF projection and PP techniques.

\renewcommand\baselinestretch{0.85}

	\bibliographystyle{siamplain}
	\bibliography{references_article}

\end{document}